\documentclass[11pt, twoside, reqno]{amsart} 	

\usepackage[colorinlistoftodos]{todonotes}
\usepackage[T1]{fontenc}
\usepackage[utf8]{inputenc}

\usepackage[backend=biber,style=alphabetic,sorting=nyt]{biblatex}

\usepackage{geometry}                		
\usepackage{graphicx}				
\usepackage{eucal, amscd, changepage, cancel}
\usepackage[normalem]{ulem}
\usepackage{faktor}
\usepackage{hyperref}
\hypersetup{
    colorlinks=true,
    linkcolor=blue,
    citecolor=forestgreen,
    filecolor=blue,      
    urlcolor=black,
}

\usepackage{icomma}
\usepackage{bbold}
\usepackage{comment}
\usepackage{centernot}
\usepackage{amssymb}
\usepackage{amsfonts}
\usepackage{amsthm}
\usepackage{amsmath}
\usepackage{bbm}
\usepackage{xcolor}
\usepackage{mathtools}
\usepackage{tikz}

\usepackage{fancyhdr}
\usepackage{lipsum} 

\usepackage{etoolbox}

\definecolor{forestgreen}{rgb}{0.13, 0.55, 0.13}
\definecolor{greenao}{rgb}{0.0, 0.5, 0.0}
\definecolor{dimgray}{rgb}{0.41, 0.41, 0.41}
\definecolor{afb}{rgb}{0.36, 0.54, 0.66}
\definecolor{bncs}{rgb}{0.0, 0.53, 0.74}
\definecolor{purple}{rgb}{0.47, 0.32, 0.66}
\definecolor{tangerine}{rgb}{0.95, 0.52, 0.0}
\definecolor{darklavender}{rgb}{0.45, 0.31, 0.59}
\definecolor{darkorchid}{rgb}{0.6, 0.2, 0.8}
\definecolor{do}{rgb}{0.6, 0.2, 0.8}

\definecolor{bdf}{rgb}{0.19,0.55, 0.91}

\DeclareMathOperator{\supp}{supp}

\DeclareMathOperator{\Op}{Op}

\DeclareMathOperator{\BCH}{BCH}

\newcommand{\beq}{\begin{equation}}
\newcommand{\eeq}{\end{equation}}

\newcommand{\beqq}{\begin{equation*}}
\newcommand{\eeqq}{\end{equation*}}

\newcommand{\mB}{\mathcal{B}}

\newcommand{\mF}{\mathcal{F}}

\newcommand{\mP}{\mathcal{P}}
\newcommand{\mS}{\mathcal{S}}

\newcommand{\R}{\mathbb{R}}
\newcommand{\C}{\mathbb{C}}
\newcommand{\N}{\mathbb{N}}
\newcommand{\Z}{\mathbb{Z}}

\newcommand{\fg}{\mathfrak{g}}
\newcommand{\fI}{\mathfrak{I}}

\newcommand{\vk}{\vec{k}}
\newcommand{\vR}{\vec{R}}

\newcommand{\la}{\langle}
\newcommand{\ra}{\rangle}

\newcommand{\ap}{\alpha}
\newcommand{\vp}{\varphi}
\newcommand{\de}{\delta}

\newcommand{\g}{\gamma}

\newcommand{\ep}{\epsilon}

\newcommand{\wh}{\widehat}

\newcommand{\wt}{\widetilde}

\newcommand{\mub}{\overline{\mu}}

\newcommand{\p}{\partial}

\newcommand{\lp}{\left(}
\newcommand{\rp}{\right)}

\newcommand{\lf}{\left|}
\newcommand{\rf}{\right|}

\newcommand{\norm}[2]{\left\| #1\right\|_{#2}}

\newcommand{\brac}[2]{\left[ \  #1 \  \right]_{#2}}

\numberwithin{equation}{section}
\allowdisplaybreaks

\newtheorem{theorem}{Theorem}[section]

\newtheorem{proposition}[theorem]{Proposition}

\newtheorem{lemma}[theorem]{Lemma}

\theoremstyle{definition}
\newtheorem{definition}[theorem]{Definition}
\newtheorem{claim}[theorem]{Claim}

\theoremstyle{remark}
\newtheorem{remark}[theorem]{Remark}

\newtheorem{notation}{Notation}

\makeatletter
\def\@tocline#1#2#3#4#5#6#7{\relax
  \ifnum #1>\c@tocdepth 
  \else
    \par \addpenalty\@secpenalty\addvspace{#2}%
    \begingroup \hyphenpenalty\@M
    \@ifempty{#4}{%
      \@tempdima\csname r@tocindent\number#1\endcsname\relax
    }{%
      \@tempdima#4\relax
    }%
    \parindent\z@ \leftskip#3\relax \advance\leftskip\@tempdima\relax
    \rightskip\@pnumwidth plus4em \parfillskip-\@pnumwidth
    #5\leavevmode\hskip-\@tempdima
      \ifcase #1
       \or\or \hskip 1em \or \hskip 2em \else \hskip 3em \fi%
      #6\nobreak\relax
    \hfill\hbox to\@pnumwidth{\@tocpagenum{#7}}\par
    \nobreak
    \endgroup
  \fi}
\makeatother

\newcommand{\Addresses}{{
\noindent  \textsc{Amelia Stokolosa,\\
Department of Mathematics, University of Wisconsin-Madison,}\par\nopagebreak
\noindent \textit{E-mail address}: \texttt{stokolosa@wisc.edu}\\
MSC class - 42B20 (Primary); 42B37 and 43A85 (Secondary).
}}

\title{Tame algebra estimates for product and flag kernels on graded Lie groups}

\author{Amelia Stokolosa}

\date{\today}

\begin{document}

\begin{abstract}
    
We prove that product kernels and flag kernels on a direct product of graded Lie groups $G_1 \times \cdots \times G_{\nu}$ satisfy so-called \emph{tame algebra estimates}. Tame algebra estimates are central to the study of nonlinear partial differential equations via, for instance, the Nash-Moser inverse function theorem. In addition, the special structure of these estimates generates a new Banach-algebraic proof of an inversion theorem for product kernels and flag kernels.

\end{abstract}

\maketitle

\tableofcontents

\section{Introduction}

\subsection{Main results} \

We prove that two classes of multi-parameter singular kernels, namely product kernels and flag kernels defined on a direct product of graded Lie groups $G_1 \times \cdots \times G_{\nu}$, satisfy so-called \textit{tame algebra estimates}. Nagel, Ricci, Stein, and Wainger proved that flag kernels on a homogeneous Lie group $G$  form an \emph{algebra} in \cite{NRSW12}. To introduce multi-parameter dilations, we instead consider product kernels and flag kernels in the setting of a direct product of graded Lie groups $ G_1 \times \cdots \times G_{\nu}$. Our new result establishes that product kernels and flag kernels on $G_1 \times \cdots \times G_{\nu}$ form two algebras exhibiting special additional \textit{tame algebra estimates}.

Product singular integral operators first appeared in the work of R. Fefferman and Stein in \cite{FS82}, and Journ\'e in \cite{Jou85}; while flag kernels were first introduced by Müller, Ricci, and Stein in their work on spectral multipliers on Heisenberg-type groups \cite{MRS95}. The study of product-type singular integral operators and flag kernels gained interest thereafter. See the related work by Nagel, Ricci, Stein, and Wainger on product kernels and flag kernels in \cite{NRS01}, \cite{NRSW12} and \cite{NRSW18}, and by G{\l}owacki in \cite{Glo10}, \cite{Glo10_correction}, \cite{Glo13}. Many authors have pursued the study of flags, flag kernels, flag singular integral operators along with the associated Hardy spaces and weighted norm inequalities. See for instance, \cite{Yang_Dachun09}, \cite{DLM10}, \cite{WL_Besov_12}, \cite{LuZhu13_besov}, \cite{SteinYung13_subalgebra_flag}, \cite{Wu_flag_weight_14}, \cite{Han19_weighted}, \cite{Duong_Ji_Ou_Pipher_wick19}, \cite{Glo19_related_flags}, \cite{Han_Li_wick_22_flag_hardy}, and the references therein.


In the interest of clarity, consider for now the $2$-parameter setting. \textit{Product kernels} relative to the decomposition $\R^{q_1} \times \R^{q_2}$ are distributions satisfying a growth condition given as follows: for every multi-index $(\ap_1, \ap_2) \in \N^{q_1} \times  \N^{q_2}$,
\begin{equation}
    |\p^{\ap_1}_{t_1} \p^{\ap_2}_{t_2} K(t_1, t_2)| \leq C_{\ap} |t_1|_1^{-Q_1- \deg \ap_1} |t_2|_2^{-Q_2- \deg \ap_2},
\end{equation}
where $|\cdot|_{\mu}$ is a ``homogeneous norm'' on $\R^{q_{\mu}}$, for $\mu=1, 2$ (see Definition \ref{def homogeneous norm}). In particular, product kernels are smooth away from the ``cross'' $t_1 =0$, $t_2 =0$. Product kernels also satisfy a cancellation condition defined recursively (see Definition \ref{def pk}). On the other hand, \textit{flag kernels} satisfy a growth condition that presents more singularity in the first variable:
\begin{equation}
    |\p^{\ap_1}_{t_1} \p^{\ap_2}_{t_2} K(t_1, t_2)| \leq C_{\ap} |t_1|_1^{-Q_1- \deg \ap_1} (|t_1|_1+|t_2|_2)^{-Q_2- \deg \ap_2}.
\end{equation}
$2$-parameter flag kernels are thus smooth away from the coordinate axis $t_1 =0$. Flag kernels also satisfy a cancellation condition defined recursively (see Definition \ref{def FK}). Notably, neither class of distributions is pseudolocal. More recently, \cite{NRSW18} studied a subalgebra of flag kernels, related to subelliptic problems, which are better behaved and are pseudolocal. 

The main idea in the proof is the construction of two countable families of seminorms $\norm{K}{(k_1, k_2)}$ (see Definition \ref{def mP seminorms}) and $ \brac{K}{(k_1, k_2)}$ (see Definition \ref{def fk new seminorm}) that are adapted to product kernels and flag kernels respectively. Using our notation, the integers $k_1, k_2 \in \Z_{\geq 0}$ measure the order of regularity of the distribution $K$ in each factor space. 

Our first main result is the following: 

\begin{theorem}[Tame algebra estimate for product kernels]\label{thm tame estimate pk}
    Let $(k_1, k_2) \in \Z^2_{\geq 0}$. Suppose $K, L \in \mP^{(k_1, k_2)} (G_1 \times G_2)$. Then, we have
    \begin{equation}\label{eq pk tame estimate}
        \begin{split}
            \norm{K*L}{(k_1, k_2)} \lesssim & \norm{\Op(K)}{\mB(L^2(G_1 \times G_2))} \norm{L}{(k_1, k_2)} + \norm{K}{(k_1, 0)}^{\{1\}} \norm{L}{(0, k_2)}^{\{2\}} \\
            & + \norm{K}{(0, k_2)}^{\{2\}} \norm{L}{(k_1, 0)}^{\{1\}} +\norm{K}{(k_1, k_2)}\norm{\Op(L)}{\mB(L^2(G_1 \times G_2))}, 
        \end{split}
    \end{equation}
    where $\Op(K)$ denotes the right-invariant operator given by group convolution $\Op(K)f= K*f$. The implicit constant depends on $(k_1, k_2) \in \Z^2_{\geq 0}$.
\end{theorem}

Such estimates are called \textit{tame algebra estimates}\footnote{See Definition 1.2.1 p.135 in \cite{Ham82}.}  because of two special features.
\begin{enumerate}
    \item The estimates are \textit{tame} because the highest order seminorm in each parameter appears only once in every summand on the right-hand side of the inequality. All other seminorms in the chosen parameter and the chosen summand are of order zero.

    \item In addition, the estimates are \textit{tame algebra estimates} because we compose two distributions together via the noncommutative group convolution determined by the direct product of graded Lie groups $G= G_1 \times G_2$. 
\end{enumerate}
Tame estimates are central to the study of nonlinear partial differential equations via the Nash-Moser inverse function theorem (see \cite{Ham82}). Related tame-like estimates also appear under the name of \textit{fractional Leibniz rule} or \textit{Kato-Ponce inequalities}. See for instance the related works by Kato and Ponce in \cite{kato_ponce_88}, Muscalu, Pipher, Tao, and Thiele in \cite{Muscalu_pipher_tao_thiele_04}, Grafakos and S. Oh in \cite{grafakos_kato_14}, Bernicot, Maldonado, Moen, and Naibo in \cite{Bernicot_naibo_14}, and the references therein. 

Our estimate is therefore a multi-parameter version of a tame algebra estimate as they are referred to in the literature. Our second main result is the following tame algebra estimate for flag kernels under group convolution. 

\begin{theorem}[Tame algebra estimate for flag kernels]\label{thm tame FK}
Let $(k_1, k_2) \in \Z^2_{\geq 0}$. Suppose $K, L \in \mF^{(k_1, k_2)} (G_1 \times G_2)$. Then, we have
    \begin{equation}
        \begin{split}
        \brac{K*L}{(k_1, k_2)} \lesssim & \norm{\Op(K)}{\mB(L^2(G_1 \times G_2))} \brac{L}{(k_1, k_2)} + \norm{K}{(k_1, 0)}^{\{1\}} \norm{L}{( 0, k_2)}^{\{2\}} \\
        &+  \norm{K}{( 0, k_2)}^{\{2\}}\norm{L}{(k_1, 0)}^{\{1\}} + \brac{K}{(k_1, k_2)} \norm{\Op(L)}{\mB(L^2(G_1 \times G_2))},
        \end{split}
    \end{equation}
    where $\Op(K)$ denotes the right-invariant operator given by group convolution $\Op(K)f= K*f$. The implicit constant depends on $(k_1, k_2) \in \Z^2_{\geq 0}$.
\end{theorem}
Notice the presence of product kernel seminorms in the tame algebra estimate for flag kernels. The general $\nu$-parameter estimates with $\nu \geq 2$ are recorded in Theorem \ref{thm pk tame general nu} and Theorem \ref{thm general nu fk tame}. For the sake of clarity, we present the proof of both tame algebra estimates in the $2-$parameter case. The general $\nu$-parameter case follows with a few natural modifications.

The key idea in our proof is the construction of a family of well-adapted seminorms for the class of product kernels (respectively flag kernels) which satisfies two conflicting characteristics. 
\begin{enumerate}
    \item On the one hand, the seminorms have to be amenable to an $L^2(G_1 \times G_2)$ theory. More specifically, we apply the spectral theorem on $L^2(G_1 \times G_2)$ to obtain a new proof of the author's multi-parameter inversion theorem in \cite{Stoko23} extending the single-parameter inversion theorem by Christ, Geller, G{\l}owacki, and Polin in \cite{CGGP92}.

    \item On the other hand, a well-adapted characterization of the multi-parameter kernels should encapsulate their behavior away from the singularity restricted to each factor space $G_1$ and $G_2$. To do so, we have to introduce tensor products of functions living on each factor space in a meaningful way. 
\end{enumerate}

We thus construct a family of seminorms which simultaneously extends the single-parameter $L^2$-seminorms defined for a class of single-parameter homogeneous kernels\footnote{See definition 5.4 p.51 in \cite{CGGP92}.} by Christ et al. and builds upon the multi-parameter ``strengthened cancellation conditions'' defined for a more general class of multi-parameter singular integral operators\footnote{See Definition 5.1.4 p.269 in \cite{Str14} for a multi-parameter construction that applies to a more general class of operators.} by Street.

\subsection{Acknowledgments} \ 

I would like to thank my advisor Brian Street for proposing this problem and for the many discussions that guided my experimentation with various topologies. This material is based upon work supported by the NSF under Grant No. DMS-2037851.

\vspace{.2in}

\section{Background}

For every $\mu=1, \ldots, \nu$, let $\fg_{\mu}$ be a finite-dimensional graded Lie algebra. By definition, $\fg_{\mu}$ decomposes into a direct sum of vector spaces
\begin{align*}
    \fg_{\mu} = \bigoplus_{l=1}^{n_{\mu}} V_l^{\mu}, 
\end{align*}
where $[V_{l_1}^{\mu}, V_{l_2}^{\mu}] \subseteq V_{l_1 + l_2}^{\mu}$ and $V_l^{\mu} = \{0\}$, for $l > n_{\mu}$. The exponential map $\exp: \fg_{\mu} \rightarrow G_{\mu}$, where $G_{\mu}$ is the associated connected, simply connected graded Lie group, is a diffeomorphism\footnote{See Proposition 1.2 p.3 in \cite{FS82} for a proof of this.}. Let $q_l^{\mu} = \dim V_l^{\mu}$ and $q_{\mu} = \sum_{l=1}^{n_{\mu}} q_{l}^{\mu}$, the topological dimension of $G_{\mu}$. After picking a basis of left-invariant vector fields $\{X_1^{\mu}, \ldots, X_{q_{\mu}}^{\mu}\}$ for $\fg_{\mu}$, we obtain global coordinates $\R^{q_{\mu}} \rightarrow G_{\mu}$: 
\begin{align*}
    (t_1^{\mu}, \ldots, t_{q_{\mu}}^{\mu}) \mapsto \exp(t_1^{\mu} X_1^{\mu} + \ldots + t^{\mu}_{q_{\mu}} X_{q_{\mu}}^{\mu} ). 
\end{align*}
Given $x_{\mu} = (x_1^{\mu}, \ldots, x_{q_{\mu}}^{\mu})$ and $y_{\mu}= (y_1^{\mu}, \ldots, y_{q_{\mu}}^{\mu}) \in \R^{q_{\mu}}$, one can compute the non-commutative group multiplication $x_{\mu} \cdot y_{\mu}$ which is given by the coefficients of the basis vectors after applying the Baker-Campbell-Hausdorff formula:
\begin{align*}
    \BCH(x_1^{\mu} X_1^{\mu} + \ldots + x^{\mu}_{q_{\mu}} X_{q_{\mu}}^{\mu}, y_1^{\mu} X_1^{\mu} + \ldots + y^{\mu}_{q_{\mu}} X_{q_{\mu}}^{\mu}).
\end{align*}  
We henceforth identify $G_{\mu}$ with $\R^{q_{\mu}} = \R^{q_1^{\mu}} \times \cdots \times \R^{q_{n_{\mu}}^{\mu}}$ and obtain a family of automorphisms, called \textit{single-parameter dilations}, on $\R^{q_{\mu}}$: for $r_{\mu}>0$,
\begin{align*}
    r_{\mu} \cdot t_{\mu} = (r_{\mu} t^{\mu}_1, r_{\mu}^2 t^{\mu}_2, \ldots, r_{\mu}^{n_{\mu}} t^{\mu}_{n_{\mu}}). 
\end{align*}
Let $Q_{\mu} = \sum_{l=1}^{n_{\mu}}l \cdot q_l^{\mu}$ denote the associated ``homogeneous dimension'' of $\R^{q_{\mu}}$. 

Let $\{X_1^{\mu}, \ldots, X_{q_{\mu}}^{\mu}\}$ and $\{Y_1^{\mu}, \ldots, Y_{q_{\mu}}^{\mu}\}$ denote spanning sets of left- and right-invariant vector fields on $G_{\mu}$ s.t. at the origin, $X_j^{\mu} = Y_j^{\mu} = \frac{\p}{\p x_j^{\mu}}$. Note that $X_j^{\mu}$ and $Y_j^{\mu}$ are both homogeneous\footnote{That is, $D(f(r_{\mu} \cdot t_{\mu})) = r_{\mu}^l (Df)(r_{\mu} \cdot t_{\mu})$, for all $r_{\mu} >0$ and $D= X_j^{\mu}$ or $Y_j^{\mu}$.} of degree $l$, provided $x_j^{\mu} \in \R^{q_l^{\mu}}$. For $r \in [0,\infty)^{\nu}$, let $rX$ denote the following ordered list of vector fields with appropriate dilations:
\begin{equation*}
    rX= r_1^{\wh{d}^1}  X^1, \ldots, r_{\nu}^{\wh{d}^{\nu}}  X^{\nu}
    = r_1^{d_1^1}X_1^1, \ldots , r_1^{d_{q_1}^1} X_{q_1}^1, \ldots, r_{\nu}^{d_1^{\nu}} X_1^{\nu}, \ldots, r_{\nu}^{ d_{q_{\nu}}^{\nu}}X_{q_{\nu}}^{\nu},
\end{equation*}
where $\wh{d}^{\mu} = (d_1^{\mu}, \ldots, d_{q_{\mu}}^{\mu})$ where in turn each $d^{\mu}_j \in \{ 1, \ldots, n_{\mu} \}$, for every $j = 1, \ldots, q_{\mu}$ and for every $\mu = 1, \ldots, \nu$.

Given a multi-index $\ap_{\mu} = (\ap_1^{\mu}, \ldots, \ap_{n_{\mu}}^{\mu}) \in \N^{q_{\mu}} = \N^{q_1^{\mu}} \times \cdots \times \N^{q_{n_{\mu}}^{\mu}}$, let $\deg \ap_{\mu} = \sum_{l=1}^{n_{\mu}} l \norm{\ap_l^{\mu}}{l^1}$, denote its homogeneous degree and $|\ap_{\mu}| = \sum_{l=1}^{n_{\mu}} \norm{\ap_l^{\mu}}{l^1}$, its isotropic degree. In addition, given a multi-index $\ap = (\ap_1, \ldots, \ap_{\nu}) \in \N^{q_1} \times \cdots \times \N^{q_{\nu}}$, let $|\ap|= (|\ap_1|, \ldots, |\ap_{\nu}|)$ and $\deg \ap = \lp \deg \ap_{1}, \ldots, \deg \ap_{\nu} \rp$.

\begin{notation}
    Given a multi-index $\ap \in \N^{q_1} \times \cdots \times \N^{q_{\nu}}$, we will denote by $|\ap| \leq \vk$ the component-wise inequality $|\ap_{\mu}| \leq k_{\mu}$ for $\mu =1, \ldots, \nu$. 
\end{notation}

\begin{definition}\label{def homogeneous norm}
A \emph{homogeneous norm} on $\R^{q_{\mu}}$ is a continuous function $|\cdot|_{\mu}: \R^{q_{\mu}} \rightarrow [0,\infty)$ that is smooth away from $0$ with $|t_{\mu}|_{\mu} =0$ $\Leftrightarrow \ t_{\mu} =0$ and $|r_{\mu} \cdot t_{\mu}|_{\mu} = r_{\mu}|t_{\mu}|_{\mu}$ for $r_{\mu}>0$.
\end{definition}

Select any homogeneous norm $|\cdot|_{\mu}$ on each factor $\R^{q_{\mu}}$ since all such norms are equivalent. For example, for $X = \sum_{l=1}^{n_{\mu}} \sum_{k_l =1}^{ q_l^{\mu} } t_{l, k_l}^{\mu} X_{k_l}^{\mu}$, we can take
\beq\label{eq hom norm}
    |t_{\mu}|_{\mu}:= \lp \sum_{l=1}^{n_{\mu}} \sum_{k_l =1}^{q_l^{\mu} } |t_{l, k_l}^{\mu}|^{2(n_{\mu}!)/ l} \rp^{1/(2 (n_{\mu}!))}.
\eeq
For $x_{\mu} \in \R^{q_{\mu}}$ and $r_{\mu}>0$, we define
\begin{align*}
    B^{q_{\mu}}(x_{\mu}, r_{\mu}):= \{y_{\mu} \in \R^{q_{\mu}}; |x_{\mu}^{-1}y_{\mu}|_{\mu} < r_{\mu}\}.
\end{align*}
Euclidean Lebesgue measure on $\R^{q_{\mu}}$ is both left- and right-invariant with respect to the group operation. In other words, Lebesgue measure is a Haar measure on $G_{\mu}$ and we will denote it by $dx_{\mu}$.

\begin{notation}
    Given a distribution $K(t) \in C^{\infty}_c(\R^q)'$, by an abuse of notation, we will at times use an integral notation to denote the pairing of a distribution $K$ with a test function $\phi \in C^{\infty}_c(\R^q)$. That is, we will write $\int K(t) \phi(t)dt$. 
\end{notation}

\vspace{.1in}

\section{Product kernels}

\subsection{A Fréchet topology on product kernels} \

Before defining a carefully chosen family of seminorms which will give rise to a tame algebra estimate, we record the definition of product kernels one encounters in the literature\footnote{See for instance Definition 2.1.1 p.34 in \cite{NRS01}.}. 

\begin{definition}\label{def pk}
    The space of \textit{product kernels} is a locally convex topological vector space made of distributions $K \in C^{\infty}_c(\R^q)'$ on $\R^q = \R^{q_1}\times \cdots \times \R^{q_{\nu}} $. The space is defined recursively. For $\nu=0$, it is defined to be $\C$, with the usual locally convex topology. We assume that we have defined locally convex topological vector spaces of product kernels up to $\nu -1$ factors, and we define it for $\nu$ factors. The space of product kernels is the space of distributions $K \in C^{\infty}_c(\R^q)'$ s.t. the following two types of semi-norms are finite:
    \begin{enumerate}
        \item (Growth condition) For each multi-index $\ap = (\ap_1, \ldots, \ap_{\nu}) \in \N^{q_1} \times \cdots \times \N^{q_{\nu}} = \N^q$, we assume that there is a constant $C= C(\ap)$ s.t.
        \begin{equation}\label{eq pk gc}
            |\p^{\ap_1}_{t_1} \cdots \p^{\ap_{\nu}}_{t_{\nu}} K(t)| \leq C |t_1|_1^{-Q_1 -\deg \ap_1} \cdots |t_{\nu}|_{\nu}^{-Q_{\nu} - \deg \ap_{\nu}}.
        \end{equation}
        We define a semi-norm to be the least possible $C$.

        \item (Cancellation condition) Given $1 \leq \mu \leq \nu$, $R>0$, and a bounded\footnote{As a corollary of Proposition 14.6 p.139 in \cite{Tr67}, a set $\mB \subseteq C^{\infty}_c(\R^n)$ is bounded if the following two conditions hold: 

(1) there exists a compact set $M \Subset \R^{n}$ s.t. for all $f \in \mB$, $\supp f \subseteq M$;

(2) For every multi-index $\ap \in \N^n$, $\sup_{x\in \R^{n};  f\in \mB} | \p^{\ap} f(x)|<\infty$,

\noindent where $C^{\infty}_c(\R^n)$ denotes the set of compactly supported smooth functions.} set $\mB \subseteq C^{\infty}_c(\R^{q_{\mu}})$, for $\vp \in \mB$, we define
        \begin{equation}
            K_{\vp, R}(\ldots, t_{\mu -1}, t_{\mu +1}, \ldots) = \int K(t) \vp(R t_{\mu} )dt_{\mu},
        \end{equation}
        which defines a distribution
        \begin{equation}
            K_{\vp, R} \in C^{\infty}_c(\cdots \times \R^{q_{\mu-1} } \times \R^{q_{\mu +1}} \times \cdots)'.
        \end{equation}
        We assume that this distribution is a product kernel. Let $\norm{\cdot}{}$ be a continuous semi-norm on the space of $(\nu -1)-$factor product kernels. We define a semi-norm on $\nu-$factor product kernels by $\norm{K}{}:= \sup_{\vp \in \mB; R>0} \norm{K_{\vp, R}}{}$, which we assume to be finite. 
    \end{enumerate}
    We give the space of product kernels the coarsest topology such that all of the above semi-norms are continuous.
\end{definition}

In the construction of our new seminorms, we need to establish the following notations. 
\begin{notation}
    Given $S \subseteq \{1, \ldots, \nu\}$, and $\ap \in \N^{q_1} \times \cdots \times \N^{q_{\nu}}$, denote by $\ap^S$ the multi-index $\ap$ with all coordinates $\ap_{\mu} \in \N^{q_{\mu}}$ with $\mu \notin S$ removed. For example, for $\ap = (\ap_1, \ldots, \ap_{\nu}) \in \N^{q}$, we write $\ap^{\{1\}} = (\ap_1)$ and $\ap^{\{2, \ldots, \nu\}} = (\ap_2, \ldots, \ap_{\nu})$. In addition, given $j, l \in \Z$, let $j \vee l := \max \{j, l\}$.
\end{notation}

\begin{remark}
    Our seminorms are inspired by two families of seminorms found in previous work by Christ, Geller, G{\l}owacki, and Polin in \cite{CGGP92}, and Street in \cite{Str14}. 
    \begin{enumerate}
        \item The first family can be found in Definition 5.4 p.51 in \cite{CGGP92}. It defines an algebra of \textit{single-parameter homogeneous right-invariant operators} on a graded Lie group given by $Tf = af + K*f$, where $a \in \C$ and $K$ is a principal value distribution. For every $k \in \N$, Christ et al. define
    \begin{equation*}
        \norm{K}{k} := |a| + \norm{\eta K}{L^2_k},
    \end{equation*}
    where $\eta \in C^{\infty}_c(\R^q \backslash \{0\})$ and $\eta \equiv 1$ in a neighborhood of $\{|x| =1\}$. 

        \item The second family of seminorms applies to a large class of multi-parameter singular integrals and can be found in Definition 5.1.4 p.269 in \cite{Str14}. A single-parameter variant given in Theorem 2.7.15 (iii) in \cite{Str14} may be more illuminating so we record a modified version of it below. The seminorms on the space of Calderón-Zygmund operators $T$ on a compact subset of a graded Lie group $G$ are defined to be the least $C = C(\ap, \beta, \mB)$ in the inequality below such that for every multi-index $\ap, \beta$ and every ``bounded\footnote{In short, this means $\supp \phi_1 \subseteq B^q(x, 2^{-j_1})$, for $j_1 \geq 0$, and for every multi-index $\ap \in \N^q$, $\sup_{w \in \R^q} 2^{-j_1 Q}|(2^{-j_1}X)^{\ap} \phi(w)| \leq C(\mB, \ap) $ with $(\phi_2, z, 2^{-j_2})$ satisfying similar conditions. } set of bump functions'' $\mB$,
    \begin{equation*}
            |\la X^{\ap} \phi_1, T X^{\beta} \phi_2 \ra| \leq C (2^{-\min \{j_1, j_2 \}} + |x^{-1}z|)^{-Q-\deg \ap - \deg \beta},
    \end{equation*}
    where $(\phi_1, x, 2^{-j_1}), (\phi_2, z, 2^{-j_2}) \in \mB$.
    \end{enumerate}
\end{remark}

\begin{remark}
    To adapt the single-parameter seminorms in \cite{CGGP92} to the $2$-parameter case, one may be tempted to make use of operator-valued norms in the form of injective norms, or projective tensor norms\footnote{See chapter 44 p.446-458 and chapter 45 p.459-476 \cite{Tr67} for examples of completions of the tensor product of two locally convex topological vector spaces. }. However, the use of injective norms and of projective tensor norms defining spaces of the form $C^k(\R^{q_1}) \wh{\otimes}_{\ep} \mB(L^2(\R^{q_2}))$ and $C^k(\R^{q_1}) \wh{\otimes}_{\pi} \mB(L^2(\R^{q_2}))$ respectively prevent the application of the spectral theorem on $L^2(\R^{q_1} \times \R^{q_2})$ in the final steps of the proof of Theorem \ref{theorem inverse}. 
\end{remark}

We record the family of seminorms in the general $\nu$-parameter setting in the following definition. 
\begin{definition}\label{def mP seminorms}
    For $\vk \in \Z_{\geq 0}^{\nu}$, let $\mP^{\vk}(\R^{q_1}\times \cdots \times \R^{q_{\nu}})$ denote the space of distributions $K \in C^{\infty}_c(\R^q)'$ for which the following seminorms are finite:
\begin{equation}\label{eq pk k new seminorm}
    \begin{split}
        \norm{K}{\vk} &:= \sum_{S \subseteq \{1, \ldots,  \nu\} } \norm{K}{\vk^S}^S,
    \end{split}
\end{equation}
where 
\begin{equation*}
    \begin{split}
        \norm{K}{\vk^S}^S &:= \sum_{|\ap^S| \leq \vk^S} \sup_{\substack{\norm{f}{L^2(\R^q)}  =1;\\
        \norm{g}{L^2(\R^q)} =1 }} \sup_{\substack{j_{\mu}, l_{\mu} \in \Z; \\ \mu \in S }}  \sup_{\substack{w_{\mu}, z_{\mu}\in \R^{q_{\mu}}; \\ |w_{\mu} z_{\mu}^{-1}|_{\mu} \geq 3C_{\mu} 2^{j_{\mu} \vee l_{\mu}}; \\ \mu \in S }}
         \Big| \Big\la \bigotimes_{\mu \in S} \phi_{\mu} f, X^{\ap^S} K* \bigotimes_{\mu \in S} \g_{\mu} g \Big\ra \Big| \\
        & \hspace{3.5in} \times \prod_{\mu \in S}  |w_{\mu} z_{\mu}^{-1}|_{\mu}^{Q_{\mu} + \deg \ap_{\mu} }, 
    \end{split}
\end{equation*}
where in turn $\phi_{\mu}, \g_{\mu} \in C^{\infty}_c(\R^{q_{\mu}})$ are such that $0\leq \phi_{\mu}, \g_{\mu} \leq 1$, $\supp \phi_{\mu} \subseteq B^{q_{\mu}}(w_{\mu}, 2^{j_{\mu}})$, $\supp \g_{\mu} \subseteq B^{q_{\mu}}(z_{\mu}, 2^{l_{\mu}})$, and where $C_{\mu} \geq 1$ depends only on the homogeneous norm\footnote{The constant $C_{\mu}$ depends on the homogeneous quasi-norm as follows. It corresponds to the constant in the ``almost'' triangle inequality $|x_{\mu}y_{\mu}^{-1}|_{\mu} \leq C_{\mu}(|x_{\mu}|_{\mu} + |y_{\mu}|_{\mu})$.} $|\cdot|_{\mu}$ on $\R^{q_{\mu}}$, for $\mu =1, \ldots, \nu$.
\end{definition}

\begin{notation}
    We will henceforth write $A \lesssim B$ to mean $A \leq CB$, where $C$ denotes a scalar multiple of the constant $C_{\mu} \geq 1$ appearing in the ``almost'' triangle inequality satisfied by the homogeneous norms $|\cdot|_{\mu}$ on each factor space $\R^{q_{\mu}}$: $|x_{\mu}y_{\mu}^{-1}|_{\mu} \leq C_{\mu}(|x_{\mu}|_{\mu} + |y_{\mu}|_{\mu})$.
\end{notation}

\begin{remark}
    In the $\nu =2$ parameter case, the seminorms are given by
\begin{equation*}
    \begin{split}
        \norm{K}{(k_1, k_2)} = & \norm{\Op(K)}{\mB(L^2(\R^q))}\\
        &+ \sum_{|\ap_1| \leq k_1} \sup_{\substack{\norm{f}{L^2(\R^q)}  =1;\\
        \norm{g}{L^2(\R^q)} =1 }} \sup_{j_1, l_1 \in \Z }  \sup_{\substack{w_1, z_1 \in \R^{q_1} ; \\ |w_1 z_1^{-1}|_1 \gtrsim 2^{j_1 \vee l_1} }}  |\la \phi_1  f, X^{\ap_1} K* \g_1  g \ra| |w_1 z_1^{-1}|_1^{Q_1 + \deg \ap_1} \\
        &+\sum_{|\ap_2| \leq k_2} \sup_{\substack{\norm{f}{L^2(\R^q)}  =1;\\
        \norm{g}{L^2(\R^q)} =1 }}  \sup_{j_2, l_2 \in \Z}  \sup_{\substack{w_2, z_2\in \R^{q_2}; \\ |w_2 z_2^{-1}|_2 \gtrsim  2^{j_2 \vee l_2} }}  |\la \phi_2  f, X^{\ap_1} K* \g_2  g \ra|  |w_2 z_2^{-1}|_2^{Q_2 + \deg \ap_2}\\
        &+\sum_{\substack{|(\ap_1, \ap_2)| \leq (k_1, k_2)}}   \sup_{\substack{\norm{f}{L^2(\R^q)}  =1;\\
        \norm{g}{L^2(\R^q)} =1 }} \sup_{j,l \in \Z^2 } \sup_{\substack{w,z\in \R^q; \\ |w_1 z_1^{-1}|_1 \gtrsim  2^{j_1 \vee l_1}; \\ |w_2 z_2^{-1}|_2 \gtrsim  2^{j_2 \vee l_2} }}  |\la \phi_1 \otimes \phi_2 f, X^{\ap} K* \g_1 \otimes \g_2 g  \ra| \\
        & \hspace{2.5in}  \times  |w_1 z_1^{-1}|_1^{Q_1 + \deg \ap_1}  |w_2z_2^{-1}|_2^{Q_2 + \deg \ap_2}, 
    \end{split}
\end{equation*}
where $\phi_1, \g_1 \in C^{\infty}_c(\R^{q_1})$ are such that $0 \leq \phi_1, \g_1 \leq 1$, $\supp \phi_1 \subseteq B^{q_1}(w_1, 2^{j_1})$, $\supp \g_1 \subseteq B^{q_1}(z_1, 2^{l_1})$. $\phi_2, \g_2$ satisfy similar conditions on $\R^{q_2}$. 

In particular, observe that $\norm{K}{(k_1, 0)} \neq \norm{K}{(k_1, 0)}^{\{1\} }$. The term of the left-hand side of the inequality involves a localization in \textit{both} factor spaces $\R^{q_1}$ and $\R^{q_2}$; while the term on the right-hand side of the inequality involves a localization in a single factor space $\R^{q_1}$. This observation will play an important role in the proof of our tame algebra estimate in Theorem \ref{thm tame estimate pk} and in the proof of Theorem \ref{theorem inverse}.
\end{remark}

\begin{remark}\label{remark w=0}
    By translation-invariance, it suffices to consider seminorms defined in terms of test functions $\phi_{\mu} \in C^{\infty}_c(\R^{q_{\mu}})$, where $\supp \phi_{\mu} \subseteq B^{q_{\mu}}(0, 2^{j_{\mu}})$; that is, in the notation above, $w_{\mu} = 0$ for every $\mu =1, \ldots, \nu$. Although the symmetric nature of the seminorms in \eqref{eq pk k new seminorm} is not necessary, it will be of some use to us later on so we write it as such. 
\end{remark}

Before proving the tame algebra estimate for product kernels, we need to verify that the seminorms in Definition \ref{def mP seminorms} are suitable. We do so in the next proposition. But first, we introduce a class of functions which we will make use of repeatedly throughout the proof.

\begin{definition}\label{def S0}
    Given $S\subseteq \{1, \ldots, \nu\}$, let $\mS_0^{S}$ denote the space of functions $f \in \mS(\R^q)$ such that for every $\mu \in S$, and every $\ap \in \N^{q_{\mu}}$, 
    \begin{align*}
        \int t_{\mu}^{\ap} f(t_1, \ldots, t_{\nu})dt_{\mu} =0. 
    \end{align*}
\end{definition}

\begin{proposition}\label{prop pk equiv topologies}
    The seminorms defined above in Definition \ref{def pk} and in Definition \ref{def mP seminorms} induce equivalent topologies on the space of product kernels. 
\end{proposition}
\begin{remark}
    Although the topology as described in Definition \ref{def pk} is defined by an \textit{uncountable} family of seminorms with uncountably many bounded sets of test functions $\mB_{\mu}\subseteq C^{\infty}_c(\R^{q_{\mu}})$, the resulting space is in fact a Fréchet space by Definition \ref{def mP seminorms}. 
\end{remark}

To prove Proposition \ref{prop pk equiv topologies}, we first need to establish the following technical lemma.


\begin{lemma}\label{lemma pk l2 bddness}
    Let $K \in C^{\infty}_c(\R^q)'$ be a product kernel. Suppose $\supp \phi_{\mu} \subseteq B^{q_{\mu}}(w_{\mu}, 2^{j_{\mu}})$, $\supp \g_{\mu} \subseteq B^{q_{\mu}}(z_{\mu}, 2^{l_{\mu}})$, and $j_{\mu}, l_{\mu} \in \Z$, for $\mu \in S$ and $S \subsetneq \{1, \ldots, \nu\}$, then the operators
\begin{align*}
    \Big(\prod_{\mu \in S} |w_{\mu} z_{\mu}^{-1}|_{\mu}^{Q_{\mu} + \deg \ap_{\mu} } \Big) \bigotimes_{\mu \in S}\phi_{\mu} \Op(X^{\ap^S} K) \bigotimes_{\mu \in S} \g_{\mu}
\end{align*}
are $L^2(\R^q)$-bounded with operator norm \textit{uniformly bounded} in $w_{\mu}, z_{\mu} \in \R^{q_{\mu}}$, where $|w_{\mu} z_{\mu}^{-1}|_{\mu} \gtrsim 2^{j_{\mu} \vee l_{\mu}}$.
\end{lemma}

\begin{proof}[Proof of Lemma \ref{lemma pk l2 bddness}]
Let $\vp_n^{(2^n)}(t_1, \ldots, t_{\nu}) = 2^{n\cdot Q} \vp_n(2^{n_1} t_1, \ldots, 2^{n_{\nu}} t_{\nu})$, for $n \in \Z^{\nu}$. By Corollary 5.2.16 in \cite{Str14}, there exists a bounded set $\{ \vp_n; n \in \Z^{\nu}\} \subseteq \mS_0^{\{1, \ldots, \nu\}}$ such that\footnote{See Definition \ref{def S0} for a definition of the set $\mS_0^{\{1, \ldots, \nu\}}$.}
    \begin{align*}
        K = \sum_{n \in \Z^{\nu}} \vp_n^{(2^n)},
    \end{align*}
    where the sum converges in the sense of tempered distributions. We can thus decompose the operator of interest as follows:
    \begin{align*}
        &\Big( \prod_{\mu \in S}  |w_{\mu} z_{\mu}^{-1}|_{\mu}^{Q_{\mu} + \deg \ap_{\mu} }  \Big) \bigotimes_{\mu \in S}\phi_{\mu} \Op(X^{\ap^S} K) \bigotimes_{\mu \in S} \g_{\mu}  \\
        &= \sum_{n \in \Z^{\nu}} \Big(\prod_{\mu \in S} |w_{\mu} z_{\mu}^{-1}|_{\mu}^{Q_{\mu} + \deg \ap_{\mu} }  \Big)\bigotimes_{\mu \in S}\phi_{\mu} \Op(X^{\ap^S} \vp_n^{(2^n)}) \bigotimes_{\mu \in S} \g_{\mu},
    \end{align*}
    where $2^{M_{\mu} -1} \leq |w_{\mu}z_{\mu}^{-1}|_{\mu} \leq 2^{M_{\mu} +1}$, for some $M_{\mu} \in \N$. Applying the differential operators, the previous equation is
    \begin{align*}
        & \approx \sum_{n \in \Z^{\nu}} \Big( \prod_{\mu \in S} 2^{M_{\mu}(Q_{\mu} + \deg \ap_{\mu}) } 2^{n_{\mu} \deg \ap_{\mu}} \Big) \bigotimes_{\mu \in S}\phi_{\mu} \Op( \zeta_n^{(2^n)}) \bigotimes_{\mu \in S} \g_{\mu} ,
    \end{align*}
    where $\zeta_n = X^{\ap^S} \vp_n$ so that $\{\zeta_n; n\in \Z^{\nu}\} \subseteq \mS_0^{\{1, \ldots, \nu\}}$ is bounded. For every $n \in \Z^{\nu}$, let
    \begin{equation}\label{eq Tj def}
        \begin{split}
            T_n:= \Big( \prod_{\mu \in S} 2^{M_{\mu}(Q_{\mu} + \deg \ap_{\mu}) } 2^{n_{\mu} \deg \ap_{\mu}} \Big) \bigotimes_{\mu \in S}\phi_{\mu}  \Op( \zeta_n^{(2^n)}) \bigotimes_{\mu \in S} \g_{\mu} .
        \end{split}
    \end{equation}
    In view of applying Cotlar-Stein's lemma to bound the $L^2$ operator norm of the operator $\sum_{n \in \Z^{\nu}} T_n$ uniformly in $M_{\mu}$, we first record and verify the following lemma. 

\begin{claim}\label{claim TjTk}
For all $n, m\in \Z^{\nu}$, we have
    \begin{equation}
    \begin{split}
        \norm{T_nT_m^*}{\mB(L^1(\R^q))} & \lesssim 2^{-|n-m|}, \hspace{.5in} \norm{T_n^*T_m}{\mB(L^{1}(\R^q))} \lesssim 2^{-|n-m|},\\
        \norm{T_nT_m^*}{\mB(L^{\infty}(\R^q))} & \lesssim 2^{-|n-m|}, \hspace{.5in} \norm{T_n^*T_m}{\mB(L^{\infty}(\R^q))} \lesssim 2^{-|n-m|},
    \end{split}
    \end{equation}
    where\footnote{We use the standard multi-index notation $2^{-|n-m|} = \prod_{\mu=1}^{\nu} 2^{-|n_{\mu}-m_{\mu}|}$.} $T_n, T_m$ are defined as in \eqref{eq Tj def}. 
\end{claim}

Let $T_n(x, y)$ denote the Schwartz kernel of the operator $T_n$. It suffices to prove the bounds of the operator norms for $T_nT_m^*$. We thus need to estimate both
\begin{equation*}
    \sup_{y \in \R^q } \int_{\R^q} \lf \int_{\R^q} T_n(x,u) T_m^*(u,y) du \rf dx,
\end{equation*}
    and
\begin{equation*}
    \sup_{x \in \R^q } \int_{\R^q} \lf \int_{\R^q} T_n(x,u) T_m^*(u,y) du \rf dy
\end{equation*}
to estimate the $L^1(\R^q)$- and $L^{\infty}(\R^q)$-operator norms respectively. The two operator norms can be proved analogously. We will therefore only present the details of the former. 

By \eqref{eq Tj def}, we have an explicit formula for the Schwartz kernels. 
\begin{equation}\label{eq TjTk expand}
\begin{split}
    &\sup_{y \in \R^q } \int_{\R^q} \lf \int_{\R^q} T_n(x,u) T_n^*(u,y) du \rf dx\\
    & = \sup_{y \in \R^q} \int_{\R^q} \Big| \int_{\R^q}  \prod_{\mu \in S} 2^{2M_{\mu}(Q_{\mu} + \deg \ap_{\mu}) } 2^{(n_{\mu} + m_{\mu}) \deg \ap_{\mu}} \prod_{\mu \in S} \phi_{\mu}(x_{\mu}) \zeta_n^{(2^n)}(xu^{-1}) \prod_{\mu \in S} \g_{\mu} (u_{\mu}) \\
    &\hspace{2.5in}  \times  \overline{\prod_{\mu \in S} \g_{\mu} (u_{\mu}) \zeta_m^{(2^m)} (yu^{-1}) \prod_{\mu \in S} \phi_{\mu}(y_{\mu})} du \Big| dx,
    \end{split}
\end{equation}
where for every $\mu \in S$, $x_{\mu}, y_{\mu} \in \supp \phi_{\mu} \subseteq B^{q_{\mu}}(w_{\mu}, 2^{j_{\mu}})$ and $u_{\mu}\in \supp \g_{\mu} \subseteq B^{q_{\mu}}(z_{\mu}, 2^{l_{\mu}})$. Hence, $2^{M_{\mu}} \sim |x_{\mu}u_{\mu}^{-1}|_{\mu}, |y_{\mu}u_{\mu}^{-1}|_{\mu}$. Without loss of generality, we can assume that $m_{\mu} \leq n_{\mu}$ for every $\mu =1, \ldots, \nu$. Recall that $\{\zeta_n; n \in \Z^{\nu}\} \subseteq \mS_0^{\{1,\ldots, \nu\}}$ is bounded. We can thus write
    \begin{equation*}
        \zeta_n = \sum_{i=1}^q \frac{\p}{\p x_i} \zeta_{n, i},
    \end{equation*}
where $\{\zeta_{n, i}; n \in \Z^{\nu}, i=1, \ldots, q\} \subseteq \mS_0^{\{1, \ldots, \nu \}}$ is a bounded set\footnote{See Lemma 1.1.16 and the remark thereafter in \cite{Str14} for a proof of this.}.
Writing each standard vector field $\frac{\p}{\p x_i}$ as a linear combination of homogeneous vector fields,\footnote{For every $i=1, \ldots, q$, we have 
\begin{align*}
     \frac{\p}{\p x_i} = X_i + \sum_{\substack{1 \leq k \leq q }}  X_k p_{i,k}
\end{align*} where $p_{i,k}$ are homogeneous polynomials (see the proof of Proposition 3.1.28 in \cite{greenbook} to verify this identity). } we obtain
\begin{equation}\label{eq zetaj X}
    \zeta_n = \sum_{|\beta| =1} X^{\beta} \zeta_{n, \beta},
\end{equation}
where $\{ \zeta_{n, \beta}; |\beta| =1, n \in \Z^{\nu}\} \subseteq \mS_0^{\{1, \ldots, \nu \}}$ is a bounded set. By substituting \eqref{eq zetaj X} for $\zeta_n$ into \eqref{eq TjTk expand} and recalling that $\g_{\mu}, \phi_{\mu}$ are real-valued, we write
\begin{equation*}
\begin{split}
    &\sup_{y \in \R^q } \int_{\R^q} \lf \int_{\R^q} T_n(x,u) T_m^*(u,y) du \rf dx\\
    & = \sup_{y \in \R^q} \int_{\R^q} \Big| \int_{\R^q}  \prod_{\mu \in S} 2^{2M_{\mu}(Q_{\mu} + \deg \ap_{\mu}) } 2^{(n_{\mu} + m_{\mu}) \deg \ap_{\mu}} \prod_{\mu \in S} \phi_{\mu}(x_{\mu}) \Big( \sum_{|\beta| = N} X^{\beta} \zeta_{n, \beta}\Big)^{(2^n)}(x_1u_1^{-1}, \ldots, x_{\nu}u_{\nu}^{-1})  \\
    &\hspace{2in}  \times  \prod_{\mu \in S} \g_{\mu}^2 (u_{\mu})\overline{ \zeta_m^{(2^m)} (y_1u_1^{-1}, \ldots, y_{\nu}u_{\nu}^{-1})} \prod_{\mu \in S} \phi_{\mu}(y_{\mu}) du \Big| dx.
    \end{split}
\end{equation*}
In view of taking adjoints of the differential operator, we separate the vector fields from the dilated Schwartz functions. The previous equation is thus
\begin{equation*}
\begin{split}
    & = \sup_{y \in \R^q} \int_{\R^q} \Big| \int_{\R^q}  \prod_{\mu \in S} 2^{2M_{\mu}(Q_{\mu} + \deg \ap_{\mu}) } 2^{(n_{\mu} + m_{\mu}) \deg \ap_{\mu}} \prod_{\mu \in S} \phi_{\mu}(x_{\mu})  \sum_{|\beta| = N} (2^{-n}X)^{\beta} \zeta_{n, \beta}^{(2^n)}(x_1u_1^{-1}, \ldots, x_{\nu}u_{\nu}^{-1}) \\
    &\hspace{2in} \times  \prod_{\mu \in S} \g_{\mu}^2 (u_{\mu})  \overline{\zeta_m^{(2^m)} (y_1u_1^{-1}, \ldots, y_{\nu}u_{\nu}^{-1}) }\prod_{\mu \in S} \phi_{\mu}(y_{\mu}) du \Big| dx.
    \end{split}
\end{equation*}
For every $i=1, \ldots, q$, the homogeneous vector field $X_i$ is skew-adjoint; that is $X_i^* = -X_i$. We thus apply the differential operators to the functions $\zeta_m$. The previous equation is thus
\begin{equation*}
\begin{split}
    & = \sup_{y \in \R^q} \int_{\R^q} \Big| \int_{\R^q}  \prod_{\mu \in S} 2^{2M_{\mu}(Q_{\mu} + \deg \ap_{\mu}) } 2^{(n_{\mu} + m_{\mu}) \deg \ap_{\mu}} \prod_{\mu \in S} \phi_{\mu}(x_{\mu})  \sum_{|\beta| = N} 2^{-nN} \zeta_{n, \beta}^{(2^n)}(x_1u_1^{-1}, \ldots, x_{\nu}u_{\nu}^{-1})  \\
    &\hspace{2in}  \times  \prod_{\mu \in S} \wt{\g}_{\mu} (u_{\mu})  X^{\beta} \overline{\zeta_m^{(2^m)} (y_1u_1^{-1}, \ldots, y_{\nu}u_{\nu}^{-1}) }\prod_{\mu \in S} \phi_{\mu}(y_{\mu}) du \Big| dx,
    \end{split}
\end{equation*}
where $\wt{\g}_{\mu} = X^{\beta}\g_{\mu}^2$ for some multi-index $\beta$. We thus obtain an exponential decay as desired. The equation above is
\begin{equation*}
\begin{split}
    & = 2^{-(n-m)N} \sup_{y \in \R^q} \int_{\R^q} \Big| \int_{\R^q}  \prod_{\mu \in S} 2^{2M_{\mu}(Q_{\mu} + \deg \ap_{\mu}) } 2^{(n_{\mu} + m_{\mu}) \deg \ap_{\mu}} \prod_{\mu \in S} \phi_{\mu}(x_{\mu})  \sum_{|\beta| = N}  \zeta_{n, \beta}^{(2^n)}(x_1u_1^{-1}, \ldots, x_{\nu}u_{\nu}^{-1})  \\
    &\hspace{2in}  \times  \prod_{\mu \in S} \wt{\g}_{\mu} (u_{\mu})  \overline{ \zeta_{m, \beta}^{(2^m)} (y_1u_1^{-1}, \ldots, y_{\nu}u_{\nu}^{-1})} \prod_{\mu \in S} \phi_{\mu}(y_{\mu}) du \Big| dx,
    \end{split}
\end{equation*}
where $\zeta_{m, \beta}:= X^{\beta} \zeta_m$ and  $\{\zeta_{m, \beta}; |\beta| =N, m\in \Z^{\nu}\} \subseteq \mS_0^{\{1, \ldots, \nu\} }$ is a bounded set. By this boundedness and by the triangle inequality, the previous equation is
\begin{equation}\label{eq pk decay unif}
    \begin{split}
        \lesssim & 2^{-|n-m|N} \sup_{\substack{y_{\mu} \in \supp \phi_{\mu}; \\ \mu \in S}}\prod_{\mu \in S} 2^{2M_{\mu}(Q_{\mu} + \deg \ap_{\mu}) } 2^{(n_{\mu} + m_{\mu})\deg \ap_{\mu}}  \int_{\R^q} \int_{\R^q} \prod_{\mu \in S} \phi_{\mu}(x_{\mu}) \wt{\g}_{\mu} (u_{\mu}) \phi_{\mu}(y_{\mu}) \\
        & \hspace{1in} \times  \prod_{\mu =1}^{\nu} 2^{n_{\mu} Q_{\mu}} (1+2^{n_{\mu}}|x_{\mu}u_{\mu}^{-1}|_{\mu})^{-t_{\mu}} 2^{m_{\mu}  Q_{\mu}} (1+2^{m_{\mu}} |y_{\mu}u_{\mu}^{-1}|_{\mu})^{-t_{\mu}} du dx,
    \end{split}
\end{equation}
for all $t = (t_1, \ldots, t_{\nu}) \in \Z_{\geq 0}^{\nu}$. 
Notice that the integrals in $\R^{q_{\mu'}}$ with $\mu' \notin S$ are uniformly convergent\footnote{The dilations are $L^1$-normalized.}. In addition, by the compactness of both $\supp \phi_{\mu}$ and $ \supp \wt{\g}_{\mu}$, we bound the remaining $L^1$ norms by the $L^{\infty}$ norms multiplied by the size of their respective supports. The previous equation is thus
\begin{align*}
    & \lesssim 2^{-|n-m|N}  2^{2M^S \cdot (Q^S + \deg \ap^S) } 2^{(n^S + m^S) \cdot \deg \ap^S}  \\
        &\hspace{1in }  \times \prod_{\mu \in S} 2^{n_{\mu} Q_{\mu}} (1+2^{n_{\mu}}2^{M_{\mu}})^{-t_{\mu}} 2^{m_{\mu}  Q_{\mu}} (1+2^{m_{\mu}} 2^{M_{\mu}})^{-t_{\mu}} 2^{j_{\mu}Q_{\mu}} 2^{l_{\mu} Q_{\mu}}, 
\end{align*}
where $M_{\mu} \gtrsim j_{\mu} \vee l_{\mu}$. If $M_{\mu} \geq 0$, taking $t_{\mu} = Q_{\mu} + \deg \ap_{\mu}$ leads to the desired uniform decay estimate. Otherwise, if $M_{\mu} \leq 0$, we take $t_{\mu} >>1$ large enough so that 
\begin{align*}
    2^{2M^S \cdot (Q^S + \deg \ap^S) } 2^{(n^S + m^S) \cdot \deg \ap^S} \prod_{\mu \in S} 2^{n_{\mu} Q_{\mu}} (1+2^{n_{\mu}}2^{M_{\mu}})^{-t_{\mu}} 2^{m_{\mu}  Q_{\mu}} (1+2^{m_{\mu}} 2^{M_{\mu}})^{-t_{\mu}} 2^{2M_{\mu} Q_{\mu}} 
\end{align*}
is uniformly bounded in $n_{\mu}, m_{\mu}$, and $M_{\mu}$. In both cases, we obtain an exponential decay $2^{-|n-m|N}$. Taking $N=1$, we obtain the desired summable upper bound:
\begin{align*}
    \norm{T_nT_n^*}{\mB(L^1(\R^q))}&\lesssim  2^{-|n-m|}.
\end{align*}
Retracing the proof, we also obtain the remaining three desired estimates. Thus concluding the proof of Claim \ref{claim TjTk}. 

To complete the proof of Lemma \ref{lemma pk l2 bddness}, it suffices to successively interpolate and apply the Cotlar-Stein lemma. 

\end{proof}


\begin{proof}[Proof of Proposition \ref{prop pk equiv topologies}]
    Let $K \in C^{\infty}_c(\R^q)'$ be a product kernel as defined in Definition \ref{def pk}. We first want to show that $K \in \mP^{\vk}$, for all $\vk \in \Z_{\geq 0}^{\nu}$. Given $\vk \in \Z_{\geq 0}^{\nu}$, we decompose the seminorm $\norm{K}{\vk}$ into three terms associated to subsets $S$ given by $S = \emptyset$, $S \subsetneq \{1, \ldots, \nu\}$, and $S = \{1, \ldots, \nu\}$. We thus write
\begin{equation}\label{eq pk top equiv step 1}
    \begin{split}
        \norm{K}{\vk} =& \sup_{\substack{\norm{f}{L^2(\R^q)}  =1;\\
        \norm{g}{L^2(\R^q)} =1 }}  \lf \la  f, K*g \ra \rf\\
        &+\sum_{\substack{S \subsetneq \{1, \ldots, \nu\}; \\ S \neq \emptyset }} \sum_{|\ap^S| \leq \vk^S} \sup_{\substack{\norm{f}{L^2(\R^q)}  =1;\\
        \norm{g}{L^2(\R^q)} =1 }} \sup_{\substack{j_{\mu}, l_{\mu} \in \Z; \\ \mu \in S }}  \sup_{\substack{w_{\mu}, z_{\mu} \in \R^{q_{\mu}}; \\ |w_{\mu} z_{\mu}^{-1}|_{\mu} \gtrsim 2^{j_{\mu} \vee l_{\mu}}; \\ \mu \in S }} \Big| \Big\la \bigotimes_{\mu \in S} \phi_{\mu} f, X^{\ap^S} K* \bigotimes_{\mu \in S} \g_{\mu} g \Big\ra \Big| \\
        & \hspace{4in}  \times \prod_{\mu \in S}  |w_{\mu} z_{\mu}^{-1}|_{\mu}^{Q_{\mu} + \deg \ap_{\mu} }\\
        &+\sum_{|\ap| \leq \vk} \sup_{\substack{\norm{f}{L^2(\R^q)}  =1;\\
        \norm{g}{L^2(\R^q)} =1 }} \sup_{j ,l\in \Z^{\nu}}  \sup_{\substack{w, z\in \R^q  ; \\ |w_{\mu} z_{\mu}^{-1}|_{\mu} \gtrsim 2^{j_{\mu} \vee l_{\mu}}; \\ \mu = 1, \ldots, \nu }} \Big| \Big\la \bigotimes_{\mu=1}^{\nu} \phi_{\mu} f, X^{\ap} K* \bigotimes_{\mu =1}^{\nu} \g_{\mu} g \Big\ra \Big| \\
        & \hspace{4in}  \times  \prod_{\mu =1}^{\nu}  |w_{\mu} z_{\mu}^{-1}|_{\mu}^{Q_{\mu} + \deg \ap_{\mu} }.
    \end{split}
\end{equation}
Right-invariant operators $\Op(K) f = K*f$, where $K$ is a product kernel, are $L^2(\R^q)$-bounded\footnote{See Theorem 4.4 in \cite{MRS95} for a proof of this result.}. The first summand $\norm{\Op(K)}{\mB(L^2(\R^q))}$, in \eqref{eq pk top equiv step 1}, is thus bounded. By Lemma \ref{lemma pk l2 bddness}, the intermediate terms in \eqref{eq pk top equiv step 1} associated to nonempty subsets $S \subsetneq \{1, \ldots, \nu \}$ are bounded.

It remains to bound the last term in \eqref{eq pk top equiv step 1} associated to the maximal subset. Observe that the distribution $K$ is localized away from its singularity and can thus be identified with a smooth function. By the triangle inequality, we bound the last term in \eqref{eq pk top equiv step 1} by
\begin{align*}
    & \sum_{|\ap| \leq \vk} \sup_{\substack{\norm{f}{L^2(\R^q)}  =1;\\
    \norm{g}{L^2(\R^q)} =1 }} \sup_{j,l\in \Z^{\nu}}  \sup_{\substack{w, z \in \R^q  ; \\ |w_{\mu} z_{\mu}^{-1}|_{\mu} \gtrsim 2^{j_{\mu} \vee l_{\mu}}; \\ \mu = 1, \ldots, \nu }}
      \int_{\R^q} \int_{\R^q} \prod_{\mu=1}^{\nu} |\phi_{\mu} (x_{\mu})| |f(x_1, \ldots, x_{\nu})| | X^{\ap} K (x_1y_1^{-1}, \ldots, x_{\nu}y_{\nu}^{-1})| \\
    &\hspace{1.5in} \times |\g_{\mu} (y_{\mu}) ||g (y_1, \ldots, y_{\nu})| dy dx  \prod_{\mu =1}^{\nu} |w_{\mu} z_{\mu}^{-1}|_{\mu}^{Q_{\mu} + \deg \ap_{\mu} }.
\end{align*}
By the growth condition for product kernels (see \eqref{eq pk gc}), the previous equation is
\begin{align*}
    & \lesssim \sum_{|\ap| \leq \vk} \sup_{\substack{\norm{f}{L^2(\R^q)}  =1;\\
    \norm{g}{L^2(\R^q)} =1 }} \sup_{j,l\in \Z^{\nu}}  \sup_{\substack{w, z \in  \R^q ; \\ |w_{\mu} z_{\mu}^{-1}|_{\mu} \gtrsim 2^{j_{\mu} \vee l_{\mu}}; \\ \mu = 1, \ldots, \nu }}
     \int_{\R^q} \int_{\R^q} \prod_{\mu=1}^{\nu} |\phi_{\mu} (x_{\mu})| |f(x_1, \ldots, x_{\nu})| |x_{\mu}y_{\mu}^{-1}|_{\mu}^{-Q_{\mu} - \deg \ap_{\mu}}  \\
    &\hspace{1.5in} \times |\g_{\mu} (y_{\mu}) ||g (y_1, \ldots, y_{\nu})| dy dx  \prod_{\mu =1}^{\nu}  |w_{\mu} z_{\mu}^{-1}|_{\mu}^{Q_{\mu} + \deg \ap_{\mu} }.
\end{align*}
By compactness, we first bound the $L^1$ norm by the $L^2$ norm. Then by the generalized Hölder's inequality, the previous equation is
\begin{equation}\label{last step}
    \lesssim \sum_{|\ap| \leq \vk}  \sup_{j,l\in \Z^{\nu}}  \sup_{\substack{w, z\in \R^q ; \\ |w_{\mu} z_{\mu}^{-1}|_{\mu} \gtrsim 2^{j_{\mu} \vee l_{\mu}}; \\ \mu = 1, \ldots, \nu }} \sup_{\substack{x_{\mu} \in \supp \phi_{\mu}; \\ y_{\mu} \in \supp \g_{\mu}}} \prod_{\mu =1}^{\nu} |x_{\mu}y_{\mu}^{-1}|_{\mu}^{-Q_{\mu} - \deg \ap_{\mu}} \times \prod_{\mu =1}^{\nu}  |w_{\mu} z_{\mu}^{-1}|_{\mu}^{Q_{\mu} + \deg \ap_{\mu} }.
\end{equation}
By the sub-Riemannian geometry, there exists $C\geq 1$ such that
\begin{align*}
    |w_{\mu}z_{\mu}^{-1}|_{\mu} & \leq 2C 2^{j_{\mu}\vee l_{\mu}} + C |x_{\mu}y_{\mu}^{-1}|_{\mu}.
\end{align*}
By subtracting $2C 2^{j_{\mu}\vee l_{\mu}}$ on both sides, and recalling that $|w_{\mu}z_{\mu}^{-1}|_{\mu} \geq 3C 2^{j_{\mu}\vee l_{\mu}}$, we have
\begin{align*}
    |w_{\mu}z_{\mu}^{-1}|_{\mu} \lesssim  |x_{\mu}y_{\mu}^{-1}|_{\mu}.
\end{align*}
We thus obtain a uniform bound for \eqref{last step} as desired.

For the reverse direction, let $K \in \mP^{\vk}$, for all $\vk \in \Z_{\geq 0}^{\nu}$. We want to show that $K$ is a product kernel. To prove the growth condition \eqref{eq pk gc}, it suffices to consider the following two extreme cases separately:
\begin{enumerate}
    \item $|t_{\mu}|_{\mu} \geq 1$, for every $\mu =1, \ldots, \nu$, and

    \item $|t_{\mu}|_{\mu} <1$, for every $\mu =1, \ldots, \nu$.
\end{enumerate}
The intermediate cases follow from a few straightforward modifications to the above two cases. 

For the first case, by the Sobolev embedding, there exist $s = s(\ap) > 0$ such that given $t_{\mu} \geq 1$ for every $\mu =1, \ldots , \nu$, 
\begin{align}\label{eq pk step 1 gc}
    |\p^{\ap_1}_{t_1} \cdots \p^{\ap_{\nu}}_{t_{\nu}} K(t)| \lesssim \norm{\bigotimes_{\mu =1}^{\nu} \psi_{\mu} K}{L^2_s(\R^q)},
\end{align}
where $\supp \psi_{\mu} \subseteq B^{q_{\mu}}(t_{\mu}, \ep)$ for $\ep>0$ small enough so that $K$ and its derivatives up to order $s$ on each factor space $\R^{q_{\mu}}$ do not change signs and remain bounded away from zero when restricted to a $10\ep$-neighborhood of $\supp \psi_1 \otimes \cdots \otimes \psi_{\nu}$. We can thus take $\phi_{\mu}, \g_{\mu} \in C^{\infty}_c(\R^{q_{\mu}})$ defined so that $\supp \phi_{\mu} \subseteq B^{q_{\mu}}(t_{\mu}, \ep)$ and $\supp \g_{\mu} \subseteq B^{q_{\mu}} (0, \ep)$, $0 \leq \phi_{\mu}, \g_{\mu} \leq 1$, and so that 
\begin{align}\label{eq pk step 2 gc}
    \sum_{|\eta| \leq (s, \ldots, s)} \norm{\p^{\eta} \bigotimes_{\mu =1}^{\nu} \psi_{\mu} K}{L^2(\R^q)} & \lesssim \sum_{|\eta| \leq (s, \ldots, s)} \lp \int \int \lf \p^{\eta}_x \prod_{\mu=1}^{\nu} \phi_{\mu}(x_{\mu}) K(xy^{-1}) \prod_{\mu=1}^{\nu} \g_{\mu}(y_{\mu}) \rf^2 dy dx \rp^{1/2}.
\end{align}
Since the distribution $K$ is restricted to a neighborhood where it can be identified with a function whose derivatives up to order $(s, \ldots, s)$ do not change signs and remain bounded away from zero, we can move the square outward. The previous equation is thus
\begin{align*}
    & \lesssim \sum_{|\eta| \leq (s, \ldots, s)} \lp \int \lf \int  \p^{\eta}_x \prod_{\mu=1}^{\nu} \phi_{\mu}(x_{\mu}) K(xy^{-1}) \prod_{\mu=1}^{\nu} \g_{\mu}(y_{\mu})  dy \rf^2 dx \rp^{1/2}.
\end{align*}
By switching the order of integration and differentiation, it remains to estimate
\begin{align}\label{eq pk gc equiv top}
    \sum_{|\eta| \leq (s, \ldots, s)}\norm{ \p^{\eta} \bigotimes_{\mu =1}^{\nu} \phi_{\mu} K * \bigotimes_{\mu =1 }^{\nu} \g_{\mu} }{L^2(\R^q)}.
\end{align}
To do so, we write the standard vector fields $\lp \frac{\p}{\p x} \rp^{\eta}$ as linear combinations of left-invariant vector fields:
\begin{equation}\label{eq dx lin comb of X}
    \lp \frac{\p}{\p x} \rp^{\eta} = \sum_{\substack{|\beta| \leq |\eta|; \\ \deg \beta \geq \deg \eta}} p_{\eta, \beta} X^{\beta},
\end{equation}
where $p_{\eta, \beta}$ are homogeneous polynomials of homogeneous degree\footnote{See the remarks following Proposition 1.26 p.25 in \cite{FS82} for a proof of this identity.} $\deg \beta - \deg \eta$. 
We thus bound the standard Sobolev norms by homogeneous, nonisotropic Sobolev norms,
\begin{align*}
    \sum_{|\eta| \leq (s, \ldots, s)} \norm{ \lp \frac{\p}{\p x} \rp^{\eta} \bigotimes_{\mu =1}^{\nu} \phi_{\mu} K * \bigotimes_{\mu=1}^{\nu} \g_{\mu}}{L^2(\R^q)} & \lesssim  \sum_{\substack{\beta, \eta \in \N^q; \\ |\beta| \leq |\eta| \leq (s, \ldots, s); \\ \deg \beta \geq \deg \eta} } \norm{ p_{\eta, \beta} \bigotimes_{\mu =1}^{\nu} \wt{\phi}_{\mu} X^{\beta}\Big( K * \bigotimes_{\mu=1}^{\nu} \g_{\mu} \Big) }{L^2},
\end{align*}
where $\wt{\phi}_{\mu} =X^{\beta_\mu{}} \phi_{\mu}$ for some multi-index $\beta_{\mu} \in \N^{q_{\mu}}$. By Hölder's inequality,
\begin{align}\label{eq pk equiv top polys}
    \sum_{\substack{\beta, \eta \in \N^q; \\ |\beta| \leq |\eta| \leq (s, \ldots, s); \\ \deg \beta \geq \deg \eta} } \norm{p_{\eta, \beta}\bigotimes_{\mu =1}^{\nu} \wt{\phi}_{\mu} X^{\beta}K* \bigotimes_{\mu=1}^{\nu} \g_{\mu} }{L^2} \lesssim \sup_{\substack{\beta, \eta \in \N^q; \\ |\beta| \leq |\eta| \leq (s, \ldots, s); \\ \deg \beta \geq \deg \eta} }  \norm{p_{\eta, \beta}\bigotimes_{\mu =1}^{\nu} \wt{\phi}_{\mu}}{L^{\infty}}  \norm{\bigotimes_{\mu =1}^{\nu} \phi_{\mu} X^{\beta}K* \bigotimes_{\mu=1}^{\nu} \g_{\mu} }{L^2}.
\end{align}
The homogeneous polynomials $p_{\eta, \beta}$ of homogeneous degree $\deg \beta - \deg \eta$ restricted to the compact subset $B^{q_1}(t_1, \ep) \times \cdots \times B^{q_{\nu}}(t_{\nu}, \ep)$ are uniformly bounded by a scalar multiple of $\prod_{\mu=1}^{\nu} |t_{\mu}|_{\mu}^{\deg \beta_{\mu} - \deg \eta_{\mu}}$. By the $\mP^{\vk}$ seminorms in Definition \ref{def mP seminorms}, the above expression is in turn
\begin{align*}
    & \lesssim \norm{K}{(s, \ldots, s)} \sup_{\substack{\beta, \eta \in \N^q; \\ |\beta| \leq |\eta| \leq (s, \ldots, s); \\ \deg \beta \geq \deg \eta} }  \prod_{\mu=1}^{\nu} |t_{\mu}|_{\mu}^{ \deg \beta_{\mu} - \deg \eta_{\mu}} |t_{\mu}|_{\mu}^{-Q_{\mu} - \deg \beta_{\mu}}. 
\end{align*}
We observe that for every original multi-index $\ap_{\mu}$ in \eqref{eq pk step 1 gc}, there is a multi-index $\eta_{\mu}$ in the expression above s.t. $\deg \eta_{\mu} \geq \deg \ap_{\mu}$. In addition, recall, that $|t_{\mu}|_{\mu} \geq 1$. We thus obtain the growth condition:
\begin{align*}
    |\p^{\ap_1}_{t_1} \cdots \p^{\ap_{\nu}}_{t_{\nu}} K(t)| \lesssim \prod_{\mu=1}^{\nu} |t_{\mu}|_{\mu}^{-Q_{\mu} -\deg \ap_{\mu}}. 
\end{align*}

In the second case, by scaling considerations, it suffices to bound\footnote{Recall, for $R = (R_1, \ldots R_{\nu})\in (0, \infty)^{\nu}$, we denote $K^{(R)}(t) := R_1^{Q_1} \cdots R^{Q_{\nu}}_{\nu} K(R_1 t_1, \ldots, R_{\nu} t_{\nu})$. } 
\begin{equation*}
    \sup_{\substack{|t_{\mu}|_{\mu} \sim 1; \\ \mu =1, \ldots, \nu}} \sup_{\substack{0 \leq R_{\mu} \leq 1 }} |\p^{\ap_1}_{t_1} \cdots \p^{\ap_{\nu}}_{t_{\nu}} K^{(R)}(t)| \lesssim 1.
\end{equation*}
To do so, by the Sobolev embedding, there exists $s= s(\ap)>0$ s.t.
\begin{align}\label{eq 3}
    |\p^{\ap_1}_{t_1} \cdots \p^{\ap_{\nu}}_{t_{\nu}} K^{(R)}(t)| \lesssim \norm{\bigotimes_{\mu =1}^{\nu} \psi_{\mu} K^{(R)}}{L^2_s(\R^q)},
\end{align}
where we again choose $\psi_{\mu} \in C^{\infty}_c$ so that $\supp \psi_{\mu} \subseteq B^{q_{\mu}}(t_{\mu}, \ep)$, for $\ep>0$ small enough so that $K$ and its derivatives up to order $s$ in each factor space $\R^{q_{\mu}}$ do not change signs and remain bounded away from zero on a $10\ep$-neighborhood of $\supp \psi_1 \otimes \cdots \otimes \psi_{\nu}$. Repeating the procedure as in \eqref{eq pk step 2 gc} through \eqref{eq pk gc equiv top}, we take $\phi_{\mu}, \g_{\mu} \in C^{\infty}_c(\R^{q_{\mu}})$ so that $\supp \phi_{\mu} \subseteq B^{q_{\mu}}(t_{\mu}, \ep)$, $\supp \g_{\mu} \subseteq B^{q_{\mu}}(0, \ep)$, $0 \leq \phi_{\mu}, \g_{\mu} \leq 1$ and so that
\begin{align*}
    \norm{\bigotimes_{\mu =1}^{\nu} \psi_{\mu} K^{(R)}}{L^2_s(\R^q)} &  \lesssim \sum_{|\eta| \leq (s, \ldots, s)} \norm{ \lp \frac{\p}{\p x} \rp^{\eta} \bigotimes_{\mu =1}^{\nu} \phi_{\mu} K^{(R)} * \bigotimes_{\mu =1}^{\nu} \g_{\mu} }{L^2(\R^q)}.
\end{align*}

Notice that $K^{(R)}$ is the Schwartz kernel of the operator $D_{R}\Op(K) D_{R^{-1}}$, where $D_Rf(x_1, \ldots, x_{\nu}) := f(R_1x_1, \ldots, R_{\nu}x_{\nu})$. We can thus rewrite the expression on the right-hand side of the inequality above as
\begin{align*}
    & \sum_{|\eta| \leq (s, \ldots, s)} \norm{ \lp \frac{\p}{\p x} \rp^{\eta} \bigotimes_{\mu =1}^{\nu} \phi_{\mu}  D_R \Op(K) D_{R^{-1}} \bigotimes_{\mu =1}^{\nu} \g_{\mu}  }{L^2(\R^q)}.
\end{align*}
By moving the dilation operators outward, the above equation is
\begin{align*}
    & = \sum_{|\eta| \leq (s, \ldots, s)} \norm{ \lp \frac{\p}{\p x} \rp^{\eta} D_R  \bigotimes_{\mu =1}^{\nu}  \phi_{\mu}  (R_{\mu}^{-1} \cdot )\Op(K)  \bigotimes_{\mu =1}^{\nu} \g_{\mu} (R_{\mu}^{-1} \cdot )  }{L^2(\R^q)}.
\end{align*}
We write the standard vector fields $\lp \frac{\p}{\p x} \rp^{\eta}$ as linear combinations of left-invariant vector fields as in \eqref{eq dx lin comb of X}. The previous Sobolev norm is thus
\begin{align*}
    &\leq \sum_{\substack{ \eta, \beta; \\ |\beta| \leq |\eta| \leq (s, \ldots, s); \\ \deg \beta \geq \deg \eta} } \norm{p_{\eta, \beta} X^{\beta} D_R  \bigotimes_{\mu =1}^{\nu} \phi_{\mu} (R_{\mu}^{-1} \cdot ) \Op(K)  \bigotimes_{\mu =1}^{\nu} \g_{\mu}  (R_{\mu}^{-1} \cdot )  }{L^2(\R^q)}.
\end{align*}
By Hölder's inequality, the equation above is
\begin{align*}
    &\leq \sum_{\substack{ \eta, \beta; \\ |\beta| \leq |\eta| \leq (s, \ldots, s); \\ \deg \beta \geq \deg \eta} }  \norm{p_{\eta, \beta} }{L^{\infty}(B)} \norm{ X^{\beta} D_R  \bigotimes_{\mu =1}^{\nu}  \phi_{\mu}(R_{\mu}^{-1}\cdot) \Op(K)  \bigotimes_{\mu =1}^{\nu} \g_{\mu}(R_{\mu}^{-1}\cdot) }{L^2(\R^q)},
\end{align*}
where $p_{\eta, \beta}$ are homogeneous polynomials restricted to $B$, a product of unit balls in each factor space $\R^{q_{\mu}}$.
By commuting the dilation operator $D_R$ and the \textit{homogeneous} vector fields $X^{\beta}$, the previous equation is
\begin{align*}
    &\lesssim  \sum_{\substack{|\beta| \leq (s, \ldots, s); \\ \deg \beta \geq \deg \eta} } R^{\deg \beta} \norm{  D_R X^{\beta} \Big( \bigotimes_{\mu =1}^{\nu}  \phi_{\mu}(R_{\mu}^{-1}\cdot)  \Op(K)  \bigotimes_{\mu =1}^{\nu} \g_{\mu}(R_{\mu}^{-1}\cdot)  \Big)  }{L^2(\R^q)}.
\end{align*}
By successively applying the Leibniz rule to commute $X^{\beta}$ and $ \phi_{\mu}(R_{\mu}^{-1}\cdot)$ and then by the homogeneity of the \textit{left-invariant} vector field applied to $K * \bigotimes_{\mu =1}^{\nu} \g_{\mu}(R_{\mu}^{-1}\cdot)$, we obtain an additional $R^{-\deg \beta}$ scaling factor. The previous equation is thus
\begin{align*}
    &\leq  \sum_{\substack{|\beta| \leq (s, \ldots, s); \\ \deg \beta \geq \deg \eta} }   \norm{  D_R  \Big( \bigotimes_{\mu =1}^{\nu}  \wt{\phi}_{\mu} (R_{\mu}^{-1}\cdot) \Big( X^{\beta} \Big( K*  \bigotimes_{\mu =1}^{\nu} \g_{\mu} \Big) \Big) D_{R^{-1}}  \Big)   }{L^2(\R^q)},
\end{align*}
where $\wt{\phi}_{\mu} = X^{\beta_{\mu}} \phi_{\mu}$ for some multi-index $\beta_{\mu}$. Furthermore, let $ \phi_{\mu}(R_{\mu}^{-1}\cdot ) \prec \psi_{\mu}$\footnote{Henceforth, $\phi \prec \vp$ will mean that $\vp$ is supported in a ball slightly larger than $\supp \phi$ and that $\vp \equiv 1$ on $\supp \phi$; that is, $\phi \vp = \phi$.} and $\g_{\mu}(R_{\mu}^{-1}\cdot) \prec \vp_{\mu}$, we have $ \supp \psi_{\mu} \subseteq B^{q_{\mu}}(t_{\mu}, R_{\mu}\ep)$ and $\vp_{\mu}\subseteq B^{q_{\mu}}(0_{\mu}, R_{\mu}\ep)$. Note that $R_{\mu}\ep \lesssim 1$ as required. By applying Hölder's inequality, the equation above is thus 
\begin{align*}
    &\leq \norm{D_R}{\mB(L^2)}  \sum_{\substack{|\beta| \leq (s, \ldots, s); \\ \deg \beta \geq \deg \eta} } \norm{\bigotimes_{\mu =1}^{\nu}  \wt{\phi}_{\mu} (R_{\mu}^{-1}\cdot)}{L^{\infty}}  \norm{ \bigotimes_{\mu =1}^{\nu}  \psi_{\mu}  X^{\beta} \Op( K)  \bigotimes_{\mu =1}^{\nu} \vp_{\mu}  }{\mB(L^2(\R^q))} \norm{D_{R^{-1}}}{\mB(L^2)} \norm{\bigotimes_{\mu=1}^{\nu} \g_{\mu}}{L^{\infty}},
\end{align*}
The last equation is bounded above by $\norm{K}{(s, \ldots, s)}$ uniformly in $0 < R_{\mu}\leq 1$.

To complete the proof it remains to show that, given $\vk \in \Z_{\geq 0}^{\nu}$ and $K \in \mP^{\vk}$, $K$ satisfies the cancellation condition in Definition \ref{def pk}. To do so, let $\mB_{\mu} \subseteq C^{\infty}_c(\R^{q_{\mu}})$ be a bounded set, for every $\mu= 1, \ldots, \nu$. By the Sobolev embedding, there exists $s \in \N$ s.t.
\begin{align*}
    \sup_{\substack{\vp_{\mu} \in \mB_{\mu}; \\ R_{\mu}>0 ; \\ \mu= 1,\ldots, \nu}}\lf \int K(t_1, \ldots, t_{\nu}) \prod_{\mu=1}^{\nu} \vp_{\mu} (R_{\mu}t_{\mu})  dt_1 \cdots dt_{\nu}\rf & \lesssim \sup_{\substack{\vp_{\mu} \in \mB_{\mu}; \\ R_{\mu}>0 ; \\ \mu= 1,\ldots, \nu}} \norm{\bigotimes_{\mu =1}^{\nu} \psi_{\mu} K^{(R_1^{-1}, \ldots, R_{\nu}^{-1})}*\bigotimes_{\mu=1}^{\nu} \vp_{\mu} }{L^2_s},
\end{align*}
where $\psi_{\mu} \in C^{\infty}_c(B^{q_{\mu}}(0, 1))$ with $\psi_{\mu}(0) =1$ and $K^{(R_1^{-1}, \ldots, R_{\nu}^{-1})}(t_1, \ldots, t_{\nu}) = R^{-Q} K(R_1^{-1}t_1, \ldots, R_{\nu}^{-1}t_{\nu})$. After writing the standard vector fields as a finite linear combination of \textit{left-invariant} vector fields as in \eqref{eq dx lin comb of X}, the standard Sobolev norm above is
\begin{align*}
    \lesssim \sup_{\substack{\vp_{\mu} \in \mB_{\mu}; \\ R_{\mu}>0 ; \\ \mu= 1,\ldots, \nu}} \sum_{\substack{|\beta| \leq |\eta|\leq (s, \ldots, s); \\ \deg \beta \geq \deg \eta}} \norm{p_{\eta, \beta} \bigotimes_{\mu =1}^{\nu} \wt{\psi}_{\mu} \Big( K^{(R_1^{-1}, \ldots, R_{\nu}^{-1})}* \bigotimes_{\mu=1}^{\nu} X^{\beta}\vp_{\mu} \Big) }{L^2},
\end{align*}
where $\wt{\psi}_{\mu} = X^{\beta_{\mu}} \psi_{\mu}$, for some multi-index $\beta_{\mu} \in \N^{q_{\mu}}$. By Hölder's inequality, the previous $L^2$ norm is
\begin{align*}
    &\lesssim \norm{p_{\eta, \beta}\bigotimes_{\mu =1}^{\nu} \wt{\psi}_{\mu} }{L^{\infty}(B)} \sup_{\substack{\vp_{\mu} \in \mB_{\mu}; \\ R_{\mu}>0 ; \\ \mu= 1,\ldots, \nu}} \sum_{\substack{|\beta| \leq |\eta|; \\ \deg \beta \geq \deg \eta}} \norm{ K^{(R_1^{-1}, \ldots, R_{\nu}^{-1})}* \bigotimes_{\mu=1}^{\nu} X^{\beta}\vp_{\mu} }{L^2},
\end{align*}
where $B$ is a unit ball. Finally, we take out the $L^2$-operator norm and by compactness, we bound the $L^2$ norms of the test functions by their respective $L^{\infty}$ norms. The previous equation is thus
\begin{align*}
    &\lesssim \norm{\Op(K)}{\mB(L^2(\R^q))}  \sup_{\substack{\vp_{\mu} \in \mB_{\mu} ; \\ \mu= 1,\ldots, \nu}} \sum_{\substack{|\beta| \leq |\eta|; \\ \deg \beta \geq \deg \eta}} \norm{\bigotimes_{\mu=1}^{\nu} X^{\beta}\vp_{\mu} }{L^2(\R^q)}.
\end{align*}
Recall that $\mB_{\mu}$ is a bounded set. Hence, the derivatives of all functions $\vp_{\mu} \in \mB_{\mu}$ up to finite order $s$ are uniformly bounded. To summarize, we obtain the following bound:
\begin{align*}
    \sup_{\substack{\vp_{\mu} \in \mB_{\mu}; \\ R_{\mu}>0 ; \\ \mu= 1,\ldots, \nu}}\lf \int K(t_1, \ldots, t_{\nu}) \prod_{\mu=1}^{\nu} \vp_{\mu} (R_{\mu}t_{\mu})  dt_1 \cdots dt_{\nu}\rf \lesssim \norm{\Op(K)}{\mB(L^2(\R^q))}.
\end{align*}
In order words, $\int K(t_1, \ldots, t_{\nu}) \prod_{\mu=1}^{\nu} \vp_{\mu} (R_{\mu}t_{\mu})  dt_1 \cdots dt_{\nu}$ is a family of product kernels on a $0-$factor space; that is, a family of complex numbers, which are uniformly bounded in $R_{\mu}>0$ and $\vp_{\mu} \in \mB_{\mu}$ for $\mu =1, \ldots, \nu$.

For our induction hypothesis, given $n>0$, we assume that for any nonempty subset $S \subsetneq \{1, \ldots, \nu\}$ of size $|S|= n >0$, the distributions
\begin{align*}
    \int K(t_1, \ldots, t_{\nu}) \prod_{\mu \in S} \vp_{\mu} (R_{\mu}t_{\mu}) dt_{\mu}
\end{align*}
form a family of product kernels on $(\nu -n)-$factor spaces with norms uniformly bounded in $R_{\mu} >0$ and $\vp_{\mu} \in \mB_{\mu}$, for every $\mu \in S$. We need to show that, for every subset $S_1$ with $|S_1| =n -1$, the distributions
\begin{align*}
    \int K(t_1, \ldots, t_{\nu}) \prod_{\mu \in S_1} \vp_{\mu} (R_{\mu}t_{\mu}) dt_{\mu}
\end{align*}
form a family of product kernels on $(\nu - n+1 )$-factor spaces. 

To prove the growth condition for product kernels on $(\nu - n+1)$-factor spaces, let $\ap^{S_1^c} \in \times_{\mu' \notin S_1} \N^{q_{\mu'}}$. As before there are two extreme cases to consider separately:
\begin{enumerate}
    \item $|t_{\mu'}|_{\mu'} \geq 1$ for every $\mu' \notin S_1$, and

    \item $|t_{\mu'}|_{\mu'} <1$ for every $\mu' \notin S_1$.
\end{enumerate}
The intermediate cases follow from a few straightforward modifications. We will only outline the proof of the latter case to avoid redundancy. Indeed, the first case follows from a few simpler modifications of the general growth condition outlined above. We thus leave it to the reader. 

It remains to show that 
\begin{align*}
    \sup_{\substack{|t_{\mu'}|_{\mu'} \sim 1; \\ 0 < R_{\mu'} \leq 1; \\ \mu' \notin S_1}}\sup_{\substack{R_{\mu}>0 ; \\ \vp_{\mu} \in \mB_{\mu}; \\ \mu \in S_1 }}\lf \p^{\ap^{S_1^c}} \int K^{(R^{S_1^c})}(t) \prod_{\mu \in S_1} \vp_{\mu} (R_{\mu}t_{\mu}) dt_{\mu} \rf \lesssim 1. 
\end{align*}
To do so, we introduce tests functions on every factor space as follows.
\begin{itemize}
    \item For $\mu \in S_1$, let $\psi_{\mu} \in C^{\infty}_c(B^{q_{\mu}} (0, 1))$, and $\psi_{\mu}(0) =1$.

    \item For $\mu' \notin S_1$, let $\phi_{\mu'}, \g_{\mu'} \in C^{\infty}_c(\R^{q_{\mu'}})$ with $\supp \phi_{\mu'} \subseteq B^{q_{\mu'}} (t_{\mu'}, \ep)$ and $\supp \g_{\mu'} \subseteq B^{q_{\mu'}} (0, \ep)$, where $\ep$ is small enough\footnote{See the remark following \eqref{eq 3} for details on how to choose $\ep>0$. The only modification necessary is to replace $K$ with $\int K^{(R^{S_1^c})}(t) \prod_{\mu \in S_1} \vp_{\mu} (R_{\mu}t_{\mu}) dt_{\mu}$ defined on the factor space $\times_{\mu' \notin S_1} \R^{q_{\mu'}}$.} so that, following the procedure as in \eqref{eq pk step 1 gc} through \eqref{eq pk gc equiv top}, 
\end{itemize}
\begin{align*}
    &\Big| \p^{\ap^{S_1^c}} \int K^{(R^{S_1^c})}(t) \prod_{\mu \in S_1} \vp_{\mu} (R_{\mu}t_{\mu}) dt_{\mu} \Big|\\
    \lesssim&  \norm{ \prod_{\mu' \notin S_1} D_{R_{\mu'}} \bigotimes_{\mu' \notin S_1} \phi_{\mu'}(R_{\mu'}^{-1} \cdot) \bigotimes_{\mu \in S_1} \psi_{\mu} K^{(R^{-1})^{S_1}} *\bigotimes_{\mu' \notin S_1} \g_{\mu'} (R_{\mu'}^{-1} \cdot) \bigotimes_{\mu \in S_1} \vp_{\mu}  }{L^{2}_s}.
\end{align*}
Notice that $K^{(\vR^{-1})^{S_1}}$ is the convolution kernel associated to the operator $\prod_{\mu \in S_1} D_{R_{\mu}^{-1}} \Op(K)\\ \prod_{\mu \in S_1} D_{R_{\mu}}$. Writing the standard vector fields $\lp \frac{\p}{\p x} \rp^{\eta}$ as a finite linear combination of \textit{left-invariant} vector fields, the standard Sobolev norm above is
\begin{align*}
    & \leq \sum_{\substack{|\beta| \leq |\eta| \leq (s, \ldots, s); \\ \deg \beta \geq \deg \eta}}  \norm{p_{\eta, \beta} X^{\beta} \prod_{\mu' \notin S_1} D_{R_{\mu'}} \bigotimes_{\mu' \notin S_1} \phi_{\mu'}(R_{\mu'}^{-1} \cdot) \bigotimes_{\mu \in S_1} \psi_{\mu} K^{(R^{-1})^{S_1}} *\bigotimes_{\mu' \notin S_1} \g_{\mu'} (R_{\mu'}^{-1} \cdot) \bigotimes_{\mu \in S_1} \vp_{\mu}  }{L^{2}}.
\end{align*}
The homogeneous polynomials are restricted to a unit ball and are therefore uniformly bounded. As in the proof of the general growth condition, by commuting the dilations operator $D_{R_{\mu'}}$ and the homogeneous vector fields $X^{\beta}$, the previous equation is
\begin{align*}
    \leq& \sum_{\substack{|\beta| \leq |\eta| \leq (s, \ldots, s); \\ \deg \beta \geq \deg \eta}} \prod_{\mu' \notin S_1} R^{\deg \beta_{\mu'}}_{\mu'} \norm{ \prod_{\mu' \notin S_1} D_{R_{\mu'}}X^{\beta}\Big( \bigotimes_{\mu' \notin S_1} \phi_{\mu'}(R_{\mu'}^{-1} \cdot) \bigotimes_{\mu \in S_1} \psi_{\mu} K^{(R^{-1})^{S_1}} *\bigotimes_{\mu' \notin S_1} \g_{\mu'} (R_{\mu'}^{-1} \cdot) \bigotimes_{\mu \in S_1} \vp_{\mu} \Big) }{L^{2}}.
\end{align*}
By the Leibniz rule and by the homogeneity of the left-invariant vector field $X^{\beta}$, we obtain an additional $R^{-\deg \beta_{\mu'}}_{\mu'}$ scaling factor. The previous equation is thus
\begin{align*}
    \leq& \sum_{\substack{|\beta| \leq |\eta| \leq (s, \ldots, s); \\ \deg \beta \geq \deg \eta}} \\
    &\norm{ \prod_{\mu' \notin S_1} D_{R_{\mu'}}\Big( \bigotimes_{\mu' \notin S_1} \wt{\phi}_{\mu'}(R_{\mu'}^{-1} \cdot) \bigotimes_{\mu \in S_1} \wt{\psi}_{\mu} X^{\beta^{S_1^c}} \Big(K^{(R^{-1})^{S_1}} *\bigotimes_{\mu' \notin S_1} \g_{\mu'} \bigotimes_{\mu \in S_1} X^{\beta_{\mu}} \vp_{\mu} \Big) \Big) \prod_{\mu' \notin S_1} D_{R_{\mu'}^{-1}} }{L^{2}}.
\end{align*}
As in the growth condition above, let $ \phi_{\mu'}(R_{\mu'}^{-1}\cdot ) \prec \psi_{\mu'}$ and $\g_{\mu'}(R_{\mu'}^{-1}\cdot) \prec \vp_{\mu'}$, we have $ \supp \psi_{\mu'} \subseteq B^{q_{\mu}}(t_{\mu'}, R_{\mu'}\ep)$ and $\vp_{\mu'}\subseteq B^{q_{\mu'}}(0_{\mu'}, R_{\mu'}\ep)$. Note that $R_{\mu'}\ep \lesssim 1$ as required. The equation above is thus 
\begin{align*}
    \lesssim & \sum_{\substack{|\beta| \leq |\eta| \leq (s, \ldots, s); \\ \deg \beta \geq \deg \eta}}   \prod_{\mu' \notin S_1} \norm{D_{R_{\mu'}}}{\mB(L^2(\R^{q_{\mu'}}))} \norm{\Big( \bigotimes_{\mu' \notin S_1} \psi_{\mu'}  X^{\beta^{S_1^c}} \Big(K^{(R^{-1})^{S_1}} *\bigotimes_{\mu' \notin S_1} \vp_{\mu'} \Big) \Big)  }{\mB(L^{2})}\\
    &\hspace{1in} \prod_{\mu \in S_1}\norm{ X^{\beta_{\mu}} \vp_{\mu} }{L^2(\R^{q_{\mu}})  } \prod_{\mu' \notin S_1} \norm{D_{R_{\mu'}^{-1}}}{\mB(L^2(\R^{q_{\mu'}}))}.
\end{align*}
The remaining dilations are on $\R^{q_{\mu}}$ where $\mu \in S_1$. The equation above is thus
\begin{align*}
    = & \sum_{\substack{|\beta| \leq |\eta| \leq (s, \ldots, s); \\ \deg \beta \geq \deg \eta}}  \norm{ \bigotimes_{\mu' \notin S_1} \psi_{\mu'}  X^{\beta^{S_1^c}} \Big(K *\bigotimes_{\mu' \notin S_1} \vp_{\mu'} \Big)   }{\mB(L^{2})} \prod_{\mu \in S_1}\norm{ X^{\beta_{\mu}} \vp_{\mu} }{L^2(\R^{q_{\mu}})  }, 
\end{align*}
where we recognize the seminorm $\norm{K}{\vk}$ where $\vk$ is such that $k_{\mu'} =s$ for all $\mu' \notin S_1$ and $k_{\mu}= 0$ for all $\mu \in S_1$. By Definition \ref{def mP seminorms}, the previous $L^2$ norm is
\begin{align*}
    & \lesssim \norm{K}{\vk} \prod_{\mu \in S_1} \norm{X^{\beta_{\mu}} \vp_{\mu}  }{L^{\infty}(\R^{q_{\mu}})},
\end{align*}
where $\mB_{\mu}$ is a bounded set of test functions for every $\mu \in S_1$. Thus concluding the proof of the growth condition on $(\nu - n+1)$-factor spaces. 

Finally, we want to prove that the distribution $\int K(t_1, \ldots, t_{\nu}) \prod_{\mu \in S_1} \vp_{\mu} (R_{\mu}t_{\mu}) dt_{\mu}$ satisfies the cancellation condition for product kernels on $(\nu - n +1)-$factor spaces. By the induction hypothesis, for all $\mu' \notin S_1$, the distribution
\begin{align*}
    &\int K(t_1, \ldots, t_{\nu}) \prod_{\mu \in S_1} \vp_{\mu} (R_{\mu}t_{\mu})  \vp_{\mu'}(R_{\mu'}t_{\mu'}) dt_{\mu} dt_{\mu'} =\int K(t_1, \ldots, t_{\nu}) \prod_{\mu \in S} \vp_{\mu} (R_{\mu}t_{\mu})   dt_{\mu} 
\end{align*}
is a product kernel in a $(\nu -n)$-factor space. Thus concluding the proof of Proposition \ref{prop pk equiv topologies}. 
\end{proof}

\vspace{.1in}

\subsection{Multi-parameter tame algebra estimate for $\mP^{\vk}$} \

We record the tame estimate for product kernels in $\nu \geq 2$ parameters in the next theorem. 

\begin{theorem}\label{thm pk tame general nu}
    Let $\vk \in \Z^{\nu}_{\geq 0}$. Suppose $K, L \in \mP^{\vk}$. Then, we have
    \begin{equation*}
        \begin{split}
            \norm{K*L}{\vk} \lesssim & \norm{\Op(K)}{\mB(L^2)} \norm{L}{\vk}+\norm{K}{\vk}\norm{\Op(L)}{\mB(L^2)} + \sum_{ \substack{S \subseteq \{1, \ldots, \nu\}; \\ S \neq \emptyset} } \norm{K}{\vk^S_0}^{S}\norm{L}{\vk^{S^c}_0}^{S^c} ,
        \end{split}
    \end{equation*}
    where $\vk_0^S$ denotes the vector $\vk$ with all $\mu$th coordinates with $\mu \notin S$ replaced with $0$.\footnote{For example, $\vk_0^{\{1\}} = (0, k_2, \ldots, k_{\nu})$ and $\vk_0^{\{2, \ldots, \nu\}} = (k_1, 0, \ldots, 0 )$.} The implicit constant depends on $\vk \in \Z^{\nu}_{\geq 0}$. 
\end{theorem}

\begin{remark} 
    Recall the $\nu=2$ parameter case recorded in Theorem \ref{thm tame estimate pk}. For $K, L \in \mP^{(k_1, k_2)}$, we have
    \begin{equation}
        \begin{split}
            \norm{K*L}{(k_1, k_2)} \lesssim & \norm{\Op(K)}{\mB(L^2)} \norm{L}{(k_1, k_2)}+ \norm{K}{(k_1, 0)}^{\{1\}}\norm{L}{(0, k_2)}^{\{2\}}\\
            &+ \norm{K}{(0, k_2)}^{\{2\}} \norm{L}{(k_1, 0)}^{\{1\}}+ \norm{K}{(k_1, k_2)}\norm{\Op(L)}{\mB(L^2)}. 
        \end{split}
    \end{equation}
\end{remark}

We only present the proof of the $2$-parameter case in Theorem \ref{thm tame estimate pk} in order to highlight the key ideas in the proof. The general $\nu$-parameter case follows from a few straightforward modifications.

\subsection{Proof of the tame algebra estimate for $\mP^{(k_1, k_2)}$} \

By Remark \ref{remark w=0}, without loss of generality we may assume  that the test functions $\phi_{\mu}$ are centered in a unit ball around the identity in $\R^{q_{\mu}}$ so that $w_{\mu}=0$ for $\mu =1, 2$. By definition \ref{def mP seminorms}, for $\vk=(k_1, k_2) \in \Z_{\geq 0}^2$, we have
\begin{equation}\label{eq pk KL step 1}
    \begin{split}
        \norm{K*L}{(k_1, k_2)}\leq &  \norm{\Op(K)}{\mB(L^2(\R^q))} \norm{\Op(L)}{\mB(L^2(\R^q))} \\
        &+ \sum_{|\ap_1| \leq k_1}  \sup_{\substack{\norm{f}{L^2(\R^q)}  =1;\\
        \norm{g}{L^2(\R^q)} =1 }} \sup_{j_1,l_1 \in \Z }  \sup_{\substack{z_1\in \R^{q_1}; \\ |z_1|_1 \gtrsim  2^{j_1 \vee l_1} }} |\la \phi_1 f, X^{\ap_1} K*L *\g_1 g \ra| \  |z_1|_1^{Q_1+ \deg \ap_1} \\
        &+ \sum_{|\ap_2| \leq k_2}  \sup_{\substack{\norm{f}{L^2(\R^q)}  =1;\\
        \norm{g}{L^2(\R^q)} =1 }} \sup_{j_2,l_2 \in \Z }  \sup_{\substack{z_2\in \R^{q_2}; \\ |z_2|_2 \gtrsim  2^{j_2 \vee l_2} }} |\la \phi_2 f, X^{\ap_2} K*L *\g_2 g \ra| \  |z_2|_2^{Q_2+ \deg \ap_2} \\
        &+\sum_{|\ap| \leq (k_1, k_2)}  \sup_{\substack{\norm{f}{L^2(\R^q)}  =1;\\
        \norm{g}{L^2(\R^q)} =1 }} \sup_{\substack{ (j_1, j_2) \in \Z^2; \\ (l_1, l_2) \in \Z^2 }}  \sup_{\substack{(z_1, z_2) \in \R^q; \\ |z_1|_1 \gtrsim  2^{j_1 \vee l_1}; \\ |z_2|_2 \gtrsim  2^{j_2 \vee l_2} }} |\la \phi_1 \otimes \phi_2 f, X^{\ap} K*L *\g_1 \otimes \g_2 g \ra| \\
        & \hspace{3in} \times |z_1|_1^{Q_1+ \deg \ap_1} |z_2|_2^{Q_2+ \deg \ap_2},
    \end{split}
\end{equation}
where $\supp \phi_{\mu} \subseteq B^{q_{\mu}}(0, 2^{j_{\mu}})$ and $\supp \g_{\mu} \subseteq B^{q_{\mu}} (z_{\mu}, 2^{l_{\mu}})$, for $j_{\mu}, l_{\mu} \in \Z$ and $\mu= 1, 2$.

\subsubsection{Single-parameter tame algebra estimate} \ 

We first prove a single-parameter estimate to bound the second summand on the right-hand side of the equation in \eqref{eq pk KL step 1}, which we record in the following lemma.
\begin{lemma}\label{lemma pk 2nu k1}
Let $k_1 \in \Z_{\geq 0}$. For $K, L \in \mP^{(k_1, 0)}$, we have
    \begin{equation}\label{eq pk kl k1}
        \begin{split}
            \norm{K*L}{(k_1, 0)}^{\{1\}}& \lesssim \norm{K}{(k_1, 0)}^{\{1\}} \norm{\Op(L)}{\mB(L^2(\R^q))} + \norm{\Op(K)}{\mB(L^2(\R^q))} \norm{L}{(k_1, 0)}^{\{1\}}.
        \end{split}
    \end{equation}
\end{lemma}

\begin{proof}[Proof of Lemma \ref{lemma pk 2nu k1}]
By Remark \ref{remark w=0}, we can assume that $\supp \phi_1 \subseteq B^{q_1}(0, 2^{j_1})$. Let $\vp_1 \in C^{\infty}_c(\R^{q_1})$ with $\supp \vp_1 \subseteq B^{q_1}(0, c \cdot 2^{j_1})$, for some $c \geq 1$, so that $\phi_1 \prec \vp_1$. By the triangle inequality, 
    \begin{equation}\label{eq lemma pk 2nu k1 step 1}
        \begin{split}
            & \sum_{|\ap_1| \leq k_1}  \sup_{\substack{\norm{f}{L^2(\R^q)}  =1;\\
        \norm{g}{L^2(\R^q)} =1 }} \sup_{j_1,l_1 \in \Z }  \sup_{\substack{z_1 \in \R^{q_1}; \\ |z_1|_1 \gtrsim  2^{j_1 \vee l_1} }} |\la \phi_1 f, X^{\ap_1} K*L *\g_1 g \ra|   \times |z_1|_1^{Q_1+ \deg \ap_1}\\
            & \leq \sum_{|\ap_1| \leq k_1}   \sup_{\substack{\norm{f}{L^2(\R^q)}  =1;\\
        \norm{g}{L^2(\R^q)} =1 }} \sup_{j_1,l_1 \in \Z }  \sup_{\substack{z_1 \in \R^{q_1}; \\ |z_1|_1 \gtrsim  2^{j_1 \vee l_1} }} |\la \phi_1 f, X^{\ap_1} K*\vp_1 L *\g_1 g \ra| \times |z_1|_1^{Q_1+ \deg \ap_1} \\
            &+ \sum_{|\ap_1| \leq k_1}  \sup_{\substack{\norm{f}{L^2(\R^q)}  =1;\\
        \norm{g}{L^2(\R^q)} =1 }} \sup_{j_1,l_1 \in \Z }  \sup_{\substack{z_1 \in \R^{q_1}; \\ |z_1|_1 \gtrsim  2^{j_1 \vee l_1} }} |\la \phi_1 f, X^{\ap_1} K*(1-\vp_1) L *\g_1 g \ra| \times |z_1|_1^{Q_1+ \deg \ap_1}.
        \end{split}
    \end{equation}
    We bound the first and second term on the right-hand side of the inequality in \eqref{eq lemma pk 2nu k1 step 1} separately. Notice that, in the first term, the product kernel $L$ is localized away from its singularity restricted to $\R^{q_1}$ since $\supp \vp_1 \cap \supp \g_1 = \emptyset$. In addition, by the Schwartz kernel theorem, the \textit{left-invariant} vector fields $X^{\ap_1}$ are given by \textit{right-convolution} with a kernel. In other words, for all $f, g\in C^{\infty}_c(\R^{q_1})$, by associativity of convolution\footnote{See Proposition 3.2.3 in \cite{greenbook} for details.}, we can write
    \begin{align}\label{eq X associativity of convo}
        X^{\ap_1} (f*g) = (f*g)*X^{\ap_1}\de_0=f*(g*X^{\ap_1}\de_0) = f*(X^{\ap_1}g).
    \end{align}
    As such, we begin by applying the \textit{left-invariant} vector fields $X^{\ap_1}$ to right of the leftmost convolution; that is, the first term on the right-hand side of the inequality in \eqref{eq lemma pk 2nu k1 step 1} is thus
    \begin{align*}
        &= \sum_{|\ap_1| \leq k_1}\sup_{\substack{\norm{f}{L^2(\R^q)}  =1;\\         \norm{g}{L^2(\R^q)} =1 }} \sup_{j_1,l_1 \in \Z }  \sup_{\substack{z_1 \in \R^{q_1}; \\ |z_1|_1 \gtrsim  2^{j_1 \vee l_1} }}  |\la \phi_1 f,  K* X^{\ap_1}\vp_1 ( L *\g_1 g) \ra|  |z_1|_1^{Q_1+ \deg \ap_1}.
    \end{align*}
    In view of obtaining a tame algebra estimate, we take adjoints and note that the kernel of the $L^2$-adjoint operator $\Op(K)^*$ is given by $\wt{K}(t) = \overline{K}(t^{-1})$. The previous equation is thus
    \begin{align*}
        & =\sum_{|\ap_1| \leq k_1} \sup_{\substack{\norm{f}{L^2(\R^q)}  =1;\\         \norm{g}{L^2(\R^q)} =1 }} \sup_{j_1,l_1 \in \Z }  \sup_{\substack{z_1 \in \R^{q_1}; \\ |z_1|_1 \gtrsim  2^{j_1 \vee l_1} }}  |\la \wt{K} * \phi_1 f,  X^{\ap_1}\vp_1 L *\g_1 g \ra|  |z_1|_1^{Q_1+ \deg \ap_1}.
    \end{align*}
    By the Cauchy-Schwartz inequality, we can at last separate the two product kernels $K$ and $L$. Indeed, the previous equation is
    \begin{align*}
        & \leq \norm{\Op(K)^*}{\mB(L^2(\R^q))}  \sum_{|\ap_1| \leq k_1}  \sup_{\substack{   \norm{g}{L^2(\R^q)} =1 }} \sup_{j_1,l_1 \in \Z }  \sup_{\substack{z_1 \in \R^{q_1}; \\ |z_1|_1 \gtrsim  2^{j_1 \vee l_1} }}\norm{X^{\ap_1}\vp_1 L *\g_1 g}{L^2(\R^q)}  |z_1|_1^{Q_1+ \deg \ap_1}.
    \end{align*}
    By the Leibniz rule, the equation above is
    \begin{equation}\label{eq use once}
        \leq \norm{\Op(K)}{\mB(L^2(\R^q))}  \sum_{|\ap_1| \leq k_1} \sup_{\substack{ \norm{g}{L^2(\R^q)} =1 }} \sup_{j_1,l_1 \in \Z }  \sup_{\substack{z_1 \in \R^{q_1}; \\ |z_1|_1 \gtrsim  2^{j_1 \vee l_1} }} \norm{\wt{\vp}_1 X^{\ap_1} L *\g_1 g}{L^2(\R^q)}  |z_1|_1^{Q_1+ \deg \ap_1},
    \end{equation}
    where $\wt{\vp}_1 = X^{\beta_1} \vp_1$ for some multi-index $|\beta_1| \leq k_1$. 
    Recall that $\phi_1 \prec \vp_1$, we can thus in turn write $\vp_1 \prec \sum_{i \in \fI} \phi_1^i$ where $\{\phi_1^i; i \in \fI \} \subseteq C^{\infty}_c(\R^{q_1})$ is a finite collection of test functions with $0 \leq \phi_1^i \leq 1$, $\supp \phi_1^i \subseteq B^{q_1}(w_1^i, 2^{j_1})$ and $|w_1^i|_1 \leq b_1 |z_1|_1$ for every $i \in \fI$, where the choice of $b_1 \in (0, 1)$ depends on the sub-Riemannian geometry as described in the following remark.

    \begin{remark}\label{remark b1}
        $|\cdot|_1$ is a homogeneous quasi-norm. As such, there exists $C_1\geq 1$ s.t. 
        \begin{align*}
            |z_1|_1 \leq C_1 |w_1^i z_1^{-1}|_1 + C_1 |w_1^i|_1.
        \end{align*}
        Taking $|w_1^i|_1 \leq b_1 |z_1|_1$, we have
        \begin{equation}\label{eq w and b subriem}
            |z_1|_1 - C_1b_1|z_1|_1  \leq  C_1 |w_1^i z_1^{-1}|_1.
        \end{equation}
        We choose $b_1 \in (0, 1)$ so that $C_1b_1 \leq 1/2$. This choice of $b_1$ ensures that $|w_1^i z_1^{-1}|_1 \sim |z_1|_1$. 
    \end{remark}

     By the triangle inequality followed by Hölder's inequality, \eqref{eq use once} is
    \begin{align*}
        & \lesssim \norm{\Op(K) }{\mB(L^2(\R^q))} \sum_{|\ap_1| \leq k_1} \sup_{\substack{ \norm{g}{L^2(\R^q)} =1 }} \sup_{j_1,l_1 \in \Z }  \sup_{i\in \fI} \sup_{\substack{z_1\in \R^{q_1}; \\ |w_1^iz_1|_1 \gtrsim  2^{j_1 \vee l_1} }} \norm{\wt{\vp}_1}{L^{\infty}(\R^{q_1})} \norm{\phi_1^{i} X^{\ap_1} L *\g_1 g}{L^2(\R^q)} \\
        & \hspace{4in} \times |w_1^iz_1^{-1}|_1^{Q_1+ \deg \ap_1}, 
    \end{align*}
    where $\norm{\wt{\vp}_1}{L^{\infty}} \lesssim 1$ since $\wt{\vp}_1 = X^{\beta_1} \vp_1$, up to some finite order $|\beta_1| \leq k_1$. By duality, we recover $\norm{L}{(k_1, 0)}^{\{1\}}$. To summarize, we bound the first term on the right-hand side of the inequality in \eqref{eq lemma pk 2nu k1 step 1} as follows
    \begin{equation*}
        \begin{split}
            \sum_{|\ap_1| \leq k_1} \sup_{\substack{\norm{f}{L^2(\R^q)}  =1;\\         \norm{g}{L^2(\R^q)} =1 }} \sup_{j_1,l_1 \in \Z }  \sup_{\substack{z_1 \in \R^{q_1}; \\ |z_1|_1 \gtrsim  2^{j_1 \vee l_1} }} |\la \phi_1 f, X^{\ap_1} K*\vp_1 L *\g_1 g \ra|  |z_1|_1^{Q_1+ \deg \ap_1}  \\
            \lesssim \norm{\Op(K)}{\mB(L^2(\R^q))}  \norm{L}{(k_1, 0)}^{ \{1\} }.
        \end{split}
    \end{equation*}
    
To estimate the second term on the right-hand side of the inequality in \eqref{eq lemma pk 2nu k1 step 1}, we decompose $(1-\vp_1)$ into dyadic annuli.  
\begin{equation*}
    1- \vp_1 (x_1) = \sum_{m_1 \geq j_1} \eta_1(2^{-m_1} x_1),
\end{equation*}
where $\supp \eta_1(2^{-m_1}x_1) \subseteq \{x_1 \in \R^{q_1}; |x_1|_1 \sim 2^{m_1}\}$. We further decompose each annulus and write
\begin{equation}\label{eq decomp 1-vp1}
    1- \vp_1 (x_1) =  \sum_{m_1 \geq j_1} \sum_{i \in \fI} \g_1^{ m_1, i},
\end{equation}
where $\# \fI \sim 1$, $\supp \g_1^{m_1, i} \subseteq B^{q_1} (z_1^{m_1, i}, 2^{m_1})$ and $|z_1^{m_1, i}|_1 \sim 2^{m_1}$. By substituting the decomposition \eqref{eq decomp 1-vp1} into the second term on the right-hand side of the inequality in \eqref{eq lemma pk 2nu k1 step 1}, we are left to consider
    \begin{equation}\label{eq pk cite next}
        \begin{split}
            & \sum_{|\ap_1| \leq k_1} \sup_{\substack{\norm{f}{L^2(\R^q)}  =1;\\         \norm{g}{L^2(\R^q)} =1 }} \sup_{j_1,l_1 \in \Z }  \sup_{\substack{z_1 \in \R^{q_1}; \\ |z_1|_1 \gtrsim  2^{j_1 \vee l_1} }} |\la \phi_1 f, X^{\ap_1} K* \sum_{m_1 \geq j_1} \sum_{i \in \fI} \g_1^{ m_1, i} L *\g_1 g \ra|  |z_1|_1^{Q_1+ \deg \ap_1}.
        \end{split}
    \end{equation}
    There are two cases to examine separately: 
    \begin{enumerate}
        \item $2^{j_1} \leq 2^{m_1} \leq b_1|z_1|_1$, and

        \item $b_1|z_1|_1 \leq 2^{m_1}$,
    \end{enumerate}
    where $b_1 \in (0, 1)$ depends on the sub-Riemannian geometry as detailed in Remark \ref{remark b1}. 

\begin{remark}\label{remark pk distances}
    In the first case, the test functions $\g_1^{m, i}$ are supported close to $\phi_1$ so that the distance between $\supp \g_1^{m_1, i}$ and $\supp \g_1$ is comparable to the distance between $\supp \phi_1$ and $\supp \g_1$. Indeed, by the ``almost'' triangle inequality satisfied by the homogeneous norm $|\cdot|_1$, there exists $C_1 \geq 1$ s.t. in case $(1)$ above, 
\begin{align*}
    |z_1|_1 \leq C_1  |z_1^{m_1, i}z_1^{-1}|_1 +  C_1 b_1 |z_1|_1.
\end{align*}
Recall $b_1 \in (0, 1)$ is chosen so that $C_1 b_1 \leq 1/2$. This choice guarantees that
\begin{align*}
    |z_1|_1 \sim  |z_1^{m_1, i} z_1^{-1}|_1.
\end{align*}
    In the second case, the test functions $\g_1^{m, i}$ are supported away from $\phi_1$. As such, the distance between $\supp \g_1^{m_1, i}$ and $\supp \phi_1$ is comparable to $2^{m_1}$. 
\end{remark}

We break up the sum accordingly. By the triangle inequality, \eqref{eq pk cite next} is
\begin{equation}\label{eq nu2 k1 g}
    \begin{split}
        & \leq \sum_{|\ap_1| \leq k_1} \sup_{\substack{\norm{f}{L^2(\R^q)}  =1;\\     \norm{g}{L^2(\R^q)} =1 }} \sup_{j_1,l_1 \in \Z }  \sup_{\substack{z_1 \in \R^{q_1}; \\ |z_1|_1 \gtrsim  2^{j_1 \vee l_1} }} \sum_{\substack{m_1 \geq j_1; \\ 2^{m_1} \leq b_1|z_1|_1}} \sup_{i \in \fI} |\la \phi_1 f, X^{\ap_1} K*  \g_1^{ m_1, i} L *\g_1 g \ra|  |z_1|_1^{Q_1+ \deg \ap_1}\\
        &+ \sum_{|\ap_1| \leq k_1}  \sup_{\substack{\norm{f}{L^2(\R^q)}  =1;\\     \norm{g}{L^2(\R^q)} =1 }} \sup_{j_1,l_1 \in \Z }  \sup_{\substack{z_1 \in \R^{q_1}; \\ |z_1|_1 \gtrsim  2^{j_1 \vee l_1} }} \sum_{\substack{m_1 \geq j_1; \\ 2^{m_1} \geq b_1 |z_1|_1}} \sup_{i \in \fI} |\la \phi_1 f, X^{\ap_1} K*  \g_1^{ m_1, i} L *\g_1 g \ra|  |z_1|_1^{Q_1+ \deg \ap_1}.
    \end{split}
\end{equation}


The second term in \eqref{eq nu2 k1 g} can be dealt with in fewer steps so we estimate it next. By Remark \ref{remark pk distances}, $K$ is localized away from its singularity restricted to $\R^{q_1}$. As such, after taking adjoints of the operators $X^{\ap_1}$ and $\Op(K)$, we apply the Cauchy-Schwartz inequality. The second term in \eqref{eq nu2 k1 g} is thus
\begin{align*}
    & \leq \sum_{|\ap_1| \leq k_1}  \sup_{\substack{\norm{f}{L^2(\R^q)}  =1}} \sup_{j_1,l_1 \in \Z }  \sup_{\substack{z_1 \in \R^{q_1}; \\ |z_1|_1 \gtrsim  2^{j_1 \vee l_1} }} \sum_{\substack{m_1 \geq j_1; \\ 2^{m_1} \geq b_1 |z_1|_1}} \sup_{i \in \fI} \norm{\g_1^{ m_1, i} \wt{K}* X^{\ap_1} \phi_1 f}{L^2(\R^q)} \\
    & \times |z_1|_1^{Q_1+ \deg \ap_1}  \norm{\Op(L)}{\mB(L^2(\R^q)} .
\end{align*}
In view of simultaneously obtaining a convergent sum in $m_1$ and recovering $\norm{K}{(k_1, 0)}^{\{1\}}$, we introduce scaling factors and subsequently apply Hölder's inequality to the sum in $m_1$. The equation above is thus
\begin{align*}
    \leq & \sum_{|\ap_1| \leq k_1}  \sup_{\substack{\norm{f}{L^2(\R^q)}  =1 }} \sup_{j_1,l_1 \in \Z }  \sup_{\substack{z_1 \in \R^{q_1}; \\ |z_1|_1 \gtrsim  2^{j_1 \vee l_1} }} \sup_{\substack{m_1 \geq j_1; \\ 2^{m_1} \geq b_1 |z_1|_1}} \sup_{i \in \fI} \norm{\g_1^{ m_1, i} \wt{K}* X^{\ap_1} \phi_1 f}{L^2(\R^q)} 2^{m_1(Q_1 + \deg \ap_1)} \\
    & \times \sum_{\substack{m_1 \geq j_1; \\ 2^{m_1} \geq b_1 |z_1|_1}}  2^{-m_1(Q_1 + \deg \ap_1)} |z_1|_1^{Q_1+ \deg \ap_1} \norm{\Op(L)}{\mB(L^2(\R^q)} .
\end{align*}
The geometric sum is bounded above by a scalar multiple of its largest term. As such, the expression on the right-hand side of the inequality above is
\begin{align*}
    \lesssim & \sum_{|\ap_1| \leq k_1}  \sup_{\substack{\norm{f}{L^2(\R^q)}  =1 }} \sup_{j_1,l_1 \in \Z }  \sup_{\substack{z_1 \in \R^{q_1}; \\ |z_1|_1 \gtrsim  2^{j_1 \vee l_1} }} \sup_{\substack{m_1 \geq j_1; \\ 2^{m_1} \geq b_1 |z_1|_1}} \sup_{i \in \fI} \norm{\g_1^{ m_1, i} \wt{K}* X^{\ap_1} \phi_1 f}{L^2(\R^q)} 2^{m_1(Q_1 + \deg \ap_1)} \\
    & \times |z_1|_1^{-(Q_1 + \deg \ap_1)} |z_1|_1^{Q_1+ \deg \ap_1} \norm{\Op(L)}{\mB(L^2(\R^q)} .
\end{align*}
Recall that $\supp \g_1^{m_1,i} \subseteq B^{q_1}(z_1^{m_1, i}, 2^{m_1})$ and $\supp \phi_1 \subseteq B^{q_1}(0, 2^{j_1})$, where $|z_1^{m_1, i}|_1 \sim 2^{m_1}$. The equation above is thus
\begin{align*}
    &\approx  \sum_{|\ap_1| \leq k_1}  \sup_{\substack{\norm{f}{L^2(\R^q)}  =1  }} \sup_{j_1, m_1 \in \Z }  \sup_{\substack{z_1^{m_1, i} \in \R^{q_1}; \\ |z_1^{m_1, i}|_1 \gtrsim  2^{j_1 \vee m_1} }}  \sup_{i \in \fI} \norm{\g_1^{ m_1, i} \wt{K}* X^{\ap_1} \phi_1 f}{L^2(\R^q)}  |z_1^{m_1, i}|_1^{(Q_1 + \deg \ap_1)}   \\
    &\hspace{2in} \times\norm{\Op(L)}{\mB(L^2(\R^q))}.
\end{align*}
By duality, and by taking adjoints of the operators $X^{\ap_1}$ and $\Op(K)^*$ once more, the above equation is
\begin{align*}
    &=  \sum_{|\ap_1| \leq k_1}  \sup_{\substack{\norm{f}{L^2(\R^q)}  =1;\\     \norm{g}{L^2(\R^q)} =1 }} \sup_{j_1, m_1 \in \Z }  \sup_{\substack{z_1^{m_1, i} \in \R^{q_1}; \\ |z_1^{m_1, i}|_1 \gtrsim  2^{j_1 \vee m_1} }}  \sup_{i \in \fI} | \la   \phi_1 f, X^{\ap_1} K*\g_1^{ m_1, i} g\ra |  |z_1^{m_1, i}|_1^{Q_1 + \deg \ap_1}  \\
    &\hspace{2in} \times \norm{\Op(L)}{\mB(L^2(\R^q))}.
\end{align*}
We recognize the seminorms at last. Indeed, the expression on the right-hand side of the equation above is
\begin{align*}
    &\lesssim \norm{K}{(k_1, 0)}^{\{1\} } \norm{\Op(L)}{\mB(L^2(\R^q))}.
\end{align*}
We have thus bounded the second term on the right-hand side of the inequality in \eqref{eq nu2 k1 g}.

To complete the proof of Lemma \ref{lemma pk 2nu k1}, it remains to estimate the first term on the right-hand side of the inequality in \eqref{eq nu2 k1 g}. In contrast to the situation above, notice that the distribution $L$ is localized away from its singularity restricted to $\R^{q_1}$. Recalling that the \textit{left-invariant} differential operators $X^{\ap_1}$ commute with \textit{right-invariant} operators $\Op(K)$, we write
\begin{align*}
    &\sum_{|\ap_1| \leq k_1} \sup_{\substack{\norm{f}{L^2(\R^q)}  =1;\\     \norm{g}{L^2(\R^q)} =1 }} \sup_{j_1,l_1 \in \Z }  \sup_{\substack{z_1 \in \R^{q_1}; \\ |z_1|_1 \gtrsim  2^{j_1 \vee l_1} }} \sum_{\substack{m_1 \geq j_1; \\ 2^{m_1} \leq b_1|z_1|_1}} \sup_{i \in \fI} |\la \phi_1 f, X^{\ap_1} ( K*  (\g_1^{ m_1, i} L *\g_1 g)) \ra|  |z_1|_1^{Q_1+ \deg \ap_1}\\
    &=\sum_{|\ap_1| \leq k_1}  \sup_{\substack{\norm{f}{L^2(\R^q)}  =1;\\     \norm{g}{L^2(\R^q)} =1 }} \sup_{j_1,l_1 \in \Z }  \sup_{\substack{z_1 \in \R^{q_1}; \\ |z_1|_1 \gtrsim  2^{j_1 \vee l_1} }} \sum_{\substack{m_1 \geq j_1; \\ 2^{m_1} \leq b_1|z_1|_1}}\sup_{i \in \fI} |\la \phi_1 f,  K*X^{\ap_1}  (\g_1^{ m_1, i} L *\g_1 g) \ra|  |z_1|_1^{Q_1+ \deg \ap_1}. 
\end{align*}
By retracing the steps above, we successively take adjoints of $\Op(K)$, apply the Cauchy-Schwartz inequality, introduce scaling factors, and apply Hölder's inequality on the sum in $m_1$. The equation above is thus
\begin{align*}
    & \lesssim \norm{\Op(K)^*}{\mB(L^2(\R^q))} \sum_{|\ap_1| \leq k_1}  \sup_{\substack{     \norm{g}{L^2(\R^q)} =1 }} \sup_{j_1,l_1 \in \Z }  \sup_{\substack{z_1 \in \R^{q_1}; \\ |z_1|_1 \gtrsim  2^{j_1 \vee l_1} }} \sup_{\substack{m_1 \geq j_1; \\ 2^{m_1} \leq b_1|z_1|_1}} \sup_{i \in \fI} \\
    & \times \norm{X^{\ap_1}  \g_1^{ m_1, i} L *\g_1 g}{L^2(\R^q)}    |z_1^{m_1, i} z_1^{-1}|_1^{Q_1+ \deg \ap_1} |z_1|_1^{Q_1+ \deg \ap_1} \sum_{\substack{m_1 \geq j_1; \\ 2^{m_1} \leq b_1|z_1|_1; \\ |z_1^{m_1, i} z_1^{-1}|_1 \gtrsim |z_1|_1 }} |z_1^{m_1, i} z_1^{-1}|_1^{-(Q_1+ \deg \ap_1)}. 
\end{align*}
After bounding the geometric sum by a scalar multiple of its largest summand, the equation above is
\begin{align*}
    \lesssim & \norm{\Op(K)}{\mB(L^2(\R^q))} \sum_{|\ap_1| \leq k_1}  \sup_{\substack{     \norm{g}{L^2(\R^q)} =1 }} \sup_{j_1,l_1 \in \Z }  \sup_{\substack{z_1 \in \R^{q_1}; \\ |z_1|_1 \gtrsim  2^{j_1 \vee l_1} }} \sup_{\substack{m_1 \geq j_1; \\ 2^{m_1} \leq b_1|z_1|_1}} \sup_{i \in \fI} \\
    & \times \norm{X^{\ap_1}  \g_1^{ m_1, i} L *\g_1 g}{L^2(\R^q)}    |z_1^{m_1, i} z_1^{-1}|_1^{Q_1+ \deg \ap_1} |z_1|_1^{Q_1+ \deg \ap_1}  |z_1|_1^{-(Q_1+ \deg \ap_1)}. 
\end{align*}
After commuting the differential operators and the test functions $\g_1^{m_1, i}$, by duality, we recover the seminorm $\norm{L}{(k_1, 0)}^{\{1\}}$. Thus concluding the proof of Lemma \ref{lemma pk 2nu k1}.
\end{proof}

By symmetry, we estimate the second term in \eqref{eq pk KL step 1} similarly and obtain
\begin{equation}
    \begin{split}
        \norm{K*L}{(0, k_2)} & \lesssim \norm{K}{(0, k_2)}^{ \{2\} } \norm{\Op(L)}{\mB(L^2(\R^q)} + \norm{\Op(K)}{\mB(L^2(\R^q)} \norm{L}{(0, k_2)}^{ \{2\} }.
    \end{split}
\end{equation}

\subsubsection{Multi-parameter tame algebra estimate}  \ 

To complete the proof of the tame algebra estimate in Theorem \ref{thm tame estimate pk}, it remains to estimate the fourth and last term on the right-hand side of the inequality in \eqref{eq pk KL step 1}. We do so in the next lemma. 
\begin{lemma}\label{lemma pk 2nu k1 k2}
For $k_1, k_2 \in \Z_{\geq 0}$, we have
\begin{equation*}
    \begin{split}
        &\norm{K*L}{(k_1, k_2)}^{ \{1, 2\} } \lesssim \norm{\Op(K)}{\mB(L^2(\R^q))} \norm{L}{(k_1, k_2)} +  \norm{K}{(k_1, k_2)}\norm{\Op(L)}{\mB(L^2(\R^q))} \\
        &\hspace{2in} + \norm{K}{(k_1, 0)}^{ \{1\} } \norm{L}{(0, k_2)}^{ \{2\} } +  \norm{K}{(0, k_2)}^{ \{2\} } \norm{L}{(k_1, 0)}^{ \{1\} }.
    \end{split}
\end{equation*}
The implicit constant depends on $(k_1, k_2) \in \Z^2_{\geq 0}$.
\end{lemma}
We first localize further in both factor spaces. Let $\phi_1 \prec \vp_1$ and $\phi_2 \prec \vp_2$. By the triangle inequality, 
\begin{equation}\label{eq pk vp1 vp2}
    \begin{split}
    & \sum_{|\ap| \leq (k_1, k_2)}  \sup_{\substack{\norm{f}{L^2(\R^q)}  =1;\\
        \norm{g}{L^2(\R^q)} =1 }} \sup_{j,l \in \Z^2 }  \sup_{\substack{z\in \R^q; \\ |z_1|_1 \gtrsim  2^{j_1 \vee l_1}; \\ |z_2|_2 \gtrsim  2^{j_2 \vee l_2} }} |\la \phi_1 \otimes \phi_2 f, X^{\ap} K*L *\g_1 \otimes \g_2 g \ra| \\
    & \hspace{2in} \times  |z_1|^{Q_1+ \deg \ap_1}  |z_2|_2^{Q_2+ \deg \ap_2}\\
        &\leq \sum_{|\ap| \leq (k_1, k_2)}  \sup_{\substack{\norm{f}{L^2(\R^q)}  =1;\\
        \norm{g}{L^2(\R^q)} =1 }} \sup_{j,l \in \Z^2 }  \sup_{\substack{z\in \R^q; \\ |z_1|_1 \gtrsim  2^{j_1 \vee l_1}; \\ |z_2|_2 \gtrsim  2^{j_2 \vee l_2} }} |\la \phi_1 \otimes \phi_2 f, X^{\ap} K*\vp_1 \otimes \vp_2 L *\g_1 \otimes \g_2 g \ra| \\
    & \hspace{2in} \times  |z_1|_1^{Q_1+ \deg \ap_1}  |z_2|_2^{Q_2+ \deg \ap_2}\\
    &+ \sum_{|\ap| \leq (k_1, k_2)}  \sup_{\substack{\norm{f}{L^2(\R^q)}  =1;\\
        \norm{g}{L^2(\R^q)} =1 }} \sup_{j,l \in \Z^2 }  \sup_{\substack{z\in \R^q; \\ |z_1|_1 \gtrsim  2^{j_1 \vee l_1}; \\ |z_2|_2 \gtrsim  2^{j_2 \vee l_2} }} |\la \phi_1 \otimes \phi_2 f, X^{\ap} K*\vp_1 \otimes (1-\vp_2) L *\g_1 \otimes \g_2 g \ra| \\
    & \hspace{2in} \times  |z_1|_1^{Q_1+ \deg \ap_1}  |z_2|_2^{Q_2+ \deg \ap_2}\\
    &+ \sum_{|\ap| \leq (k_1, k_2)}  \sup_{\substack{\norm{f}{L^2(\R^q)}  =1;\\
        \norm{g}{L^2(\R^q)} =1 }} \sup_{j,l \in \Z^2 }  \sup_{\substack{z\in \R^q; \\ |z_1|_1 \gtrsim  2^{j_1 \vee l_1}; \\ |z_2|_2 \gtrsim  2^{j_2 \vee l_2} }} |\la \phi_1 \otimes \phi_2 f, X^{\ap} K*(1-\vp_1) \otimes \vp_2 L *\g_1 \otimes \g_2 g \ra| \\
    & \hspace{2in}  \times  |z_1|_1^{Q_1+ \deg \ap_1}  |z_2|_2^{Q_2+ \deg \ap_2}\\
    &+ \sum_{|\ap| \leq (k_1, k_2)}  \sup_{\substack{\norm{f}{L^2(\R^q)}  =1;\\
        \norm{g}{L^2(\R^q)} =1 }} \sup_{j,l \in \Z^2 }  \sup_{\substack{z\in \R^q; \\ |z_1|_1 \gtrsim  2^{j_1 \vee l_1}; \\ |z_2|_2 \gtrsim  2^{j_2 \vee l_2} }} |\la \phi_1 \otimes \phi_2 f, X^{\ap} K*(1-\vp_1) \otimes (1-\vp_2) L *\g_1 \otimes \g_2 g \ra| \\
    & \hspace{2in} \times  |z_1|_1^{Q_1+ \deg \ap_1}  |z_2|_2^{Q_2+ \deg \ap_2}.
    \end{split}
\end{equation}
The estimate in Lemma \ref{lemma pk 2nu k1 k2} follows from the next four lemmas. Each lemma establishes an estimate for each of the four terms on the right-hand side of the inequality in \eqref{eq pk vp1 vp2}. 
\begin{lemma}\label{lemma pk vp1 vp2 term 1}
    Let $(k_1, k_2) \in \Z_{\geq 0}^2$. Suppose $K, L \in \mP^{(k_1, k_2)}$. Then, we have
    \begin{align*}
        \sum_{|\ap| \leq (k_1, k_2)} \sup_{\substack{\norm{f}{L^2}  =1;\\
        \norm{g}{L^2} =1 }} \sup_{j,l \in \Z^2 }  \sup_{\substack{(z_1, z_2) \in \R^q; \\ |z_1|_1 \gtrsim  2^{j_1 \vee l_1}; \\ |z_2|_2 \gtrsim  2^{j_2 \vee l_2} }} |\la \phi_1 \otimes \phi_2 f, X^{\ap} K*\vp_1 \otimes \vp_2 L *\g_1 \otimes \g_2 g \ra| |z_1|_1^{Q_1+ \deg \ap_1}  |z_2|_2^{Q_2+ \deg \ap_2}\\
        \lesssim \norm{\Op(K)}{\mB(L^2(\R^q))} \norm{L}{(k_1, k_2)}. 
    \end{align*}
\end{lemma}

\begin{proof}[Proof of Lemma \ref{lemma pk vp1 vp2 term 1}]
We observe that $L$ is localized away from its singularity on the ``cross'' in $\R^{q_1}\times \R^{q_2}$. It can thus be identified with a smooth function on the whole space. Since the \textit{left-invariant} operators $X^{\ap_1}$ commute with the \textit{right-invariant} operators $\Op(K)$, 
\begin{align*}
    &\sum_{|\ap| \leq (k_1, k_2)}  \sup_{\substack{\norm{f}{L^2(\R^q)}  =1;\\
    \norm{g}{L^2(\R^q)} =1 }} \sup_{j,l \in \Z^2 }  \sup_{\substack{z\in \R^q; \\ |z_1|_1 \gtrsim  2^{j_1 \vee l_1}; \\ |z_2|_2 \gtrsim  2^{j_2 \vee l_2} }} |\la \phi_1 \otimes \phi_2 f, X^{\ap} (K*\vp_1 \otimes \vp_2 L *\g_1 \otimes \g_2 g )\ra| \\
    & \hspace{3in} \times |z_1|_1^{Q_1+ \deg \ap_1}  |z_2|_2^{Q_2+ \deg \ap_2} \\
    &= \sum_{|\ap| \leq (k_1, k_2)}  \sup_{\substack{\norm{f}{L^2(\R^q)}  =1;\\
    \norm{g}{L^2(\R^q)} =1 }} \sup_{j,l \in \Z^2 }  \sup_{\substack{z\in \R^q; \\ |z_1|_1 \gtrsim  2^{j_1 \vee l_1}; \\ |z_2|_2 \gtrsim  2^{j_2 \vee l_2} }} |\la \phi_1 \otimes \phi_2 f,  K*  \wt{\vp}_1 \otimes \wt{\vp}_2 X^{\ap} (L *\g_1 \otimes \g_2 g )\ra| \\
    & \hspace{3in} \times |z_1|_1^{Q_1+ \deg \ap_1}  |z_2|_2^{Q_2+ \deg \ap_2},
\end{align*}
where $\wt{\vp}_{\mu} =X^{\beta_{\mu}} \vp_{\mu}$, for some multi-index $\beta_{\mu} \in \N^{q_{\mu}}$ with $\mu =1, 2$. As in the proof of Lemma \ref{lemma pk 2nu k1}, we successively, take adjoints of $\Op(K)$ and apply the Cauchy-Schwartz inequality. The equation above is thus
\begin{equation}
    \begin{split}
        & \lesssim \norm{\Op(K)}{\mB(L^2(\R^q))} \sum_{|\ap| \leq (k_1, k_2)}  \sup_{\substack{
        \norm{g}{L^2(\R^q)} =1 }} \sup_{j,l \in \Z^2 }  \sup_{\substack{z\in \R^q; \\ |z_1|_1 \gtrsim  2^{j_1 \vee l_1}; \\ |z_2|_2 \gtrsim  2^{j_2 \vee l_2} }}  \norm{\vp_1 \otimes \vp_2 X^{\ap} L *\g_1 \otimes \g_2 g}{L^2(\R^q)}\\
        & \hspace{3in} \times \norm{\wt{\vp}_1 \otimes \wt{\vp}_2}{L^{\infty}}  |z_1|_1^{Q_1+ \deg \ap_1}  |z_2|_2^{Q_2+ \deg \ap_2},
    \end{split}
\end{equation}
where $\norm{\wt{\vp}_1 \otimes \wt{\vp}_2}{L^{\infty}}\lesssim 1 $ as $\wt{\vp}_{\mu} = X^{\beta_{\mu}}\vp_{\mu}$ up to some finite order $|\beta_{\mu}| \leq k_{\mu}$ for $\mu =1, 2$. 

Recall, that $\phi_{\mu} \prec \vp_{\mu}$ for $\mu = 1, 2$. We write, $\vp_{\mu} = \sum_{i_{\mu} \in \fI_{\mu}} \phi_{\mu}^{i_{\mu}}$, for some finite index set $\fI_{\mu}$, with $\supp \phi_{\mu}^{i_{\mu}} \subseteq B^{q_{\mu}} (w_{\mu}^{i_{\mu}}, 2^{j_{\mu}})$ and $|w_{\mu}^{i_{\mu}}|_{\mu} \leq b_{\mu}|z_{\mu}|_{\mu}$ (see Remark \ref{remark b1} for details on the choice of $b_1, b_2 \in (0, 1)$). By our choice of $b_1, b_2 \in (0, 1)$, we can ensure that $|w_1^{i_1}z_1^{-1}|_1 \sim |z_1|_1$ and $|w_2^{i_2}z_2^{-1}|_2 \sim |z_2|_2$. Hence, by the triangle inequality, the equation above is
\begin{equation*}
    \begin{split}
        & \lesssim \norm{\Op(K)}{\mB(L^2(\R^q))} \sum_{|\ap| \leq (k_1, k_2)}   \sup_{\substack{
        \norm{g}{L^2(\R^q)} =1 }} \sup_{j,l \in \Z^2 }  \sup_{\substack{z\in \R^q; \\ |z_1|_1 \gtrsim  2^{j_1 \vee l_1}; \\ |z_2|_2 \gtrsim  2^{j_2 \vee l_2} }} \sup_{\substack{i_1 \in \fI_1; \\ i_2 \in \fI_2}}\norm{\phi_1^{i_1} \otimes \phi_2^{i_{2}} X^{\ap} L *\g_1 \otimes \g_2 g}{L^2(\R^q)}  \\
        &\hspace{2.5in} \times |w_1^{i_1}z_1^{-1}|_1^{Q_1+ \deg \ap_1}  |w_2^{i_2}z_2^{-1}|_2^{Q_2+ \deg \ap_2}.
    \end{split}
\end{equation*}
By duality, we recognize $\norm{L}{(k_1, k_2)}$ and conclude the proof of Lemma \ref{lemma pk vp1 vp2 term 1}. 
\end{proof}

We estimate the second and third terms in \eqref{eq pk vp1 vp2} with the following two lemmas. 

\begin{lemma}\label{lemma vp1 1-vp2 term 2}
    Let $(k_1, k_2) \in \Z_{\geq 0}^2$. Suppose $K, L \in \mP^{(k_1, k_2)}$. Then, we have
        \begin{equation*}
        \begin{split}
           \sum_{|\ap| \leq (k_1, k_2)}  \sup_{\substack{\norm{f}{L^2(\R^q)}  =1;\\
            \norm{g}{L^2(\R^q)} =1 }} \sup_{j,l \in \Z^2 }  \sup_{\substack{z\in \R^q; \\ |z_1|_1 \gtrsim  2^{j_1 \vee l_1}; \\ |z_2|_2 \gtrsim  2^{j_2 \vee l_2} }} |\la \phi_1 \otimes \phi_2 f, X^{\ap} K*\vp_1 \otimes (1-\vp_2) L *\g_1 \otimes \g_2 g \ra|\\
            \times  |z_1|_1^{Q_1+ \deg \ap_1}  |z_2|_2^{Q_2+ \deg \ap_2}\\
            \lesssim \norm{\Op(K)}{\mB(L^2(\R^q))} \norm{L}{(k_1, k_2)}  + \norm{K}{(0, k_2)}^{ \{2\} } \norm{L}{(k_1, 0)}^{ \{1\} }.
        \end{split}
        \end{equation*}
    \end{lemma}

\begin{lemma}\label{lemma term 3}
Let $(k_1, k_2) \in \Z_{\geq 0}^2$. Suppose $K, L \in \mP^{(k_1, k_2)}$. Then, we have
    \begin{align*}
            \sum_{|\ap| \leq (k_1, k_2)}  \sup_{\substack{\norm{f}{L^2(\R^q)}  =1;\\
            \norm{g}{L^2(\R^q)} =1 }} \sup_{j,l \in \Z^2 }  \sup_{\substack{z\in \R^q; \\ |z_1|_1 \gtrsim  2^{j_1 \vee l_1}; \\ |z_2|_2 \gtrsim  2^{j_2 \vee l_2} }} |\la \phi_1 \otimes \phi_2 f, X^{\ap} K*(1-\vp_1) \otimes \vp_2 L *\g_1 \otimes \g_2 g \ra| \\
            \times  |z_1|_1^{Q_1+ \deg \ap_1}  |z_2|_2^{Q_2+ \deg \ap_2}\\
            \lesssim \norm{\Op(K)}{\mB(L^2(\R^q))} \norm{L}{(k_1, k_2)} + \norm{K}{(k_1, 0)}^{ \{1\} } \norm{L}{(0, k_2)}^{ \{2\} }.
\end{align*}
\end{lemma}

By symmetry, the proofs of Lemma \ref{lemma vp1 1-vp2 term 2} and Lemma \ref{lemma term 3} are nearly identical. We will therefore only detail the proof of the estimate appearing in Lemma \ref{lemma vp1 1-vp2 term 2}.

\begin{proof}[Proof of Lemma \ref{lemma vp1 1-vp2 term 2}]
Observe that the product kernel $L$ is localized away from $t_1 =0$. After commuting the \textit{left-invariant} vector fields $X^{\ap_1}$ with the \textit{right-invariant} operator $\Op(K)$, we are left to consider
\begin{align*}
    &\sum_{|\ap| \leq (k_1, k_2)}  \sup_{\substack{\norm{f}{L^2(\R^q)}  =1;\\
    \norm{g}{L^2(\R^q)} =1 }} \sup_{j,l \in \Z^2 }  \sup_{\substack{z\in \R^q; \\ |z_1|_1 \gtrsim  2^{j_1 \vee l_1}; \\ |z_2|_2 \gtrsim  2^{j_2 \vee l_2} }} |\la \phi_1 \otimes \phi_2 f, X^{\ap_2} K* \wt{\vp}_1 \otimes (1-\vp_2) X^{\ap_1}( L *\g_1 \otimes \g_2 g )\ra| \\
    & \hspace{4in}\times  |z_1|_1^{Q_1+ \deg \ap_1}  |z_2|_2^{Q_2+ \deg \ap_2},
\end{align*}
    where $\wt{\vp}_1= X^{\ap_1} \vp$, for some multi-index $\ap_1 \in \N^{q_1}$. 
    As in \eqref{eq decomp 1-vp1}, we decompose $1-\vp_2$ into a sum of compactly supported test functions:
    \begin{equation}\label{eq decomp 1-vp2}
        1-\vp_2 = \sum_{m_2 \geq j_2} \sum_{i_2 \in \fI_2} \g_2^{m_2, i_2},
    \end{equation}
    where $\# \fI_2 \sim 1$, $\supp \g_2^{m_2, i_2} \subseteq B^{q_2}(z_2^{m_2, i_2}, 2^{m_2})$ and $|z_1^{m_2, i_2}|_2 \sim 2^{m_2}$. As in Lemma \ref{lemma pk 2nu k1}, there are two cases to be dealt with separately:
    \begin{enumerate}
        \item $2^{m_2} \leq b_2|z_2|_2$, and

        \item $b_2|z_2|_2  \leq 2^{m_2}$,
    \end{enumerate}
    where $b_2 \in (0, 1)$ is chosen\footnote{We take $b_2$ small enough so that $C_2 b_2 \leq 1/2$, where $C_2 \geq 1$ is the constant appearing in the ``almost'' triangle inequality satisfied by the homogeneous norm $|\cdot|_2$ on $\R^{q_2}$. } as in Remark \ref{remark b1}. In the first case, the distance between $\supp \g_2^{m_2, i_2}$ and $\supp \g_2$ is comparable to the distance between $\supp \phi_2$ and $\supp \g_2$; that is, $|z_2^{m_2, i_2} z_2^{-1}|_2 \sim  |z_2|_2$. In the second case, the distance between $\supp \g_2^{m_2, i_2}$ and $\supp \phi_2$ is either comparable or larger than the distance between $\supp \g_2$ and $\supp \phi_2$. 

    By substituting \eqref{eq decomp 1-vp2} into the previous equation, it suffices to estimate 
\begin{equation}\label{eq 2}
    \begin{split}
        & \sum_{|\ap| \leq (k_1, k_2)}  \sup_{\substack{\norm{f}{L^2}  =1;\\
        \norm{g}{L^2} =1 }} \sup_{j,l \in \Z^2 }  \sup_{\substack{z\in \R^q; \\ |z_1|_1 \gtrsim  2^{j_1 \vee l_1}; \\ |z_2|_2 \gtrsim  2^{j_2 \vee l_2} }} \sum_{\substack{m_2 \geq j_2; \\ 2^{m_2} \leq b_2|z_2|_2}}  \\
        &\hspace{1in} \times \sup_{i_2 \in \fI_2} |\la \phi_1 \otimes \phi_2 f, X^{\ap_2} K* \wt{\vp}_1 \otimes  \g_2^{m_2, i_2} X^{\ap_1} L *\g_1 \otimes \g_2 g \ra|   |z_1|_1^{Q_1+ \deg \ap_1}  |z_2|_2^{Q_2+ \deg \ap_2}\\
        &+\sum_{|\ap| \leq (k_1, k_2)}  \sup_{\substack{\norm{f}{L^2}  =1;\\
        \norm{g}{L^2} =1 }} \sup_{j,l \in \Z^2 }  \sup_{\substack{z\in \R^q; \\ |z_1|_1 \gtrsim  2^{j_1 \vee l_1}; \\ |z_2|_2 \gtrsim  2^{j_2 \vee l_2} }} \sum_{\substack{m_2 \geq j_2; \\ 2^{m_2} \geq b_2|z_2|_2}}  \\
        & \hspace{1in} \times \sup_{i_2 \in \fI_2} |\la \phi_1 \otimes \phi_2 f, X^{\ap_2} K* \wt{\vp}_1 \otimes  \g_2^{m_2, i_2} X^{\ap_1} L *\g_1 \otimes \g_2 g \ra| |z_1|_1^{Q_1+ \deg \ap_1}  |z_2|_2^{Q_2+ \deg \ap_2}. 
        \end{split}
    \end{equation}

To bound the first term in \eqref{eq 2}, notice that $L$ is localized away from its singularity on the ``cross'' and can thus be identified with a smooth function on $\R^{q_1} \times \R^{q_2}$. This case is thus similar to that of Lemma \ref{lemma pk vp1 vp2 term 1}. As in the proof of Lemma \ref{lemma pk vp1 vp2 term 1}, we successively commute the \textit{left-invariant} differential operator $X^{\ap_2}$ and the \textit{right-invariant} operator $\Op(K)$, take the adjoint of $\Op(K)$, apply the Cauchy-Schwartz inequality, introduce appropriate scaling factors, and at last apply Hölder's inequality to the sum over $m_2$ so that the first term in \eqref{eq 2} is
\begin{align*}
    &\lesssim \norm{\Op(K)}{\mB(L^2(\R^q))}\\
    & \times \sum_{|\ap| \leq (k_1, k_2)}  \sup_{\substack{
    \norm{g}{L^2(\R^q)} =1 }} \sup_{j,l \in \Z^2 }  \sup_{\substack{z\in \R^q; \\ |z_1|_1 \gtrsim  2^{j_1 \vee l_1}; \\ |z_2|_2 \gtrsim  2^{j_2 \vee l_2} }} \sup_{\substack{m_2 \geq j_2; \\ 2^{m_2} \leq b_2 |z_2|_2}} \sup_{i_2 \in \fI_2} \norm{ \wt{\vp}_1 \otimes    \wt{\g_2}^{m_2, i_2} X^{\ap_1}X^{\ap_2} L *\g_1 \otimes \g_2 g}{L^2(\R^q)} \\
    &\hspace{1in} \times  |z_1|_1^{Q_1+ \deg \ap_1}  |z_2|_2^{Q_2+ \deg \ap_2}  |z_2^{m_2, i_2} z_2^{-1}|_2^{Q_2+ \deg \ap_2} \sum_{\substack{m_2 \geq j_2; \\ 2^{m_2} \leq b_2 |z_2|_2}}  |z_2^{m_2, i_2} z_2^{-1}|_2^{-(Q_2+ \deg \ap_2)} .
\end{align*}
where $\norm{\wt{\vp}_1}{L^{\infty}}, \norm{\wt{\g}_2^{m_2, i_2}}{L^{\infty}} \lesssim 1$ since $\wt{\vp}_1 =X^{\beta_1} \vp_1$ and $\wt{\g}_2^{m_2, i_2} = X^{\beta_2} \g_2^{m_2, i_2}$ up to some finite order $|\beta_1| \leq k_1$ and $|\beta_2| \leq k_2$. The geometric series is bounded above by a scalar factor of its largest term. As such, the equation above is
\begin{align*}
    &\lesssim \norm{\Op(K)}{\mB(L^2(\R^q))}\\
    & \times \sum_{|\ap| \leq (k_1, k_2)}  \sup_{\substack{
    \norm{g}{L^2(\R^q)} =1 }} \sup_{\substack{(j_1, m_2) \in \Z^2; \\ (l_1, l_2) \in \Z^2 } }  \sup_{\substack{z\in \R^q; \\ |z_1|_1 \gtrsim  2^{j_1 \vee l_1}; \\ |z_2^{m_2, i_2} z_2^{-1}|_2 \gtrsim  2^{m_2 \vee l_2} }} \sup_{i_2 \in \fI_2} \norm{ \vp_1 \otimes    \g_2^{m_2, i_2} X^{\ap_1}X^{\ap_2} L *\g_1 \otimes \g_2 g}{L^2(\R^q)} \\
    &\hspace{1in} \times  |z_1|_1^{Q_1+ \deg \ap_1}   |z_2^{m_2, i_2} z_2^{-1}|_2^{Q_2+ \deg \ap_2},
\end{align*}
By duality, we recover $\norm{L}{(k_1, k_2)}$.

To complete the proof of Lemma \ref{lemma vp1 1-vp2 term 2}, it remains to bound the second term in \eqref{eq 2}. Recall that in this case, $K$ is localized away from $t_2 = 0$ and $L$ is localized away from $t_1 =0$; that is, neither distribution can be identified with a smooth function on the whole space. After taking adjoints of $X^{\ap_2}$ and $\Op(K)$, we apply the Cauchy-Schwartz inequality to separate the two product kernels. The second term in \eqref{eq 2} is thus
\begin{align*}
    \lesssim \sum_{|\ap| \leq (k_1, k_2)}  \sup_{\substack{\norm{f}{L^2(\R^q)}  =1;\\
    \norm{g}{L^2(\R^q)} =1 }} \sup_{j,l \in \Z^2 }  \sup_{\substack{z\in \R^q; \\ |z_1|_1 \gtrsim  2^{j_1 \vee l_1}; \\ |z_2|_2 \gtrsim  2^{j_2 \vee l_2} }} \sum_{\substack{m_2 \geq j_2; \\ 2^{m_2} \geq b_2 |z_2|_2}} \sup_{i_2 \in \fI_2} \norm{\g_2^{m_2, i_2} \wt{K}*  X^{\ap_2} \phi_1 \otimes \phi_2 f}{L^2(\R^q)} \\
    \times  |z_2|_2^{Q_2+ \deg \ap_2} \norm{\wt{\vp}_1   X^{\ap_1} L *\g_1 \otimes \g_2 g}{L^2(\R^q)}  |z_1|_1^{Q_1+ \deg \ap_1}.
\end{align*} 
As in the proof of Lemma \ref{lemma vp1 1-vp2 term 2}, we bound the term associated to the product kernel $L$ by $\norm{L}{(k_1, 0)}^{\{1\}}$. In addition, after introducing relevant scaling factors and applying Hölder's inequality to the sum in $m_2$, the equation above is
\begin{align*}
    \lesssim \sum_{|\ap| \leq (k_1, k_2)}  \sup_{\substack{\norm{f}{L^2(\R^q)}  =1 }} \sup_{j_2 , l_2 \in \Z }  \sup_{\substack{z_2 \in \R^{q_2}; \\  |z_2|_2 \gtrsim  2^{j_2 \vee l_2} }} \sup_{\substack{m_2 \geq j_2; \\ 2^{m_2} \geq b_2 |z_2|_2}} \sup_{i_2 \in \fI_2} \norm{\g_2^{m_2, i_2} \wt{K}*  X^{\ap_2} \phi_1 \otimes \phi_2 f}{L^2(\R^q)} 2^{m_2(Q_2+ \deg \ap_2)} \\
    \times   \sum_{\substack{m_2 \geq j_2; \\ 2^{m_2} \geq b_2 |z_2|_2}}2^{-m_2(Q_2+ \deg \ap_2)} |z_2|_2^{Q_2+ \deg \ap_2} \norm{L}{(k_1, 0)}^{\{1\}}.
\end{align*}
The geometric sum is bounded above by a scalar multiple of its largest term. Moreover, recall that $\supp \g_2^{m_2, i_2} \subseteq B^{q_2} (z_2^{m_2, i_2}, 2^{m_2})$ where $|z_2^{m_2, i_2}|_2 \sim 2^{m_2}$ As such, the equation above is
\begin{align*}
    &\lesssim \sum_{|\ap| \leq (k_1, k_2)}  \sup_{\substack{\norm{f}{L^2(\R^q)}  =1 }} \sup_{j_2 , l_2 \in \Z }  \sup_{\substack{z_2\in \R^{q_2} ;  \\ |z_2|_2 \gtrsim  2^{j_2 \vee l_2} }} \sup_{\substack{m_2 \geq j_2; \\ 2^{m_2} \geq b_2 |z_2|_2}} \sup_{i_2 \in \fI_2} \norm{\g_2^{m_2, i_2} \wt{K}*  X^{\ap_2} \phi_1 \otimes \phi_2 f}{L^2(\R^q)} |z_2^{m_2, i_2}|_2^{Q_2+ \deg \ap_2}  \\
    &\hspace{1in} \times  |z_2|^{-(Q_2+ \deg \ap_2)} |z_2|_2^{Q_2+ \deg \ap_2} \norm{L}{(k_1, 0)}^{\{1\}}. 
\end{align*}
By duality, the previous equation is
\begin{align*}
    &\lesssim \sum_{|\ap| \leq (k_1, k_2)}  \sup_{\substack{\norm{f}{L^2(\R^q)}  =1;\\
    \norm{g}{L^2(\R^q)} =1 }} \sup_{\substack{j_2, m_2 \in \Z } } \sup_{\substack{ z_2^{m_2, i_2} \in \R^{q_2};  \\ |z_2^{m_2, i_2}|_2 \gtrsim  2^{m_2 \vee j_2} }}  \sup_{i_2 \in \fI_2} | \la      \phi_2 f, X^{\ap_2} K* \g_2^{m_2, i_2}g \ra| |z_2^{m_2,, i_2}|_2^{(Q_2+ \deg \ap_2)} \\
    &\hspace{1in} \times  \norm{L}{(k_1, 0)}^{\{1\}},
\end{align*}
where we recognize $\norm{K}{(0, k_2)}^{\{2\}}$. Thus concluding the proof of Lemma \ref{lemma vp1 1-vp2 term 2}. 
\end{proof}

Finally, we estimate the fourth term in \eqref{eq pk vp1 vp2} with the following lemma.
\begin{lemma}\label{lemma pk 1_vp1 1_vp2 term 4}
Let $(k_1, k_2) \in \Z_{\geq 0}^2$. Suppose $K, L \in \mP^{(k_1, k_2)}$. Then, we have
    \begin{align*}
         \sum_{|\ap| \leq (k_1, k_2)}  \sup_{\substack{\norm{f}{L^2(\R^q)}  =1;\\
        \norm{g}{L^2(\R^q)} =1 }} \sup_{j,l \in \Z^2 }  \sup_{\substack{z\in \R^q; \\ |z_1|_1 \gtrsim  2^{j_1 \vee l_1}; \\ |z_2|_2 \gtrsim  2^{j_2 \vee l_2} }} |\la \phi_1 \otimes \phi_2 f, X^{\ap} K*(1-\vp_1) \otimes (1-\vp_2) L *\g_1 \otimes \g_2 g \ra| \\
         \times  |z_1|_1^{Q_1+ \deg \ap_1}  |z_2|_2^{Q_2+ \deg \ap_2}\\
        \lesssim \norm{\Op(K)}{\mB(L^2(\R^q))} \norm{L}{(k_1, k_2)} +\norm{K}{(0,k_2)}^{ \{2\} } \norm{L}{(k_1, 0)}^{ \{1\} } \\ 
        +\norm{K}{(k_1, 0)}^{ \{1\} } \norm{L}{(0, k_2)}^{ \{2\} } +\norm{K}{(k_1, k_2)} \norm{\Op(L)}{\mB(L^2(\R^q))}.
    \end{align*}
\end{lemma}
\begin{proof}[Proof of Lemma \ref{lemma pk 1_vp1 1_vp2 term 4}]
To avoid redundancy, we will hereafter only highlight the modifications necessary to adapt the proof of the previous lemmas. We decompose $(1-\vp_1)$ and $(1-\vp_2)$ as in \eqref{eq decomp 1-vp1} and \eqref{eq decomp 1-vp2} respectively. 
For $\mu = 1, 2$,
\begin{enumerate}
    \item $2^{m_{\mu}} \leq b_{\mu} |z_{\mu}|_{\mu}$, and

    \item $b_{\mu}|z_{\mu}|_{\mu} \leq 2^{m_{\mu}}$,
\end{enumerate}
where $b_1, b_2$ are chosen as in Remark \ref{remark b1}. Breaking up the sum accordingly,
\begin{align*}
    & \sum_{|\ap| \leq (k_1, k_2)}  \sup_{\substack{\norm{f}{L^2(\R^q)}  =1;\\
    \norm{g}{L^2(\R^q)} =1 }} \sup_{j,l \in \Z^2 }  \sup_{\substack{z\in \R^q; \\ |z_1|_1 \gtrsim  2^{j_1 \vee l_1}; \\ |z_2|_2 \gtrsim  2^{j_2 \vee l_2} }} |\la \phi_1 \otimes \phi_2 f, X^{\ap} K*(1-\vp_1) \otimes (1-\vp_2) L *\g_1 \otimes \g_2 g \ra| \\
    & \hspace{3in} \times  |z_1|_1^{Q_1+ \deg \ap_1}  |z_2|_2^{Q_2+ \deg \ap_2}\\
    &\leq \sum_{|\ap| \leq (k_1, k_2)}  \sup_{\substack{\norm{f}{L^2(\R^q)}  =1;\\ \norm{g}{L^2(\R^q)} =1 }} \sup_{j,l \in \Z^2 }  \sup_{\substack{z\in \R^q; \\ |z_1|_1 \gtrsim  2^{j_1 \vee l_1}; \\ |z_2|_2 \gtrsim  2^{j_2 \vee l_2} }}\sum_{\substack{m_1 \geq j_1; \\ 2^{m_1} \leq b_1 |z_1|_1 }}  \sum_{\substack{m_2 \geq j_2; \\ 2^{m_2} \leq b_2 |z_2|_2}}\sup_{(i_1, i_2) \in \fI_1 \times I_2} \\
    &\hspace{1in} \times  |\la \phi_1 \otimes \phi_2 f, X^{\ap} K* \g_1^{ m_1, i_1} \otimes  \g_2^{m_2, i_2} L *\g_1 \otimes \g_2 g \ra|  |z_1|_1^{Q_1+ \deg \ap_1}  |z_2|_2^{Q_2+ \deg \ap_2}\\
    &+\sum_{|\ap| \leq (k_1, k_2)}  \sup_{\substack{\norm{f}{L^2(\R^q)}  =1;\\ \norm{g}{L^2(\R^q)} =1 }} \sup_{j,l \in \Z^2 }  \sup_{\substack{z\in \R^q; \\ |z_1|_1 \gtrsim  2^{j_1 \vee l_1}; \\ |z_2|_2 \gtrsim  2^{j_2 \vee l_2} }}\sum_{\substack{m_1 \geq j_1; \\ 2^{m_1} \leq b_1 |z_1|_1 }}  \sum_{\substack{m_2 \geq j_2; \\ 2^{m_2} \geq b_2 |z_2|_2}}\sup_{(i_1, i_2) \in \fI_1 \times I_2} \\
    &\hspace{1in} \times  |\la \phi_1 \otimes \phi_2 f, X^{\ap} K* \g_1^{ m_1, i_1} \otimes  \g_2^{m_2, i_2} L *\g_1 \otimes \g_2 g \ra|  |z_1|_1^{Q_1+ \deg \ap_1}  |z_2|_2^{Q_2+ \deg \ap_2}\\
    &+ \sum_{|\ap| \leq (k_1, k_2)}  \sup_{\substack{\norm{f}{L^2(\R^q)}  =1;\\ \norm{g}{L^2(\R^q)} =1 }} \sup_{j,l \in \Z^2 }  \sup_{\substack{z\in \R^q; \\ |z_1|_1 \gtrsim  2^{j_1 \vee l_1}; \\ |z_2|_2 \gtrsim  2^{j_2 \vee l_2} }}\sum_{\substack{m_1 \geq j_1; \\ 2^{m_1} \geq  b_1 |z_1|_1 }}  \sum_{\substack{m_2 \geq j_2; \\ 2^{m_2} \leq b_2 |z_2|_2}}\sup_{(i_1, i_2) \in \fI_1 \times I_2} \\
    &\hspace{1in} \times  |\la \phi_1 \otimes \phi_2 f, X^{\ap} K* \g_1^{ m_1, i_1} \otimes  \g_2^{m_2, i_2} L *\g_1 \otimes \g_2 g \ra|  |z_1|_1^{Q_1+ \deg \ap_1}  |z_2|_2^{Q_2+ \deg \ap_2}\\
    &+\sum_{|\ap| \leq (k_1, k_2)}  \sup_{\substack{\norm{f}{L^2(\R^q)}  =1;\\ \norm{g}{L^2(\R^q)} =1 }} \sup_{j,l \in \Z^2 }  \sup_{\substack{z\in \R^q; \\ |z_1|_1 \gtrsim  2^{j_1 \vee l_1}; \\ |z_2|_2 \gtrsim  2^{j_2 \vee l_2} }}\sum_{\substack{m_1 \geq j_1; \\ 2^{m_1} \geq  b_1 |z_1|_1 }}  \sum_{\substack{m_2 \geq j_2; \\ 2^{m_2} \geq b_2 |z_2|_2}}\sup_{(i_1, i_2) \in \fI_1 \times I_2} \\
    &\hspace{1in} \times  |\la \phi_1 \otimes \phi_2 f, X^{\ap} K* \g_1^{ m_1, i_1} \otimes  \g_2^{m_2, i_2} L *\g_1 \otimes \g_2 g \ra|  |z_1|_1^{Q_1+ \deg \ap_1}  |z_2|_2^{Q_2+ \deg \ap_2}.
\end{align*}

Recall the two cases to consider. 
\begin{enumerate}
    \item If $2^{m_1} \leq b_1|z_1|_1$, then the test functions $\g_1^{m_1, i_1}$ are close to $\phi_1$. As such $L$ is localized away from its singularity restricted to $\R^{q_1}$ and $|z_1^{m_1, i_1}z_1^{-1}|_1 \sim |z_1|_1$.

    \item If $2^{m_1} \geq b_1|z_1|_1$, then the test functions $\g_1^{m_1, i_1}$ are far from $\phi_1$. As such, $K$ is localized localized away from its singularity restricted to $\R^{q_1}$ and $|z_1^{m_1, i_1}|_1 \sim 2^{m_1}$. 
\end{enumerate}
By symmetry, an analogous observation holds on the second factor space $\R^{q_2}$. We follow a similar methodology as before. For each sum, we successively:
\begin{itemize}
    \item apply the differential operators $X^{\ap_1}$ and $X^{\ap_2}$ to the product kernel $K$ or $L$ localized away from $t_1 =0$ and $t_2 =0$ respectively;

    \item after taking adjoints of $\Op(K)$, we apply the Cauchy-Schwartz inequality to separate the two product kernels;

    \item in view of simultaneously obtaining a convergent sum in $m_1$ and $m_2$, and recovering seminorms of $K$ and $L$, we introduce relevant scaling factors;

    \item finally, after applying Hölder's inequality to the sums in $m_1$ and $m_2$, we obtain a tame algebra estimate. 
\end{itemize}
The four terms above are thus bounded above by 
\begin{align*}
    &\norm{\Op(K)}{\mB(L^2(\R^q))} \norm{L}{(k_1, k_2)}, \ \norm{K}{(0,k_2)}^{ \{2\} } \norm{L}{(k_1, 0)}^{ \{1\} }, \norm{K}{(k_1, 0)}^{ \{1\} } \norm{L}{(0, k_2)}^{ \{2\} }, \textrm{and }\norm{K}{(k_1, k_2)} \norm{\Op(L)}{\mB(L^2(\R^q))}
\end{align*}
respectively. 
\end{proof}

\vspace{.1in}

\section{Flag kernels}

\subsection{A Fréchet topology on flag kernels} \ 

We first record the definition of flag kernels commonly found in the literature\footnote{See for instance, Definition 2.3 in \cite{NRSW12} and Definition 4.2.1 in \cite{Str14} which can be readily adapted to our ``product setting''. }.

\begin{definition}\label{def FK}
We define the locally convex topological vector space of flag kernels $K$ on $\R^q = \R^{q_1} \times \cdots \times \R^{q_{\nu}}$ recursively on $\nu$. For $\nu =0$, it is defined to be $\C$ with the usual topology. For $\nu >0$, it is the space of all distributions $K \in C^{\infty}_c(\R^q)'$ such that
\begin{enumerate}
    \item (Growth condition) $K$ is $C^{\infty}$ away from $t_1 =0$ and for every multi-index $\ap = (\ap_1, \ldots, \ap_{\nu}) \in \N^{q_1} \times \cdots  \times \N^{q_{\nu}}$, there is a constant $C_{\ap}$ s.t.
\begin{equation}\label{GC FK}
|\p_{t_1}^{\ap_1} \cdots \p_{t_{\nu}}^{\ap_{\nu}} K(t)| \leq C_{\ap} \prod_{\mu=1}^{\nu} \lp |t_1|_1 + \ldots +|t_{\mu}|_{\mu} \rp^{-Q_{\mu} - \deg \ap_{\mu}}.
\end{equation}
We define the least possible $C_{\ap}$ to be a seminorm.

    \item (Cancellation condition) For any $\mu= 1, \ldots, \nu$, any bounded set $\mB_{\mu} \subseteq C^{\infty}_c(\R^{q_{\mu}})$,  $R_{\mu}>0$, and any $\vp_{\mu} \in \mB_{\mu}$, we assume that the distribution
    \begin{align*}
        K_{\vp_{\mu},R_{\mu}} \in C^{\infty}_0(\cdots \times \R^{q_{\mu-1}} \times \R^{q_{\mu+1}} \times \cdots )'
    \end{align*}
    defined by
    \begin{equation}\label{CC2 FK}
    K_{\vp_{\mu}, R_{\mu}} (\ldots, t_{\mu -1},   t_{\mu+1}, \ldots) := \int K(t)\vp_{\mu}(R_{\mu} \cdot  t_{\mu})dt_{\mu}
\end{equation}
is a flag kernel on the $(\nu-1)-$factor space $ \cdots \times \R^{q_{\mu-1}} \times \R^{q_{\mu+1}} \times \cdots $ and for any continuous semi-norm $\norm{\cdot}{}$ on the locally convex space of flag kernels on $ \cdots \times \R^{q_{\mu-1}} \times \R^{q_{\mu+1}} \times \cdots $, we define a seminorm on the space of flag kernels on $\R^{q_1} \times \cdots \times \R^{q_{\nu}}$ by $\norm{K}{}:= \sup_{\substack{\vp_{\mu} \in \mB_{\mu}; \\ R_{\mu}> 0}} \norm{K_{\vp_{\mu},R_{\mu} }}{} $, which we assume to be finite. 
\end{enumerate}
We give the space of flag kernels the coarsest topology such that all the above semi-norms are continuous. 
\end{definition}

\begin{remark}
    \cite{K16} used representation theory in the particular setting of the Heisenberg group to prove an inversion theorem for flag kernels. We instead proved an inversion theorem for flag kernels defined on a direct product of graded Lie groups using \textit{multi-parameter a priori estimates} and various PDE tools in \cite{Stoko23} instead of representation theory. 
\end{remark}

\begin{remark}
Müller, Ricci, and Stein first introduced flag kernels in their study of spectral multipliers on Heisenberg-type groups in \cite{MRS95}. Nagel, Ricci, and Stein later investigated flag kernels on both a direct product of homogeneous nilpotent groups $G=G_1 \times \cdots \times G_{\nu}$ and on a single homogeneous nilpotent group $G$ in \cite{NRS01}. Nagel, Ricci, Stein, and Wainger then proved in \cite{NRSW12} that operators $Tf= f*K$, where $K$ is a flag kernel on a homogeneous nilpotent Lie group $G$, form an algebra under composition and proved the $L^p$-boundedness of such operators for $1< p < \infty$. \cite{Glo10} and \cite{Glo13} proved similar results independently. Other recent results on flags, flag singular integral operators, flag kernels and the associated Hardy spaces and weighted norm inequalities include \cite{Yang_Dachun09}, \cite{DLM10}, \cite{WL_Besov_12}, \cite{LuZhu13_besov}, \cite{SteinYung13_subalgebra_flag}, \cite{Wu_flag_weight_14}, \cite{HLW19}, \cite{Duong_Ji_Ou_Pipher_wick19}, \cite{Glo19_related_flags}, \cite{Han_Li_wick_22_flag_hardy}, and the references therein. 
\end{remark}

\begin{definition}\label{def fk new seminorm}
For $\vk\in \Z_{\geq 0}^{\nu}$, let $\mF^{\vk} (\R^{q_1} \times \cdots \times \R^{q_{\nu}})$ denote the space of distributions $K \in C^{\infty}_c(\R^{q})'$ for which the following seminorms are finite:
\begin{equation}\label{eq new seminorm FK}
    \begin{split}
        \brac{K}{\vk} &:= \norm{K}{\vk} + \sum_{\mu=1}^{\nu}  \brac{K}{(k_{\mu}, \ldots, k_{\nu})}^{ \{\mu\} },
        \end{split}
\end{equation}
where
\begin{equation}
    \begin{split}
         \brac{K}{(k_{\mu}, \ldots, k_{\nu})}^{\{\mu\} }& :=  \sum_{|(\ap_{\mu}, \ldots, \ap_{\nu})| \leq (k_{\mu}, \ldots, k_{\nu})} \sup_{\substack{j_{\mu}, l_{\mu} \in \Z }} \sup_{\substack{w_{\mu}, z_{\mu} \in \R^{q_{\mu}}; \\ |w_{\mu}z_{\mu}^{-1}|_{\mu} \geq 3 C 2^{j_{\mu} \vee l_{\mu}};}} \Big| \Big\la  \phi_{\mu} f, X^{(\ap_{\mu}, \ldots, \ap_{\nu})} K *  \g_{\mu} g  \Big\ra \Big| \\
        & \times \prod_{\mu \leq \mub \leq \nu} |w_{\mu} z_{\mu}^{-1}|_{\mu}^{Q_{\mub} + \deg \ap_{\mub}},
        \end{split}
\end{equation}
where in turn $\phi_{\mu}, \g_{\mu} \in C^{\infty}_c(\R^{q_{\mu}})$ are such that $0\leq \phi_{\mu}, \g_{\mu} \leq 1$, $\supp \phi_{\mu} \subseteq B^{q_{\mu}}(w_{\mu}, 2^{j_{\mu}})$, $\supp \g_{\mu} \subseteq B^{q_{\mu}}(z_{\mu}, 2^{l_{\mu}})$, and where $C \geq 1$ depends only on the quasi-norm \footnote{The constant $C$ depends on the homogeneous quasi-norm. It corresponds to the constant in the ``almost'' triangle inequality $|xy^{-1}|_{\mu} \leq C(|x_{\mu}|_{\mu} + |y_{\mu}|_{\mu})$.}
$|\cdot|_{\mu}$ on $\R^{q_{\mu}}$, for $\mu =1, \ldots, \nu$.

\end{definition}

\begin{remark}
    For example, in the $\nu =2$ parameter case, for every $\vk = (k_1, k_2) \in \Z^2_{\geq 0}$, we have
\begin{align*}
    \brac{K}{(k_1, k_2)} &= \norm{K}{(k_1, k_2)}\\
    & + \sum_{|\ap | \leq (k_1, k_2) } \sup_{j_1, l_1 \in \Z } \sup_{\substack{w_1, z_1 \in R^{q_1}; \\ |w_1z_1^{-1}|_1 \gtrsim 2^{j_1 \vee l_1}}} \lf \la \phi_1 f, X^{(\ap_1, \ap_2)} K * \g_1 g \ra \rf \\
    & \hspace{2in} \times |w_1z_1^{-1}|_1^{Q_1 + \deg \ap_1 } |w_1z_1^{-1}|_1^{Q_2 + \deg \ap_2 } \\
    & + \sum_{|\ap_2| \leq k_2} \sup_{j_2, l_2 \in \Z} \sup_{\substack{w_2, z_2 \in \R^{q_2}; \\ |w_2z_2^{-1}|_2 \gtrsim 2^{j_2 \vee l_2}}} \lf \la \phi_2 f, X^{\ap_2} K * \g_2 g \ra \rf \\
    & \hspace{3in} \times |w_2z_2^{-1}|_2^{Q_2 + \deg \ap_2 }.
\end{align*}
\end{remark}

\begin{remark}\label{remark fks subalg of pks}
    When defined on a direct product of graded Lie groups $G = G_1 \times \cdots \times G_{\nu}$, flag kernels form a subalgebra of product kernels. As such, observe that
    \begin{align*}
        \norm{K}{\vk} \lesssim \brac{K}{\vk}. 
    \end{align*}
    To avoid redundancy, we will henceforth only highlight the necessary modifications needed to adapt the proof of our results on product kernels to the setting of flag kernels.
\end{remark}

\begin{proposition}\label{prop equiv top fk}
    The seminorms defined in Definition \ref{def FK} and in Definition \ref{def fk new seminorm} induce equivalent topologies on the space of flag kernels. 
\end{proposition}

To prove Proposition \ref{prop equiv top fk}, we first need to record a technical lemma.

\begin{lemma}\label{lemma fk L2 bdd}
        Let $K \in C^{\infty}_c(\R^q)'$ be a flag kernel. Suppose $\supp \phi_{\mu} \subseteq B^{q_{\mu}}(w_{\mu}, 2^{j_{\mu}})$, $\supp \g_{\mu}\subseteq B^{q_{\mu}}(z_{\mu}, 2^{l_{\mu}})$, and $j_{\mu}, l_{\mu} \in \Z$, for $\mu = 1, \ldots, \nu$, then the operators
        \begin{align*}
            \Big( \prod_{\mu \leq \mub \leq \nu }  |w_{\mu} z_{\mu}^{-1}|_{\mu}^{Q_{\mub} + \deg \ap_{\mub} } \Big)  \phi_{\mu} \Op(X^{(\ap_{\mu}, \ldots, \ap_{\nu})} K)  \g_{\mu}
        \end{align*}
        are $L^2(\R^q)$-bounded with operator norm uniformly bounded in $w_{\mu}, z_{\mu}\in \R^{q_{\mu}}$, where $|w_{\mu}z_{\mu}^{-1}|_{\mu} \gtrsim 2^{j_{\mu} \vee l_{\mu}}$. 
    \end{lemma}

    \begin{proof}[Proof of Lemma \ref{lemma fk L2 bdd}]
        By Lemma 4.2.24 in \cite{Str14}, there exists\footnote{see Corollary 2.4.4 in \cite{NRS01} for an analogous formulation.} a bounded set $\{\vp_n; n \in \Z^{\nu}\} \subseteq \mS(\R^q)$ s.t. 
        \begin{align*}
            K = \sum_{\substack{n \in \Z^{\nu}; \\ n_1 \geq \cdots \geq n_{\nu}}} \vp_n^{(2^n)},
        \end{align*}
        where $\zeta_n \in \mS_0^{S(n)} (\R^q)$ where $S(n) = \{ \mu ; \mu = \nu \ \textrm{ or }n_{\mu} > n_{\mu +1}\}$\footnote{For example, for $\nu =2$, we have $\{\zeta_n; \  n \in \Z^2, \ n_1 = n_2\}\subseteq \mS_0^{\{2\}}$ is bounded, and on the other hand, $\{\zeta_n; \  n \in \Z^2, \ j_1 >j_2\} \subseteq \mS_0^{\{1,2\}}$ is bounded}. We can thus decompose the operator of interest as follows:
        \begin{align*}
            \sum_{\substack{n \in \Z^{\nu}; \\ n_1 \geq \cdots \geq n_{\nu}}} \Big( \prod_{\mu \leq \mub \leq \nu}  |w_{\mu} z_{\mu}^{-1}|_{\mu}^{Q_{\mub} + \deg \ap_{\mub} } \Big)  \phi_{\mu} \Op(X^{(\ap_{\mu}, \ldots, \ap_{\nu})}  \vp_n^{(2^n)})  \g_{\mu}.
        \end{align*}
        Applying the differential operators and noting that $|w_{\mu} z_{\mu}^{-1}|_{\mu} \sim 2^{N_{\mu}}$ for some $N_{\mu} \geq 0$, the previous equation is
        \begin{align*}
            = \sum_{\substack{n \in \Z^{\nu}; \\ n_1 \geq \cdots \geq n_{\nu}}} \Big( \prod_{\mu \leq \mub \leq \nu} 2^{n_{\mub} \cdot \deg \ap_{\mub}} 2^{N_{\mu}(Q_{\mub} + \deg \ap_{\mub}) } \Big)  \phi_{\mu} \Op(\zeta_n^{(2^n)})  \g_{\mu},
        \end{align*}
        where $\zeta_n = X^{(\ap_{\mu}, \ldots, \ap_{\nu})}\vp_n$ and $\{\zeta_n; n \in \Z^{\nu}\} \subseteq \mS$ is bounded with $\zeta_n \in \mS_0^{S(n)}$. For every $n \in \Z^{\nu}$, with $n_1 \geq \ldots \geq n_{\nu}$, let
\begin{equation}\label{eq fk Tk def}
    \begin{split}
        T_n :=  \Big( \prod_{\mu \leq \mub \leq \nu} 2^{n_{\mub} \cdot \deg \ap_{\mub}} 2^{N_{\mu}(Q_{\mub} + \deg \ap_{\mub}) } \Big)  \phi_{\mu} \Op(\zeta_n^{(2^n)})  \g_{\mu}.
    \end{split}
\end{equation}
As with distributions\footnote{Compare \eqref{eq fk Tk def} to \eqref{eq Tj def} associated to distributions in $\mF^{\vk}$ and $\mP^{\vk}$ respectively.} in $\mP^{\vk}$, to bound the $L^2$ operator norm of $\sum T_n$ uniformly in $N_{\mu}$, we use Cotlar-Stein's lemma. We need to prove the following claim adapted to distributions in $\mF^{\vk}$.

\begin{claim}\label{claim fk TjTk}
        For all $j, k \in \Z^{\nu}$, with $n_1 \geq \cdots \geq n_{\nu}$ and $m_1 \geq \cdots \geq m_{\nu}$, we have
    \begin{equation}
    \begin{split}
        \norm{T_nT_n^*}{\mB(L^1(\R^q))} & \lesssim  \prod_{1 \leq \mu \leq \nu} 2^{-|n_{\mu}-m_{\mu}|}, \\         
        \norm{T_n^*T_m}{\mB(L^{1}(\R^q))} &\lesssim  \prod_{1 \leq \mu \leq \nu} 2^{-|n_{\mu}-m_{\mu}|},\\
        \norm{T_nT_m^*}{\mB(L^{\infty}(\R^q))} & \lesssim  \prod_{1 \leq \mu \leq \nu} 2^{-|n_{\mu}-m_{\mu}|}, \\
        \norm{T_n^*T_m}{\mB(L^{\infty}(\R^q))} &\lesssim  \prod_{1 \leq \mu \leq \nu} 2^{-|n_{\mu}-m_{\mu}|}, 
    \end{split}
    \end{equation}
    where $T_n, T_m$ are defined in \eqref{eq fk Tk def}. 
\end{claim}

The proof of the claim reduces to bounding 
\begin{equation}
    \sup_{y \in \R^q } \int_{\R^q} \lf \int_{\R^q} T_n(x,u) T_m^*(u,y) du \rf dx.
\end{equation}

By \eqref{eq fk Tk def}, we have an explicit formula for the Schwartz kernels. The proof is similar to that of claim \ref{claim TjTk}. To avoid redundancy, we will only highlight the modifications needed to adapt the proof of claim \ref{claim TjTk} to that of claim \ref{claim fk TjTk}. 

\begin{enumerate}
    \item The first modification to be made has to do with the exponential decay. As before, without loss of generality, we may assume that $n_{\mu} \geq m_{\mu}$ for all $\mu =1, \ldots, \nu$. We first need to ``pull derivatives'' out of $T_n$. Since $\{\zeta_n; n \in \Z^{\nu}\} \subseteq \mS_0^{S(n)}$ is bounded, we can only ``pull out derivatives'' in certain variables $x_{\mu}$ where $\mu \in S(n)$:
\begin{align*}
    \zeta_n = \sum_{|\beta^{S(n)}| = 1} X^{\beta^{S(n)}} \zeta_{n, \beta}.
\end{align*}

• For $\mu =\nu$, $\nu \in S(n)$ so we can ``pull out derivatives'' in the corresponding variable and write:
\begin{align*}
    \zeta_n = \sum_{|\ap_{\nu}|= 1} X^{\ap_{\nu}} \zeta_{n, \ap_{\nu}}.
\end{align*}
Recall $n_{\nu} \geq m_{\nu}$. Hence, after integrating by parts, we obtain an exponential decay: $2^{-|n_{\nu} - m_{\nu}|}$. 

• For $\mu > \nu$, there are two cases to consider separately. 
\begin{enumerate}
    \item If $\mu \in S(n)$, then $n_{\mu} > n_{\mu+1}$, after pulling out derivatives and integrating by parts, we obtain an exponential decay $2^{-|n_{\mu} - m_{\mu}|}$.

    \item If $\mu \notin S(n)$, or equivalently, $n_{\mu} = n_{\mu+1}$, without loss of a generality we can assume that $\mu+1 \in S(n)$. After pulling out derivatives of order $2$ in $\R^{q_{\mu +1}}$ and integrating by parts, we obtain an exponential decay of the form $2^{-2|n_{\mu+1} - m_{\mu+1}|}$. To obtain a decay in terms of $n_{\mu}$ and $m_{\mu}$, we further have two cases to consider here:
    \begin{enumerate}
        \item either $\mu \notin S(m)$ in which case, $m_{\mu} = m_{\mu+1}$ and we have
\begin{align*}
    & 2^{-2|n_{\mu+1} - m_{\mu+1}|} = 2^{-|n_{\mu+1} - m_{\mu+1}|} 2^{-|n_{\mu} - m_{\mu}|},
\end{align*}

        \item or $\mu \in S(m)$ in which case $n_{\mu} = n_{\mu+1} \geq m_{\mu}> m_{\mu+1}$ and we have
\begin{align*}
    & 2^{-2|n_{\mu+1} - m_{\mu+1}|} = 2^{-|n_{\mu} - m_{\mu}|} 2^{-|m_{\mu} - m_{\mu+1}|} 2^{-|n_{\mu+1} - m_{\mu+1}|}.
\end{align*}
    \end{enumerate}
\end{enumerate}

We thus obtain the following summable decay: 
\begin{align*}
    2^{-C(n, m)} := \prod_{\mu \in S(n)} 2^{-|n_{\mu}- m_{\mu}|} \prod_{\substack{\mu \notin S(n); \\ \mu \notin S(m)}} 2^{-|n_{\mu} - m_{\mu}|} \prod_{\substack{\mu \notin S(n); \\ \mu \in S(m)}} 2^{-|n_{\mu}- m_{\mu}|} 2^{-|m_{\mu}- m_{\mu+1}|}. 
\end{align*}

    \item The second modification has to do with the sum of dyadic factors in the equation below. Indeed, \eqref{eq pk decay unif} is replaced by
\begin{equation}
    \begin{split}
        & 2^{-C(n, m)} \sup_{\substack{y_{\mu} \in \supp \phi_{\mu}}} \prod_{\mu \leq \mub \leq \nu} 2^{(n_{\mub}+m_{\mub}) \deg \ap_{\mub}} 2^{2N_{\mu}(Q_{\mub} + \deg \ap_{\mub}) }   \int_{\R^q} \int_{\R^q}  \phi_{\mu}(x_{\mu}) \wt{\g}_{\mu} (u_{\mu})^2 \\
        &\times \Big(  \prod_{\mu' =1}^{\nu} 2^{n_{\mu'} Q_{\mu'}} (1+2^{n_{\mu'}}|x_{\mu'}u_{\mu'}^{-1}|_{\mu'})^{-t_{\mu'}} 2^{m_{\mu'}  Q_{\mu'}} (1+2^{m_{\mu'}} |y_{\mu'}u_{\mu'}^{-1}|_{\mu'})^{-t_{\mu'}} \Big) du dx.
    \end{split}
\end{equation}
Observe that $|x_{\mu}u_{\mu}^{-1}|_{\mu}, |y_{\mu}u_{\mu}^{-1}|_{\mu} \sim 2^{N_{\mu}}$ for some $N_{\mu}$. We bound the $L^1(\R^{q_{\mu}})$ norm by its $L^{\infty}(\R^{q_{\mu}})$ norm multiplied by the size of the domain of integration. All other $L^1$ norms are uniformly bounded. Hence, the previous equation is
\begin{equation}\label{eq fk powers of 2}
\begin{split}
    & \lesssim  2^{-C(n, m)}  \prod_{\mu \leq \mub \leq \nu} 2^{(n_{\mub}+m_{\mub}) \deg \ap_{\mub}} 2^{2N_{\mu}(Q_{\mub} + \deg \ap_{\mub}) }   \\
    &\times  2^{n_{\mu} Q_{\mu}} (1+2^{n_{\mu}} 2^{N_{\mu}})^{-t_{\mu}} 2^{m_{\mu}  Q_{\mu}} (1+2^{m_{\mu}} 2^{N_{\mu}} )^{-t_{\mu}} 2^{l_{\mu} Q_{\mu}} 2^{j_{\mu} Q_{\mu}}. 
    \end{split}
\end{equation}
Observe that $N_{\mu} \gtrsim j_{\mu} \vee l_{\mu}$. As such, the previous equation is
\begin{align*}
    & \lesssim  2^{-C(n, m)}  \prod_{\mu \leq \mub \leq \nu} 2^{(n_{\mub}+m_{\mub}) \deg \ap_{\mub}}  2^{2N_{\mu} (Q_{\mub} + \deg \ap_{\mub}) }   \\
    &\times  2^{n_{\mu} Q_{\mu}} (1+2^{n_{\mu}} 2^{N_{\mu}})^{-t_{\mu}} 2^{m_{\mu}  Q_{\mu}} (1+2^{m_{\mu}} 2^{N_{\mu}} )^{-t_{\mu}}  2^{2N_{\mu} Q_{\mu}}. 
\end{align*}
It remains to show that for our given choices of $n_{\mu}, n_{\mub}, m_{\mu}$ and $m_{\mub}$ the expression above is uniformly bounded in $N_{\mu}$ for $\mu \leq \mub \leq \nu$. By taking a large enough $t_{\mu}$, we obtain a exponential decay $2^{-C(n, m)}$ as desired. 
    \end{enumerate}
    
\end{proof}

\begin{proof}[Proof of Proposition \ref{prop equiv top fk}]
    Let $K\in C^{\infty}_c(\R^q)'$ be a flag kernel as defined in Definition \ref{def FK}. We first want to show that $K \in \mF^{\vk}$ for all $\vk \in \Z_{\geq 0}^{\nu}$. Given $\vk \in \Z_{\geq 0}^{\nu}$, by Remark \ref{remark fks subalg of pks}, $\norm{K}{\vk} < \infty$. By Lemma \ref{lemma fk L2 bdd}, $\brac{K}{\vk}$ is in fact bounded for all $\vk \in \Z^{\nu}_{\geq 0}$.

For the reverse direction, let $K \in \mF^{\vk}$ for all $\vk \in \Z_{\geq 0}^{\nu}$. We want to show that $K$ is a flag kernel. It suffices to prove the growth condition \eqref{GC FK} on the complement of the set where $|t_1|_1 \lesssim \cdots \lesssim |t_{\nu}|_{\nu} $. Indeed in this case, the growth condition for flag kernels is comparable to that of product kernels and by Proposition \ref{prop pk equiv topologies}, the growth condition follows. It suffices to consider the extremal case $|t_1|_1 \gtrsim \ldots \gtrsim |t_{\nu}|_{\nu}$. The remaining cases follow from a few straightforward modifications. As with product kernels, there are two further sub-cases to consider:
\begin{enumerate}
    \item $|t_1|_1 \geq 1$, and

    \item $|t_1|_1 <1$.
\end{enumerate}
The first case follows by localizing exclusively in the first factor space $\R^{q_1}$, retracing the proof for product kernels starting at \eqref{eq pk step 1 gc} and replacing the seminorms $\norm{K}{\vk}$ with the simpler seminorms $\brac{K}{\vk}^{\{1\}}$. 

In the second case, it thus suffices to prove that for $|x_1|_1 \sim 1$ and $|x_1|_1 \gtrsim \cdots \gtrsim |x_{\nu}|_{\nu}$, we have
\begin{align*}
    \sup_{1 \geq R_1 >0} |\p^{\ap_1}_{x_1} \cdots \p^{\ap_{\nu}}_{x_{\nu}}R_1^{Q_1} \cdots R_1^{Q_{\nu}} K(R_1 x_1, \ldots, R_1 x_{\nu})| \lesssim 1. 
\end{align*}
Indeed, it follows that 
\begin{align*}
    R_1^{Q_1 + \deg \ap_1} \cdots (R_1+ R_1|x_2|_2 + \cdots + R_1 |x_{\nu}|_{\nu})^{Q_{\nu}+ \deg \ap_{\nu}} |\p^{\ap_1}_{x_1} \cdots \p^{\ap_{\nu}}_{x_{\nu}}K(R_1 x_1, \ldots, R_1 t_{\nu})| \lesssim 1.
\end{align*}
In turn, we can conclude that for $1 \geq |t_1|_1 \gtrsim \cdots \gtrsim |t_{\nu}|_{\nu}$, 
\begin{align*}
    |\p^{\ap_1}_{t_1} \cdots \p^{\ap_{\nu}}_{t_{\nu}} K(t_1, \ldots, t_{\nu})| \lesssim \prod_{\mu=1}^{\nu}(|t_1|_1 + |t_2|_2 + \cdots +|t_{\mu}|_{\mu})^{-Q_{\mu} -\deg \ap_{\mu}}.
\end{align*}
The proof follows by retracing the proof for product kernels starting at \eqref{eq 3} where we instead consider dilations $R_1 = R_2 = \ldots = R_{\nu}$. Here too, one needs to replace the seminorms $\norm{K}{\vk}$ with the simpler seminorms $\brac{K}{\vk}^{\{1\}}$. We thus refer the reader to the prior section for details. The cancellation condition also follows directly from the product kernel case with the reductions listed above. \end{proof}

\subsection{Multi-parameter tame algebra estimate for $\mF^{\vk}$} \

We record the general $\nu$-parameter tame estimate for flag kernels, for $\nu \geq 2$, in the next theorem.

\begin{theorem}[Tame estimate for flag kernels]\label{thm general nu fk tame}
Let $K, L \in \mF^{\vk}$, we then have
    \begin{equation*}
        \begin{split}
        \brac{K*L}{\vk} \lesssim & \norm{\Op(K)}{\mB(L^2(\R^q))} \brac{L}{\vk} + \brac{K}{\vk} \norm{\Op(L)}{\mB(L^2(\R^q))} + \sum_{ \substack{S \subseteq \{1, \ldots, \nu\}; \\ S \neq \emptyset} } \norm{K}{\vk^S_0}^{S}\norm{L}{\vk^{S^c}_0}^{S^c} .
        \end{split}
    \end{equation*}
\end{theorem}

Recall the $2$-parameter version recorded in Theorem \ref{thm tame FK} in the introduction. For all $(k_1, k_2) \in \N^{2}$ and $K, L \in \mF^{(k_1, k_2)}$, we have
    \begin{equation*}
        \begin{split}
        \brac{K*L}{\vk} \lesssim & \norm{\Op(K)}{\mB(L^2(\R^q))} \brac{L}{(k_1, k_2)} + \norm{K}{(k_1, 0)}^{\{1\} } \norm{L}{(0, k_2)}^{\{2\} }\\
        &+ \norm{K}{(0, k_2)}^{\{2\} } \norm{L}{(k_1, 0)}^{\{1\} } + \brac{K}{(k_1, k_2)} \norm{\Op(L)}{\mB(L^2(\R^q))}.
        \end{split}
    \end{equation*}

As with product kernels, to highlight the key ideas in the proof of the tame estimate, we will only prove the $2$-parameter case. The general $\nu$-parameter case follows from a few straightforward modifications.

\subsection{Proof of the tame algebra estimate for $\mF^{(k_1, k_2)}$} \

Let $\vk = (k_1, k_2) \in \Z_{\geq 0}^2$ and $K, L \in \mF^{\vk}$. By Definition \ref{def fk new seminorm} and Theorem \ref{thm tame estimate pk}, it suffices to bound the following terms in the definition of $\brac{K*L}{(k_1, k_2)}$:
\begin{equation}\label{eq fk tame kl step 1}
\begin{split}
    &\brac{K*L}{(k_1, k_2)}^{\{1\} } +\brac{K*L}{k_2}^{\{2\} }.
    \end{split}
\end{equation}

Observe that $\norm{K*L}{(0, k_2)}^{\{2\}} = \brac{K*L}{(0, k_2)}^{\{2\}}$. As such, by Lemma \ref{lemma pk 2nu k1}, we can bound the second summand above by
\begin{align*}
    \brac{K*L}{k_2}^{\{2\}} \lesssim \norm{\Op(K)}{\mB(L^2(\R^q))} \brac{L}{(0,k_2)}^{\{2\} } + \brac{K}{(0,k_2)}^{\{2\} } \norm{\Op(L)}{\mB(L^2(\R^q))}. 
\end{align*}

To bound the remaining term on the right-hand side of the inequality in \eqref{eq fk tame kl step 1}, it suffices to prove the estimate in the following lemma\footnote{Observe that $\norm{K}{(k_1, 0)}^{\{1\}} \neq \brac{K}{(k_1, 0)}^{\{1\}} $.}. 

\begin{lemma}\label{lemma fk brac 1}
For $K, L\in \mF^{(k_1, k_2)}$, we have
    \begin{align*}
        \brac{K*L}{(k_1, k_2)}^{\{1\} } \lesssim \brac{K}{(k_1, k_2)}^{\{1\} } \norm{\Op(L)}{\mB(L^2(\R^q))}+\norm{\Op(K)}{\mB(L^2(\R^q))} \brac{L}{(k_1, k_2)}^{\{1\} }.
    \end{align*}
\end{lemma}

\begin{proof}[Proof of Lemma \ref{lemma fk brac 1}]
Flag kernels have a singularity at $t_1 = 0$ (instead of the ``cross'' as with product kernels). As such, it suffices to localize the distributions $K$ and $L$ in the first factor space $\R^{q_1}$ to obtain a smooth function on the whole space $\R^{q_1} \times \R^{q_2}$. The proof is otherwise nearly identical to that of Lemma \ref{lemma pk 2nu k1} so we leave it to the reader. 
\end{proof}

\vspace{.2in}

\section{New proof of an inversion theorem for product kernels and flag kernels}

We present a new Banach-algebraic proof of the following inversion theorem. 

\begin{theorem}\label{theorem inverse}[Theorem 1.1 in \cite{Stoko23}]
    Let $\Op(K)$ be a right-invariant singular integral operator $\Op(K)f = K*f$, where $K$ is a product kernel (respectively a flag kernel) on a direct product of graded Lie groups $G= G_1 \times \cdots \times G_{\nu}$. If $\Op(K)$ is invertible in $L^2(G)$, then its inverse is of the form $\Op(L)(g) = L*g$, where $L$ is also a product kernel (respectively a flag kernel). 
\end{theorem}

\begin{remark}
There are several well-known sub-algebras of the space of $L^2$-bounded operators which are known to be inverse-closed. Beals notably proved in \cite{beals1977characterization} that a certain class of $0$-order pseudodifferential operators is inverse-closed by first characterizing such operators in terms of the $L^2$-boundedness of certain commutators. Cordes defined a topology on another algebra of $0$-order pseudodifferential operators in terms of various $L^2$-operator norms (see \cite{cordes1985_charac}). Payne made use of the latter topology defined by Cordes to prove that the subalgebra of inverse-closed operators satisfies \textit{smooth tame structures} (see also the ``Beals-Cordes-type'' characterization of pseudodifferential operators in \cite{taylor1997beals_cordes}, and the work on ``noncommutative Wiener algebras'' in \cite{Sjo94}, \cite{Balan08_noncomm_wie}, \cite{Gro06_time_freq}, and the references therein). 
\end{remark}

\begin{remark}
    Theorem \ref{theorem inverse} in \cite{Stoko23} extends an inversion theorem for single-parameter, homogeneous kernels by Christ and Geller in \cite{CG84} to the multi-parameter setting of ``almost homogeneous'' product kernels and flag kernels. 
\end{remark}

\begin{remark}
    In the original proof of Theorem \ref{theorem inverse}, the author constructed a multi-parameter \textit{a priori estimate} and applied various other tools from PDEs. Our new proof relies on Banach-algebraic methods and extends the single-parameter estimates by Christ, Geller, G{\l}owacki, and Polin in \cite{CGGP92} to the multi-parameter setting. 
\end{remark}

We again consider the $2$-parameter case as the general $\nu$-parameter case follows directly after a few straightforward modifications. 

\subsection{Inversion theorem for $\mP^{(k_1, k_2)}$ and $\mF^{(k_1, k_2)}$} \ 

Theorem \ref{theorem inverse} follows from the following stronger result. 

\begin{theorem}\label{thm inv pk}
    Let $\Op(K) f = K*f$, where $K \in \mP^{(k_1, k_2)}$ (respectively $\mF^{(k_1, k_2)}$). If $\Op(K)$ is $L^2$ invertible, then $\Op(K)^{-1} = \Op(L)$, where $L \in \mP^{(k_1, k_2)}$ (respectively $\mF^{(k_1, k_2)}$). 
\end{theorem}

\begin{proof}[Proof of Theorem \ref{thm inv pk}]
The proofs of the inversion theorem for operators with kernels  in $\mP^{(k_1, k_2)}$ and $\mF^{(k_1, k_2)}$ are nearly identical so we only present the proof for kernels in $\mP^{(k_1, k_2)}$.

Let $\Op(K)$ be invertible in $L^2(\R^{q_1} \times \R^{q_2})$, where $K \in \mP^{(k_1, k_2)}$. We want to prove that $\Op(K)^{-1} = \Op(L)$, where $L \in \mP^{(k_1, k_2)}$. 

Following the proof of the single-parameter inverse theorem in \cite{CGGP92}, it suffices to invert the self-adjoint operator $\ep \Op(K)^* \Op(K),$ for some $\ep > 0$. We can choose $\ep>0$ so that, by the spectral theorem, 
\begin{align*}
    \norm{I-\ep \Op(K)^* \Op(K)}{\mB(L^2(\R^q))} <1.
\end{align*}
Setting $S:= I - \ep \Op(K)^* \Op(K)$, we consider the Neumann series
\begin{equation}\label{eq pk neumann series}
    (\ep \Op(K)^* \Op(K))^{-1} = \sum_{n =0}^{\infty} S^n.
\end{equation}
By construction, 
\begin{align*}
    \norm{(\ep \Op(K)^* \Op(K))^{-1}}{\mB(L^2(\R^q)} \lesssim 1. 
\end{align*}
It thus remains to show that for every $(k_1, k_2) \in \N^2$,
\begin{align*}
    \norm{(\ep \Op(K)^* \Op(K))^{-1}}{(k_1, k_2)} \lesssim 1. 
\end{align*}
To do so, we first need to show that\footnote{We cannot apply \textit{Brandenburg's trick} as in \cite{CGGP92} because the $L^2(\R^q)$-operator norm $\norm{\Op(K)}{\mB(L^2(\R^{q_1} \times \R^{q_2}))}$ does not yet appear in every summand on the right-hand side of the tame estimate (see \cite{Brandenburg75}). }
\begin{align*}
    \lim_{n \rightarrow \infty} \norm{S^n}{(k_1, k_2)}^{1/n} <1. 
\end{align*}
By a first application of the tame estimate in Theorem \ref{thm tame estimate pk}, we obtain
\begin{align*}
    \norm{S^n}{(k_1, k_2)} \lesssim &\norm{S}{\mB(L^2(\R^q))}^{n-1} \norm{S}{(k_1, k_2)}+ \norm{S^{n-1}}{(k_1, 0)}^{ \{1\} }  \norm{S}{(0, k_2)}^{ \{2\} } \\
    &+\norm{S^{n-1}}{(0, k_2)}^{ \{2\} } \norm{S}{(k_1, 0)}^{ \{1\} } + \norm{S^{n-1}}{(k_1, k_2)} \norm{S}{\mB(L^2(\R^q))}. 
\end{align*}
After a second application of Theorem \ref{thm tame estimate pk} to $\norm{S^{n-1}}{(k_1, k_2)}$ and an application of the single-parameter tame estimate in Lemma \ref{lemma pk 2nu k1} to $\norm{S^{n-1}}{(k_1, 0)}^{ \{1\} }$ and $\norm{S^{n-1}}{(0, k_2)}^{ \{2\} }$, we obtain
\begin{align*}
    \norm{S^n}{(k_1, k_2)} \lesssim & 2 \norm{S}{\mB(L^2(\R^q))}^{n-1}   \norm{S}{(k_1, k_2)} + 2 \norm{S}{\mB(L^2(\R^q))}^{n-2} \norm{S}{(k_1, 0)}^{ \{1\} } \norm{S}{(0, k_2)}^{ \{2\} }  \\
    &+ 2 \norm{S^{n-2}}{(k_1, 0)}^{ \{1\} } \norm{S}{\mB(L^2(\R^q))}\norm{S}{(0, k_2)}^{ \{2\} }\\
    &+2 \norm{S^{n-2}}{(0, k_2)}^{ \{2\} } \norm{S}{\mB(L^2(\R^q))}\norm{S}{(k_1, 0)}^{ \{1\} }\\
    &+ \norm{S^{n-2}}{(k_1, k_2)} \norm{S}{\mB(L^2(\R^q))}^2.
\end{align*}
Notice that at every step, we localize the operator $S$ in fewer factor spaces which enables us to factor out additional powers of $\norm{S}{\mB(L^2(\R^q))}$ as needed. After $m$ steps,
\begin{align*}
    \norm{S^n}{(k_1, k_2)} \lesssim & m \norm{S}{\mB(L^2(\R^q))}^{n-1}  \norm{S}{(k_1, k_2)} \\
    &+ 2 \cdot (1+2+\cdots + (m-1)+m) \norm{S}{\mB(L^2(\R^q))}^{n-2} \norm{S}{(k_1, 0)}^{ \{1\} } \norm{S}{(0, k_2)}^{ \{2\} }\\
    &+ m \norm{S^{n-m}}{(k_1, 0)}^{ \{1\} } \norm{S}{\mB(L^2(\R^q))}^{m-1}\norm{S}{(0, k_2)}^{ \{2\} }\\
    &+m \norm{S^{n-m}}{(0, k_2)}^{ \{2\} } \norm{S}{\mB(L^2(\R^q))}^{m-1}\norm{S}{(k_1, 0)}^{ \{1\} }\\
    &+ \norm{S^{n-m}}{(k_1, k_2)} \norm{S}{\mB(L^2(\R^q))}^m .
\end{align*}
Applying Theorem \ref{thm tame estimate pk} and the single-parameter tame estimate in Lemma \ref{lemma pk 2nu k1} $n-1$ times, we obtain
\begin{align*}
    \norm{S^n}{(k_1, k_2)} \leq &  C n \norm{S}{\mB(L^2(\R^q))}^{n-1}   \norm{S}{(k_1, k_2)} + C n^2 \norm{S}{\mB(L^2(\R^q))}^{n-2} \norm{S}{(k_1, 0)} \norm{S}{(0, k_2)}.
\end{align*}
Taking $n$th roots and limits,
\begin{align*}
    \lim_{n \rightarrow \infty} \norm{S^n}{(k_1, k_2)}^{1/n} \leq &  \norm{S}{\mB(L^2(\R^q))} < 1.
\end{align*}
We can thus conclude that the Neumann series \eqref{eq pk neumann series} converges in the sense of $\mP^{\vk}$ norms for all $\vk$. In other words, $(\ep \Op(K)^* \Op(K))^{-1} \in \Op(\mP^{(k_1, k_2)})$. Since operators in $\Op(\mP^{(k_1, k_2)})$ form an algebra, we can conclude that $\Op(K)^{-1} = \ep (\ep \Op(K)^* \Op(K))^{-1} \Op(K)^* \in \Op(\mP^{\vk})$. Simply put, $\Op(K)^{-1} = \Op(L)$, for some $L \in \mP^{\vk}$. 
\end{proof}

\vspace{.5in}

\printbibliography[title={References}]


\vspace{.2in}

\Addresses

\end{document}